\documentclass[12pt]{amsart}
\usepackage{amssymb}
\usepackage{mathrsfs}
\usepackage{graphicx}
\usepackage{color}
\usepackage[all]{xy}
\renewcommand\a{\alpha}

\newcommand\g{\gamma}
\renewcommand\d{\delta}
\newcommand\la{\lambda}

\newcommand\s{\sigma}

\newcommand\vf{\varphi}

\renewcommand\t{\tau}

\newcommand\D{\Delta}

\newcommand\vL{\varLambda}

\newcommand\vG{\varGamma}
\newcommand\ve{\varepsilon}

\newcommand{\QQ}{\mathbb Q}

\newcommand{\ZZ}{\mathbb Z}

\newcommand{\KK}{\mathbb K}

\renewcommand{\AA}{\mathbb A}

\newcommand\BQ{\mathbf Q}

\newcommand\Bm{\mathbf m}

\newcommand\bx{\mathbf{x}}

\newcommand\Bvf{\boldsymbol{\varphi}}

\newcommand\ZC{\mathcal{C}}

\newcommand\CI{\mathcal{I}}
\newcommand\CJ{\mathcal{J}}
\newcommand\CE{\mathcal{E}}

\newcommand\CO{\mathcal{O}}
\newcommand\CK{\mathcal{K}}

\newcommand\CU{\mathcal{U}}

\newcommand\CT{ \mathcal{T}}

\newcommand\CX{ \mathcal{X}}
\newcommand\CY{ \mathcal{Y}}

\newcommand\FS{\mathfrak S}

\newcommand\Fg{\mathfrak g}

\newcommand\Fm{\mathfrak m}
\newcommand\Fn{\mathfrak n}

\newcommand\Fgl{\mathfrak{gl}}

\newcommand\wt{\widetilde}

\newcommand{\lan}{\langle}
\newcommand{\ran}{\rangle}

\newcommand{\ra}{\rightarrow }

\newcommand\End{\operatorname{End}}

\newcommand\res{\operatorname{res}}

\newcommand\ch{\operatorname{ch}}

\newcommand{\rad}{\operatorname{rad}}

\newcommand{\cmod}{\operatorname{-mod}}

\newcommand{\LR}{\operatorname{LR}}

\newcommand{\gr}{\operatorname{\bf gr}} 
\newcommand{\isom}{\,\raise2pt\hbox{$\underrightarrow{\sim}$}\,}
\newcounter{ichi}
\newcommand{\roi}{\roman{ichi}}
\newcounter{ni}
\newcommand{\roii}{\roman{ni}}
\newcounter{san}
\newcommand{\roiii}{\roman{san}}
\newcounter{yon}
\newcommand{\roiv}{\roman{yon}}
\newcounter{go}
\newcounter{roku}
\newcounter{nana}
\newcounter{hachi}
\newcounter{kyu}
\newcommand{\Sc}{\mathscr{S}}

\newcommand{\He}{\mathscr{H}}

\newtheorem{thm}{Theorem}[section]
\newtheorem{lem}[thm]{Lemma}
\newtheorem{cor}[thm]{Corollary}
\newtheorem{prop}[thm]{Proposition}
\newtheorem{definition}[thm]{Definition}

\def \para{\refstepcounter{thm} \par\medskip\noindent
                \textbf{\thethm .} }

\def \remark{\refstepcounter{thm} \par\medskip\noindent
                \textbf{Remark \thethm .} }

\def \remarks{\refstepcounter{thm} \par\medskip\noindent
                \textbf{Remarks \thethm .} }

\allowdisplaybreaks[4]
\numberwithin{equation}{thm}

\begin{document}
\setlength{\baselineskip}{4.9mm}
\setlength{\abovedisplayskip}{4.5mm}
\setlength{\belowdisplayskip}{4.5mm}
%%%
%%%
\renewcommand{\theenumi}{\roman{enumi}}
\renewcommand{\labelenumi}{(\theenumi)}
\renewcommand{\thefootnote}{\fnsymbol{footnote}}
%%%
\renewcommand{\thefootnote}{\fnsymbol{footnote}}
%\NoBlackBoxes
\parindent=20pt
%%%%%%%%%%%%%%%%%%%%
\newcommand{\dis}{\displaystyle}
%%%%%%%%%%%%%%%%%%%%%%%%%%%%%%%%%%%

\medskip
\begin{center}
{\large \bf New realization of cyclotomic $q$-Schur algebras I}  
\\
\vspace{1cm}
Kentaro Wada 
\\[1em]

\address{Department of Mathematics, Faculty of Science, Shinshu University, 
	Asahi 3-1-1, Matsumoto 390-8621, Japan}
\email{wada@math.shinshu-u.ac.jp} 

\end{center}

\title{}
\maketitle 
\markboth{Kentaro Wada}
{New realization  of cyclotomic $q$-Schur algebras I}

%%%%%%%%%%%%%%%%%%%%%%%%%%%%%%%%%%%%%%%%%%%%%%%%%%%%%%%%%%%%%%%

%%%%%%%%%%%%%%%%%%%%%%%%%%%%%%%%%%%%%%%%%%%%%%%%%%%%%%%%%%%%%%%

%%%%%%%%%%%%%%%%%%%%%%%%%%%%%%%%%%%%%%%%%%%%%%%%%%%%%%%%%%%%%%%

\begin{abstract}
We introduce a Lie algebra $\mathfrak{g}_{\mathbf{Q}}(\mathbf{m})$ 
and an associative algebra $\mathcal{U}_{q,\mathbf{Q}}(\mathbf{m})$ 
associated with the Cartan data of $\mathfrak{gl}_m$ 
which is separated into $r$ parts 
with respect to $\mathbf{m}=(m_1, \dots, m_r)$ such that $m_1+ \dots + m_r =m$. 
We show that the Lie algebra $\mathfrak{g}_{\mathbf{Q}} (\mathbf{m})$ is a filtered deformation 
of the current Lie algebra of $\mathfrak{gl}_m$,  
and we can regard the algebra $\mathcal{U}_{q, \mathbf{Q}}(\mathbf{m})$ 
as a \lq\lq $q$-analogue" of $U(\mathfrak{g}_{\mathbf{Q}}(\mathbf{m}))$. 
Then, we realize a cyclotomic $q$-Schur algebra as a quotient algebra of 
$\mathcal{U}_{q, \mathbf{Q}}(\mathbf{m})$ under a certain mild condition. 
We also study the representation theory for $\mathfrak{g}_{\mathbf{Q}}(\mathbf{m})$ and $\mathcal{U}_{q,\mathbf{Q}}(\mathbf{m})$, 
and we apply them to the representations of the cyclotomic $q$-Schur algebras. 
\end{abstract}

%%%%%%%%%%%%%%%%%%%%%%%%%%%%%%%%%%%%%%%%%%%%%%%%%%%%%%%%%%%%%%%
\tableofcontents 
%%%%%%%%%%%%%%%%%%%%%%%%%%%%%%%%%%%%%%%%%%%%%%%%%%%%%%%%%%%%%%%

\setcounter{section}{-1}
\section{Introduction} 

\para
Let $\He_{n,r}$ be the Ariki-Koike algebra associated with the complex reflection group 
of type $G(r,1,n)$ over a commutative ring $R$ with parameters $q, Q_0, \dots, Q_{r-1} \in R$, 
where $q$ is invertible in $R$. 
Let $\Sc_{n,r}(\Bm)$ be the cyclotomic $q$-Schur algebra associated with $\He_{n,r}$ introduced in \cite{DJM98}, 
where $\Bm=(m_1, \dots, m_r)$ is an $r$-tuple of positive integers. 
By the result in \cite{DJM98}, 
it is known that $\Sc_{n,r}(\Bm)\cmod$ is a highest weight cover of $\He_{n,r} \cmod$ 
in the sense of \cite{R} if $R$ is a field and $\Bm$ is enough large. 

In \cite{RSVV} and \cite{Lo} independently, 
it is proven that 
$\Sc_{n,r}(\Bm) \cmod$ is equivalent to a certain highest weight subcategory 
of an affine parabolic category $\mathbf{O}$ in a dominant case of an affine general linear Lie algebra 
as a highest weight cover of $\He_{n,r} \cmod$. 
It is also equivalent to the category $\CO$ of rational Cherednik algebra with the corresponding parameters. 
In the argument of \cite{RSVV}, 
the monoidal structure on the affine parabolic category $\mathbf{O}$ 
(more precisely, the structure of $\mathbf{O}$ as a bimodule category over the Kazhdan-Lusztig  category) 
has an important role. 

In the case where $r=1$, 
it is known that the $q$-Schur algbera $\Sc_{n,1}(m)$ is a quotient algebra of the quantum group 
$U_q(\Fgl_m)$ associated with the general linear lie algebra $\Fgl_m$, 
and $\bigoplus_{n \geq 0} \Sc_{n,1}(m) \cmod$ is equivalent to the category $\ZC_{U_q(\Fgl_m)}^{\geq 0}$ 
consisting of finite dimensional polynomial representations of $U_q(\Fgl_m)$ 
(\cite{BLM}, \cite{Du} and \cite{J}). 
The category $\ZC_{U_q(\Fgl_m)}^{\geq 0}$ has a (braided) monoidal structure 
which comes from the structure of $U_q(\Fgl_m)$ as a Hopf algebra. 
Then the monoidal structure on $\ZC_{U_q(\Fgl_m)}^{\geq 0}$ 
is compatible with the monoidal structure on the Kazhdan-Lusztig category 
by \cite{KL}.  
However, it is not known such structures for cyclotomic $q$-Schur algebras in the case where $r >1$ 
although we may expect such structures through the equivalence in \cite{RSVV}.  
This is a motivation of this paper. 

In \cite{W}, we obtained a presentation of cyclotomic $q$-Schur algebras by generators and defining relations. 
The argument in \cite{W} are based on the existence of 
the upper (resp.  lower) Borel subalgebra of the cyclotomic $q$-Schur algebra $\Sc_{n,r}(\Bm)$ 
which is introduced in \cite{DR}. 
In \cite{DR}, it is proven that 
the upper (resp. lower) Borel subalgebra of $\Sc_{n,r}(\Bm)$ 
is isomorphic to the upper (resp. lower) Borel subalgebra of  $\Sc_{n,1}(m)$ 
(i.e. the case where $r=1$) which is a quotient of the upper (resp. lower) Borel subalgebra 
of the quantum group $U_q(\Fgl_m)$ ($m := \sum_{k=1}^r m_k$) 
if $\Bm$ is enough large. 
The presentation  of $\Sc_{n,r}(\Bm)$ in \cite{W} 
is applied to the representation theory of cyclotomic $q$-Schur algebras 
in \cite{W-2} and \cite{W-3}. 
However, this presentation is not so useful in general 
since, in the presentation,  we need some non-commutative polynomials which are computable, 
but we can not describe them  explicitly (see \cite[Lemma 7.2]{W}). 
Hence, we hope more useful realization of cyclotomic $q$-Schur algebras 
like as the fact that the $q$-Schur algebra $\Sc_{n,1}(m)$ is a quotient of the quantum group $U_q(\Fgl_m)$ 
in the case where $r=1$. 
In this paper, by extending the argument in \cite{W}, 
we give a possibility of such realization of cyclotomic $q$-Schur algebras. 

%%%%%%%%

\para
Let $\BQ=(Q_1,\dots,Q_{r-1})$ be an $r-1$ tuple of indeterminate elements over $\ZZ$, 
and $\QQ(\BQ)$ be a field of rational functions with variables $\BQ$. 
In \S2, we introduce a Lie algebra $\Fg_{\BQ}(\Bm)$ with parameters $\BQ$  
associated with the Cartan data of $\Fgl_m$ ($m=\sum_{k=1}^r m_k$) 
which is separated into $r$ parts with respect to $\Bm$ (see the paragraph \ref{Cartan data}). 
Then, in Proposition \ref{Prop glm[x] iso gr g}, 
we prove that $\Fg_{\BQ}(\Bm)$ is a filtered deformation of the current Lie algebra 
$\Fgl_{m}[x] = \QQ(\BQ) [x] \otimes \Fgl_m$ 
of the general linear Lie algebra $\Fgl_m$. 

In Corollary \ref{Cor tri decom of g}, 
we see that $\Fg_{\BQ}(\Bm)$ has a triangular decomposition 
\begin{align*} 
\Fg_{\BQ}(\Bm) = \Fn^- \oplus \Fn^0 \oplus \Fn^+. 
\end{align*} 
Then we can develop the weight theory to study representations of $\Fg_{\BQ}(\Bm)$ in the usual manner (see \S \ref{S Rep FgQ}). 
Let $\ZC_{\BQ}(\Bm)$ be the category of finite dimensional $\Fg_{\BQ}(\Bm)$-modules 
which have the weight space decompositions, and all eigenvalues of the action of $\Fn^0$ belong to $\QQ(\BQ)$.  
Then we see that a simple $\Fg_{\BQ}(\Bm)$-module in $\ZC_{\BQ}(\Bm)$ is a highest weight module. 

There exists a surjective homomorphism of Lie algebras 
$\Fg_{\BQ}(\Bm) \ra \Fgl_m$ (see \eqref{surjection gQm to glm}) 
which can be regarded as a special case of  evaluation homomorphisms (see Remark \ref{Remark evaluation gQ}). 
Let $\ZC_{\Fgl_m}$ be the category of finite dimensional $\Fgl_m$-modules which have the weight space decompositions. 
Then $\ZC_{\Fgl_m}$ is a full subcategory of $\ZC_{\BQ}(\Bm)$ through the above surjection 
(see Proposition \ref{Prop glm fullsub ZCQ(m)}).  

Let $\wt{\BQ}=(Q_0,Q_1,\dots,Q_{r-1})$ be an $r$ tuple of indeterminate elements over $\ZZ$, 
and $\QQ(\wt{\BQ})$ be a field of rational functions with variables $\wt{\BQ}$. 
Put $\Fg_{\wt{\BQ}}(\Bm) = \QQ(\wt{\BQ}) \otimes_{\QQ(\BQ)} \Fg_{\BQ}(\Bm)$, 
and define the category $\ZC_{\wt{\BQ}}(\Bm)$ in a similar way. 
Let $\Sc_{n,r}^{\mathbf{1}}(\Bm)$ be the cyclotomic $q$-Schur algebra over $\QQ(\wt{\BQ})$ 
with parameters $q=1$ and $\wt{\BQ}$. 
In Theorem \ref{Theorem Sc1 in ZC},   
we prove that there exists a homomorphism of algebras 
\begin{align*} 
\Psi_{\mathbf{1}}: U(\Fg_{\wt{\BQ}} (\Bm)) \ra \Sc_{n,r}^{\mathbf{1}}(\Bm), 
\end{align*}  
where $U(\Fg_{\wt{\BQ}}(\Bm))$ is the universal enveloping algebra of $\Fg_{\wt{\BQ}}(\Bm)$. 
Assume that $m_k \geq n$ for all $k=1,2,\dots, r-1$, then $\Psi_{\mathbf{1}}$ is surjective. 
Then $\Sc_{n,r}^{\mathbf{1}}(\Bm) \cmod$ is a full subcategory of $\ZC_{\wt{\BQ}}(\Bm)$ 
through the surjection $\Psi_{\mathbf{1}}$ (see Theorem \ref{Theorem Sc1 in ZC} (\roii)).   
We expect that the surjectivity of $\Psi_{\mathbf{1}}$ also holds without the condition for $\Bm$. 
(We need the condition for $\Bm$ by a technical reason (see Remark \ref{Remark subjectivity of Psi}).)

It is known that $\Sc_{n,r}^{\mathbf{1}}(\Bm)$ is semi-simple,  
and the set of Weyl (cell) modules $\{ \D(\la) \,|\, \la \in \wt{\vL}^+_{n,r}(\Bm)\}$ 
gives a complete set of isomorphism classes of simple $\Sc_{n,r}^{\mathbf{1}}(\Bm)$-modules 
(see \S \ref{Review CSA} and \cite{DJM98}  for definitions).  
The characters of the Weyl modules, denoted by $\ch \D(\la)$ ($\la \in \wt{\vL}_{n,r}^+(\Bm)$), 
are studied in \cite{W-2}. 
We see that $\ch \D(\la)$ ($\la \in \wt{\vL}_{n,r}^+(\Bm)$) is a symmetric polynomial with variables $\bx_{\Bm}$ with respect to $\Bm$. 
Put $\wt{\vL}_{\geq 0}^+(\Bm) = \cup_{n \geq 0} \wt{\vL}_{n,r}^+(\Bm)$. 
Then, for $\la, \mu \in \wt{\vL}_{\geq 0}^+ (\Bm)$, 
it was conjectured that 
\begin{align}
\label{ch Dla ch Dmu}
\ch \D(\la) \ch \D(\mu) = \sum_{\nu \in \wt{\vL}_{\geq 0,r}^+ (\Bm)} \LR^\nu_{\la \mu} \ch \D(\nu) 
\end{align}
in \cite{W-2}, 
where $\LR^\nu_{\la\mu}$ is the product of Littlewood-Richardson coefficients with respect to $\la,\mu$ and $\nu$  (see \S \ref{Ch of Weyl} for details). 
We prove this conjecture in Proposition \ref{Prop ch}. 
We remark that the characters of  Weyl modules of a cyclotomic $q$-Schur algebra do not depend on the choice of a base field and parameters. 

By using the usual coproduct of the universal enveloping algebra $U(\Fg_{\wt{\BQ}}(\Bm))$ of $\Fg_{\wt{\BQ}}(\Bm)$, 
we can consider the tensor product $M \otimes N$ in  $U(\Fg_{\wt{\BQ}}(\Bm)) \cmod$ for $M,N \in U(\Fg_{\wt{\BQ}}(\Bm))\cmod$. 
We regard $\Sc_{n,r}^{\mathbf{1}}(\Bm)$-modules ($n \geq 0$) as a $U(\Fg_{\wt{\BQ}}(\Bm))$-modules through the homomorphism 
$\Psi_{\mathbf{1}}$. 
Take $n, n_1, n_2 \in \ZZ_{>0}$ such that $n=n_1 + n_2$. 
Then, in Proposition \ref{Prop decom tensor Weyl}, 
we prove that, 
for $\la \in \wt{\vL}_{n_1,r}^+(\Bm)$ and $\mu \in \wt{\vL}_{n_2,r}^+(\Bm)$, 
\begin{align} 
\label{decom tensor Weyl in intro}
\D(\la) \otimes \D(\mu) \cong \bigoplus_{ \nu \in \wt{\vL}_{n,r}^+ (\Bm)} \LR_{\la \mu}^{\nu} \D(\nu) 
\end{align}
as $U(\Fg_{\wt{\BQ}} (\Bm))$-modules if $m_k \geq n$ for all $k=1,2,\dots, r-1$, 
where $\LR_{\la \mu}^\nu \D(\nu)$ means the direct sum of $\LR_{\la \mu}^\nu$ copies of $\D(\nu)$. 
In particular, we see that $\D(\la) \otimes \D(\nu) \in \Sc_{n,r}(\Bm) \cmod$. 
The decomposition \eqref{decom tensor Weyl in intro} gives an interpretation of the formula \eqref{ch Dla ch Dmu} in the category $\ZC_{\wt{\BQ}}(\Bm)$. 
We expect that \eqref{decom tensor Weyl in intro} also holds without the condition for $\Bm$. 
(Note that we prove the formula \eqref{ch Dla ch Dmu}  without the condition for $\Bm$ in  Proposition \ref{Prop ch}.) 

%%%%%%%%%%

\para 
Put $\AA=\ZZ[q,q^{-1}, Q_1, \dots,Q_{r-1}]$, where $q, Q_1,\dots,Q_{r-1}$ are indeterminate elements over $\ZZ$, 
and let $\KK= \QQ(q,Q_1,\dots,Q_{r-1})$ be the quotient field of $\AA$. 
In \S \ref{CUqQ}, 
we introduce an associative algebra $\CU_{q,\BQ}(\Bm)$ with parameters $q$ and $\BQ$ 
associated with the Cartan data of $\Fgl_{m}$ which is separated into $r$ parts with respect to $\Bm$. 

Let $\CU_{\AA, q, \BQ}^{\star}(\Bm)$ be the $\AA$-subalgebra of $\CU_{q,\BQ}(\Bm)$ 
generated by defining generators of $\CU_{q,\BQ}(\Bm)$ (see the paragraph \ref{CUAqQstar}). 
We regard $\QQ(\BQ)$ as an $\AA$-module through the ring homomorphism 
$\AA \ra \QQ(\BQ)$ by sending $q$ to $1$, 
and we consider the specialization $\QQ(\BQ) \otimes_{\AA} \CU_{\QQ,q,\BQ}^\star(\Bm)$ 
using this ring homomorphism. 
Then we have a surjective homomorphism of algebras 
\begin{align} 
\label{Hom gQ(m) UqQ(m) in intro}
U(\Fg_{\BQ}(\Bm)) \ra \QQ(\BQ) \otimes_{\AA} \CU_{\AA,q,\BQ}^\star (\Bm) / \mathfrak{J}, 
\end{align}
where $\mathfrak{J}$ is a certain ideal of $\QQ(\BQ) \otimes_{\AA} \CU_{\AA,q,\BQ}^\star (\Bm) $ 
(see \eqref{Hom gQ(m) UqQ(m)}). 
We conjecture that the surjection \eqref{Hom gQ(m) UqQ(m) in intro} is isomorphic. 
Then we can regard $\CU_{q,\BQ}(\Bm)$ as a \lq\lq $q$-analogue" of $U(\Fg_{\BQ}(\Bm))$. 
Dividing by the ideal $\mathfrak{J}$ in \eqref{Hom gQ(m) UqQ(m) in intro} 
means that the Cartan subalgebra $U(\Fn^0)$ of $U(\Fg_{\BQ}(\Bm))$ deforms to several directions in 
$\CU_{q,\BQ}(\Bm)$ (see the paragraph \ref{CUAqQstar} and Remark \ref{Remark Hom gQ(m) UqQ(m) iso}). 

We see that $\CU_{q,\BQ}(\Bm)$ has a triangular decomposition 
\begin{align}
\label{tri decom intro} 
\CU_{q,\BQ}(\Bm) = \CU^- \CU^0 \CU^+ 
\end{align}
in a weak sense (see \eqref{tri decor U}). 
We conjecture that the multiplication map 
$\CU^- \otimes_{\KK} \CU^0 \otimes_{\KK} \CU^+ \ra \CU_{q,\BQ}(\Bm)$ 
gives an isomorphism as vector spaces.  
More precisely, we expect the existence of a PBW type basis of $\CU_{q, \BQ}(\Bm)$ 
which is compatible with a PBW basis of $U(\Fg_{\BQ}(\Bm))$ through the homomorphism 
\eqref{Hom gQ(m) UqQ(m) in intro}. 

Anyway, 
thanks to the triangular decomposition \eqref{tri decom intro}, 
we can develop the weight theory to study $\CU_{q,\BQ}(\Bm)$-modules in the usual manner 
(see \S \ref{Section Rep CUqQ}). 
Let $\ZC_{q,\BQ}(\Bm)$ be the category of finite dimensional $\CU_{q,\BQ}(\Bm)$-modules 
which have the weight space decompositions, and all eigenvalues of the action of $\CU^0$ 
belong to $\KK$. 
Then we see that a simple $\CU_{q,\BQ}(\Bm)$-module in $\ZC_{q,\BQ}(\Bm)$ is a highest weight module. 

There exists a surjective homomorphism of algebras 
$\CU_{q,\BQ}(\Bm) \ra U_q (\Fgl_m)$ (see \eqref{surjection UqQm to Uqglm}) 
which can be regarded as a special case of evaluation homomorphisms 
(see Remark \ref{Remark evaluation}). 
Let $\ZC_{U_q(\Fgl_m)}$ be the category of finite dimensional $U_q(\Fgl_m)$-modules 
which have the weight space decompositions. 
Then $\ZC_{U_q(\Fgl_m)}$ is a full subcategory of $\ZC_{q,\BQ}(\Bm)$ 
through the above surjection (see Proposition \ref{Prop ZCUqgl subset ZCqQ}). 

Put $\wt{\KK} = \KK(Q_0)$ and $\wt{\AA}=\AA[Q_0]$. 
We also put $\CU_{q,\wt{\BQ}}(\Bm) = \wt{\KK} \otimes_{\KK} \CU_{q,\BQ}(\Bm)$. 
Let $\CU_{\AA,q,\BQ}(\Bm)$ be the $\AA$-form of $\CU_{q,\BQ}(\Bm)$ taking divided powers 
(see the paragraph  \ref{A-form}), 
and put 
$\CU_{\wt{\AA},q, \wt{\BQ}}(\Bm) = \wt{\AA} \otimes_{\AA} \CU_{\AA,q, \BQ}(\Bm)$. 
Let $\Sc_{n,r}^{\wt{\KK}}(\Bm)$ (resp. $\Sc_{n,r}^{\wt{\AA}}(\Bm)$) be the cyclotomic $q$-Schur algebra 
over $\wt{\KK}$ (resp. over $\wt{\AA}$) with parameters $q$ and $\wt{\BQ}$. 
In Theorem \ref{Thm UqQ to Sc}, we prove that 
there exists a homomorphism of algebras 
\begin{align*} 
\Psi : \CU_{q,\wt{\BQ}}(\Bm) \ra \Sc_{n,r}^{\wt{\KK}}(\Bm). 
\end{align*}
By the restriction of $\Psi$ to $\CU_{\wt{\AA},q, \wt{\BQ}}(\Bm)$, 
we have the homomorphism 
$\Psi_{\wt{\AA}} : \CU_{\wt{\AA},q, \wt{\BQ}}(\Bm) \ra \Sc_{n,r}^{\wt{\AA}}(\Bm)$. 
Then we can specialize $\Psi_{\wt{\AA}}$ to any base ring and parameters. 
If $m_k \geq n$ for all $k=1,2, \dots, r-1$, 
then $\Psi$ (resp. $\Psi_{\wt{\AA}}$) is surjective 
(see also Remark \ref{Remark subjectivity of Psi} for surjectivity of $\Psi$). 
In Theorem \ref{Thm Sc Rep}, 
we prove that $\Sc_{n,r}^{\wt{\KK}}(\Bm) \cmod$ is a full subcategory of $\ZC_{q, \wt{\BQ}}(\Bm)$ 
through the surjection $\Psi$ if $\Bm$ is enough large.  

We conjecture that $\CU_{q, \wt{\BQ}}(\Bm)$ has a structure as a Hopf algebra, 
and that the decomposition \eqref{decom tensor Weyl in intro} also holds for Weyl modules of 
$\Sc_{n,r}^{\wt{\KK}}(\Bm)$ ($n \geq 0$) through the homomorphism $\Psi$ 
and the Hopf algebra structure of $\CU_{q, \wt{\BQ}}(\Bm)$.  
(Note that the formula \eqref{ch Dla ch Dmu} holds for $\Sc_{n,r}^{\wt{\KK}}(\Bm)$ ($n \geq 0$).)  

It is also interesting problem  to obtain a monoidal structure for $\CU_{q,\BQ}(\Bm)$ 
(resp. $\CU_{\AA,q,\BQ}(\Bm)$ and its specialization) 
which should be related to the monoidal structure on the affine parabolic category $\mathbf{O}$. 

\textbf{Acknowledgements.} 
This research was supported by JSPS KAKENHI Grant Number 24740007.
The author is grateful to Tatsuyuki Hikita  for his suggestion on the definition of the polynomials 
$\Phi_t^{\pm}(x_1,\dots, x_k)$ (see Remark \ref{Remark def Phi}). 

%%%%%%%%%%%%%%%%%%%%%%%%%%%%%%%%%%%%%%%%%%%%%%%%%%%%%%%%%%%%%%%

%%%%%%%%%%%%%%%%%%%%%%%%%%%%%%%%%%%%%%%%%%%%%%%%%%%%%%%%%%%%%%%

\section{Notation} 
\para 
For a condition $X$, put 
$
\d_{(X)} = \begin{cases} 1 & \text{ if $X$ is true}, \\ 0 & \text{ if $X$ is false.} \end{cases} 
$
We also put $\d_{i,j} = \d_{(i=j)}$ for simplicity. 

\para 
\textbf{$q$-integers.} 
Let $\QQ(q)$ be the field of rational functions over $\QQ$ with an indeterminate  variable $q$. 
For $d \in \ZZ$, 
put $[d] = (q^d- q^{-d})/ (q-q^{-1})\in \QQ(q)$. 
For $d \in \ZZ_{>0}$, 
put $[d]! =[d][d-1] \dots [1]$, 
and we put $[0]!=1$. 
For $d \in \ZZ$ and $c \in \ZZ_{>0}$, 
put 
\begin{align*} 
\left[ \begin{matrix} d \\ c \end{matrix} \right] = \frac{[d][d-1] \dots [d-c+1]}{[c][c-1] \dots [1]}, 
\text{ and put } 
\left[ \begin{matrix} d \\ 0 \end{matrix} \right] =1.
\end{align*} 
It is well-known that all $[d], [d]!$ and $\left[ \begin{smallmatrix} d \\ c \end{smallmatrix} \right]$ 
belong to $\ZZ[q,q^{-1}]$. 
Thus we can specialize these elements to any ring $R$ and $q \in R$ such that $q$ is invertible in $R$, 
and we denote them by same symbols. 
 
\para 
\label{Cartan data}
\textbf{Cartan data.} 
Let $\Bm=(m_1,\dots, m_r) $ be an $r$-tuple of positive integers. 
Put $m = \sum_{k=1}^r m_k$. 
Let $P= \bigoplus_{i=1}^m \ZZ \ve_i$ be the weight lattice of $\Fgl_m$, 
and let $P^\vee = \bigoplus_{i=1}^m \ZZ h_i $ be its dual with the natural pairing 
$\lan \, , \, \ran : P \times P^\vee \ra \ZZ$ such that 
$\lan \ve_i, h_j \ran = \d_{ij}$. 
put 
$P_{\geq 0} = \bigoplus_{i=1}^m \ZZ_{\geq 0} \ve_i$. 

Set $\a_i = \ve_i - \ve_{i+1}$ for $i=1,\dots, m-1$, 
then $\Pi = \{ \a_i \,|\, 1 \leq i \leq m-1\}$ 
is the set of simple roots, 
and  
$Q=\bigoplus_{i=1}^{m-1} \ZZ \a_i$ is the root lattice of $\Fgl_m$. 
Put $Q^+ = \bigoplus_{i=1}^{m-1} \ZZ_{\geq 0} \a_i$. 

Set $\a_i^{\vee} =h_i - h_{i+1}$ for $i=1,\dots, m-1$, 
then $\Pi^{\vee} =\{ \a_i^{\vee} \,|\, 1 \leq i \leq m-1 \}$ is the set of simple coroots.  

We define a partial order $\geq $ on $P$, 
so called dominance order, 
by $\la \geq \mu$ if $\la - \mu \in Q^+$. 

Put $\vG(\Bm)= \{ (i,k) \,|\, 1 \leq i \leq m_k, 1 \leq k \leq r \}$,
and 
$\vG'(\Bm)= \vG(\Bm) \setminus \{(m_r,r)\}$. 
We identify the set $\vG(\Bm)$ with the set $\{1,2,\dots,m \}$ by the bijection 
\begin{align}
\label{identify vG 1 to m}
\g : \vG(\Bm) \ra \{1,2, \dots, m\} \text{ such that } (i,k) \mapsto \sum_{j=1}^{k-1} m_j +i.
\end{align}
Then, we can identify the set $\vG'(\Bm)$ with the set $\{1,2,,\dots, m-1\}$. 
Under the identification \eqref{identify vG 1 to m}, 
for $(i,k),(j,l) \in \vG(\Bm)$, 
we define 
\begin{align*} 
(i,k)> (j,l) \text{  if } \g((i,k)) > \g((j,l)), 
\text{ and }
(i,k)  \pm (j,l) = \g((i,k)) \pm   \g((j,l)). 
\end{align*}
We also have 
$(m_k+1, k)=(1,k+1)$ for $k=1,\dots,r-1$ 
(resp. $(1-1,k) =(m_{k-1},k-1)$ for $k=2,\dots,r$).  

We may write 
\begin{align*}
P= \bigoplus_{(i,k) \in \vG(\Bm)} \ZZ \ve_{(i,k)}, 
\quad 
P^{\vee} = \bigoplus_{(i,k) \in \vG(\Bm)} \ZZ h_{(i,k)}, 
\quad 
Q= \bigoplus_{(i,k) \in \vG'(\Bm)} \ZZ \a_{(i,k)}.
\end{align*}

For $(i,k) \in \vG'(\Bm)$, $(j,l) \in \vG(\Bm)$, 
put $a_{(i,k)(j,l)} = \lan \a_{(i,k)}, h_{(j,l)} \ran$. 
Then, we have 
\begin{align*}
a_{(i,k)(j,l)} = 
	\begin{cases} 
		1 & \text{ if } (j,l) = (i,k), 		
		\\
		-1 & \text{ if } (j,l) = (i+1,k),
		\\
		0 & \text{ otherwise.} 
	\end{cases}
\end{align*}

Put 
\begin{align*} 
&P^+ =\{ \la \in P \,|\, \lan \la, \a_{(i,k)}^{\vee} \ran \in \ZZ_{\geq 0} \text{ for all } (i,k) \in \vG'(\Bm) \} 
\text{ and }
\\
&P^+_{\Bm} =\{ \la \in P \,|\, \lan \la, \a_{(i,k)}^{\vee} \ran \in \ZZ_{\geq 0} \text{ for all } (i,k) \in \vG(\Bm) \setminus \{(m_k,k)\,|\, 1 \leq k \leq r\} \}. 
\end{align*}
Then $P^+$ is the set of dominant integral weights for $\Fgl_m$, 
and 
$P^+_{\Bm}$ is the set of dominant integral weights for 
Levi subalgebra $\Fgl_{m_1}\oplus \dots \oplus \Fgl_{m_r}$ of $\Fgl_m$ 
with respect to $\Bm=(m_1,\dots,m_r)$.

%%%%%%%%%%%%%%%%%%%%%%%%%%%%%%%%%%%%%%%%%%%%%%%%%%%%%%%%%%%%%%%

%%%%%%%%%%%%%%%%%%%%%%%%%%%%%%%%%%%%%%%%%%%%%%%%%%%%%%%%%%%%%%%

\section{Lie algebra $\Fg_{\BQ}(\Bm)$} 
\label{Section FgQ(Bm)}

In this section, we introduce a Lie algebra $\Fg_{\BQ}(\Bm)$ 
with $r-1$ parameters $\BQ=(Q_1, \dots, Q_{r-1})$ 
associated with the Cartan data in the paragraph \ref{Cartan data}. 
Then we study some basic structures of $\Fg_{\BQ}(\Bm)$. 
In particular, we prove that 
$\Fg_{\BQ}(\Bm)$ is a filtered deformation of the current Lie algebra 
$\Fgl_m[x]$ of the general linear Lie algebra $\Fgl_m$.  

\para 
Let $\BQ=(Q_1, \dots, Q_{r-1})$ be an $r-1$-tuple of indeterminate elements over $\ZZ$. 
Let $\ZZ[\BQ] =\ZZ[Q_1,\dots,Q_{r-1}]$ be the polynomial ring 
with variables $Q_1,\dots, Q_{r-1}$, 
and $\QQ(\BQ)=\QQ(Q_1,\dots, Q_{r-1})$ be the quotient field of $\ZZ[\BQ]$. 

\begin{definition} 
\label{Def gQ(m)}
We define the Lie algebra $\Fg=\Fg_{\BQ}(\Bm)$ over $\QQ(\BQ)$ by the following generators and defining relations: 
\begin{description}
\item[Generators] 
$\CX_{(i,k),t}^{\pm}$, $\CI_{(j,l),t}$ ($(i,k) \in \vG'(\Bm)$, $(j,l) \in \vG(\Bm)$, $t \geq0$). 

\item[Relations] 
\begin{align*}
&\tag{L1} 
	[\CI_{(i,k),s}, \CI_{(j,l),t}]=0, 
\\
& \tag{L2} 
	[\CI_{(j,l),s}, \CX_{(i,k),t}^{\pm}] = \pm a_{(i,k)(j,l)} \CX_{(i,k),s+t}^{\pm}, 
\\
& \tag{L3} 
	[\CX_{(i,k),t}^+, \CX_{(j,l),s}^-] 
	= \d_{(i,k),(j,l)} 
		\begin{cases}
			\CJ_{(i,k),s+t} & \text{ if } i \not=m_k, 
			\\
			- Q_k \CJ_{(m_k,k),s+t} + \CJ_{(m_k,k),s+t+1} & \text{ if } i=m_k, 
		\end{cases} 
\\
& \tag{L4} 
	[ \CX_{(i,k),t}^{\pm}, \CX_{(j,l),s}^{\pm}] =0 \quad  \text{ if } (j,l) \not= (i \pm 1,k), 
\\
& \tag{L5} 
	[\CX_{(i,k),t+1}^{+}, \CX_{(i \pm1 ,k),s}^{+}] = [\CX_{(i,k),t}^{+}, \CX_{(i \pm 1, k ),s+1}^{+}], 
	\\ &
	[\CX_{(i,k),t+1}^{-}, \CX_{(i \pm1 ,k),s}^{-}] = [\CX_{(i,k),t}^{-}, \CX_{(i \pm 1, k ),s+1}^{-}], 
\\
& \tag{L6} 
	[\CX_{(i,k),s}^+, [\CX_{(i,k),t}^+, \CX_{(i \pm 1,k),u}^+]] 
	= [\CX_{(i,k),s}^-, [\CX_{(i,k),t}^-, \CX_{(i \pm 1,k),u}^-]] 
	=0, 
\end{align*}
\end{description}
where we put $\CJ_{(i,k),t} = \CI_{(i,k),t} - \CI_{(i+1,k),t}$. 
\end{definition}

% ###############################################################################
\para 
\label{Def V tau}
For $\t \in \QQ(\BQ)$, 
let $V_{\t}=\bigoplus_{(j,l) \in \vG(\Bm)} \QQ(\BQ) v_{(j,l)}$ be the $\QQ(\BQ)$-vector space with a basis $\{ v_{(j,l)} \,|\, (j,l) \in \vG(\Bm)\}$. 
We can define the action of $\Fg$ on $V_\t$ by  
\begin{align*}
&\CX_{(i,k),t}^+ \cdot v_{(j,l)} 
	= \begin{cases} 
		\t^t v_{(i,k)} & \text{ if } (j,l)=(i+1,k) \text{ and } i \not= m_k, 
		\\
		(-Q_k + \t) \t^t v_{(m_k,k)} & \text{ if } (j,l) =(1,k+1) \text{ and } i=m_k, 
		\\
		0 & \text{ otherwise}, 
	\end{cases} 
\\
& \CX_{(i,k),t}^- \cdot v_{(j,l)} 
	= \begin{cases} 
		\t^t v_{(i+1,k)} & \text{ if } (j,l) = (i,k), 
		\\
		0 & \text{ otherwise}, 
	\end{cases}
\\
& \CI_{(i,k),t} \cdot v_{(j,l)}  
	= \begin{cases} 
		\t^t v_{(j,l)} & \text{ if } (j,l)=(i,k), 
		\\
		0 & \text{ otherwise}.
	\end{cases} 
\end{align*}
We can check the well-definedness of the above action by direct calculations. 

% ###############################################################################

\para 
For $(i,k),(j,l) \in \vG(\Bm)$ and $t \geq 0$, 
we define the element $\CE_{(i,k)(j,l)}^t \in \Fg$ by 
\begin{align*}
\CE_{(i,k),(j,l)}^t =
	\begin{cases}
		\CI_{(i,k),t} & \text{ if } (j,l)=(i,k), 
		\\
		[\CX_{(i,k),0}^+, [\CX_{(i+1,k),0}^+, \dots, [\CX_{(j-2,l),0}^+, \CX_{(j-1,l),t}^+] \dots] 
		& \text{ if } (j,l) > (i,k), 
		\\
		[\CX_{(i-1,k),0}^-, [ \CX_{(i-2,k),0}^-, \dots, [\CX_{(j+1,l),0}^-, \CX_{(j,l),t}^-] \dots ] 
		& \text{ if } (j,l) < (i,k), 
	\end{cases}
\end{align*} 
in particular, 
we have $\CE_{(i,k),(i+1,k)}^t = \CX_{(i,k),t}^+$ and $\CE_{(i+1,k),(i,k)}^t = \CX_{(i,k),t}^-$. 
\\
If $(j,l) > (i,k)$, we have 
\begin{align*}
\CE_{(i,k),(j,l)}^t 
	&=[\CX_{(i,k),0}^+, \CE_{(i+1,k),(j,l)}^t] 
	\\
	&= [\CE_{(i,k),(j-1,l)}^t, \CX_{(j-1,l),0}^+]. 
\end{align*}
If $(j,l) < (i,k)$, we have 
\begin{align*}
 \CE_{(i,k),(j,l)}^t 
	&=[\CX_{(i-1,k),0}^-, \CE_{(i-1,k),(j,l)}^t] 
	\\
	&=  [\CE_{(i,k),(j+1,l)}^t, \CX_{(j,l),0}^-]. 
\end{align*}

\begin{lem}\
\label{Lemma spanning set g}
\begin{enumerate}
\item 
For $(i,k),(j,l) \in \vG(\Bm)$  such that $(j,l) > (i,k)$,  we have 
\begin{align}
&\notag\\[-2em]
\label{CX+ CE+}
&[\CX_{(a,c),s}^+, \CE_{(i,k),(j,l)}^t] 
= \begin{cases} 
	\CE_{(i-1,k),(j,l)}^{t+s} & \text{ if } (a,c)=(i-1,k), 
	\\
	- \CE_{(i,k), (j+1,l)}^{t+s} & \text{ if } (a,c)=(j,l),  
	\\
	0 & \text{ otherwise}, 
	\end{cases}
\\
\label{CI CE+}
&[\CI_{(a,c),s}, \CE_{(i,k),(j,l)}^t] 
	= \begin{cases} 
		\CE_{(i,k),(j,l)}^{t+s} & \text{ if } (a,c)=(i,k), 
		\\
		- \CE_{(i,k),(j,l)}^{t+s} & \text{ if } (a,c)=(j,l),  
		\\ 0 & \text{ otherwise}, 
	\end{cases} 
\end{align}
\begin{align}
\label{CX- CE+}
&\notag \\[-1.5em] 
&[\CX_{(a,c),s}^-, \CE_{(i,k),(j,l)}^t] 
\\ \notag
&	= \begin{cases}
		- \CE_{(i,k),(i,k)}^{t+s} +  \CE_{(i+1,k), (i+1,k)}^{t+s} 
			& \text{ if } \ell =1, (a,c)=(i,k) \text{ and } i \not=m_k, 
		\\
		Q_k (\CE_{(m_k,k),(m_k,k)}^{t+s} - \CE_{(1,k+1),(1, k+1)}^{t+s}) 
		- \CE_{(m_k,k),(m_k,k)}^{t+s+1} + \CE_{(1,k+1),(1, k+1)}^{t+s+1} 
		\hspace{-15em} 
		\\
			& \text{ if } \ell =1, (a,c)=(i,k) \text{ and } i=m_k, 
		\\
		\CE_{(i+1,k),(j,l)}^{t+s} 
			& \text{ if } \ell >1, (a,c)=(i,k) \text{ and } i \not=m_k, 
		\\
		- Q_k \CE_{(1,k+1),(j,l)}^{t+s} + \CE_{(1,k+1),(j,l)}^{t+s+1} 
			& \text{ if } \ell >1, (a,c)=(i,k) \text{ and } i=m_k, 
		\\
		- \CE_{(i,k),(j-1,l)}^{t+s} 
			& \text{ if } \ell  >1, (a,c)=(j-1,l) \text{ and } j-1 \not=m_l, 
		\\
		Q_l \CE_{(i,k),(m_l,l)}^{t+s} - \CE_{(i,k),(m_l,l)}^{t+s+1} 
			& \text{ if } \ell >1, (a,c)=(j-1,l) \text{ and } j-1=m_l, 
		\\
		0 & \text{ otherwise}, 
	\end{cases}
\end{align} 
where we put $\ell = (j,l)-(i,k)$. 

\item 
For $(i,k), (j,l) \in \vG(\Bm)$ such that $(j,l) < (i,k)$, we have 
\begin{align*}
&[\CX_{(a,c),s}^-, \CE_{(i,k),(j,l)}^t] 
	= \begin{cases} 
		\CE_{(i+1,k),(j,l)}^{t+s} & \text{ if } (a,c)=(i,k), 
		\\
		- \CE_{(i,k),(j-1,l)}^{t+s} & \text{ if } (a,c) =(j-1,l), 
		\\
		0 & \text{ otherwise}, 
	\end{cases}
\\
&[\CI_{(a,c),s}, \CE_{(i,k),(j,l)}^t] 
	= \begin{cases} 
		\CE_{(i,k),(j,l)}^{t+s} & \text{ if } (a,c)=(i,k), 
		\\
		- \CE_{(i,k),(j,l)}^{t+s} & \text{ if } (a,c)=(j,l), 
		\\
		0 & \text{ otherwise}, 
	\end{cases} 
\end{align*}
\begin{align*}
&[\CX_{(a,c),s}^+, \CE_{(i,k),(j,l)}^t] 
\\
&= \begin{cases} 
	 \CE_{(i-1,k),(i-1,k)}^{t+s} - \CE_{(i,k),(i,k)}^{t+s} 
	 	& \text{ if } \ell =1, (a,c)=(i-1,k) \text{ and } i-1 \not=m_k, 
	 \\
	 - Q_k ( \CE_{(m_k,k),(m_k,k)}^{t+s} - \CE_{(1, k+1), (1,k+1)}^{t+s}) 
	 + \CE_{(m_k,k),(m_k,k)}^{t+s+1} - \CE_{(1, k+1), (1,k+1)}^{t+s+1}
	 	\hspace{-15em} 
	 	\\
	 	&\text{ if } \ell =1, (a,c)=(i-1,k) \text{ and } i-1 =m_k, 
	\\
	\CE_{(i-1,k), (j,l)}^{t+s} 
		& \text{ if } \ell >1, (a,c)=(i-1,k) \text{ and } i-1 \not=m_k, 
	\\
	- Q_k \CE_{(m_k,k),(j,l)}^{t+s} + \CE_{(m_k,k),(j,l)}^{t+s+1} 
		& \text{ if } \ell >1, (a,c) =(i-1,k) \text{ and } i-1 =m_k, 
	\\
	- \CE_{(i,k),(j+1,l)}^{t+s} 
		& \text{ if } \ell >1, (a,c)=(j,l) \text{ and } j \not= m_l, 
	\\
	Q_l \CE_{(i,k),(1,l+1)}^{t+s} - \CE_{(i,k),(1,l+1)}^{t+s+1} 
		& \text{ if } \ell >1, (a,c)=(j,l) \text{ and } j=m_l, 
	\\
	0 & \text{ otherwise}, 
	\end{cases}
\end{align*}
where we put $\ell =(i,k) -(j,l)$. 

\item 
For $(i,k) \in \vG(\Bm)$, we have 
\begin{align*}
& [ \CI_{(a,c),s}, \CE_{(i,k),(i,k)}^t] =0, 
\\
& [\CX_{(a,c),s}^{+}, \CE_{(i,k),(i,k)}^t] 
	= - a_{(a,c)(i,k)} \CE_{(a,c),(a+1,c)}^{t+s}, 
\\
& [\CX_{(a,c),s}^-, \CE_{(i,k),(i,k)}^t] 
	=  a_{(a,c)(i,k)} \CE_{(a+1,c),(a,c)}^{t+s}. 
\end{align*}
\end{enumerate} 
\end{lem}

% #####

\begin{proof}
We prove \eqref{CX+ CE+} by the induction on $(j,l)-(i,k)$. 

In the case where $(j,l)-(i,k)=1$, it is follows from the relations (L4) and (L5). 
Assume that $(j,l)-(i,k) >1$. 
We have 
\begin{align*}
[\CX_{(a,c),s}^+, \CE_{(i,k),(j,l)}^t] 
&= [\CX_{(a,c),s}^+, [\CX_{(i,k),0}^+, \CE_{(i+1,k),(j,l)}^t]] 
\\
&=  [\CX_{(i,k),0}^+, [\CX_{(a,c),s}^+, \CE_{(i+1,k),(j,l)}^t ]] 
	+ [[ \CX_{(a,c),s}^+, \CX_{(i,k),0}^+], \CE_{(i+1,k),(j,l)}^t]. 
\end{align*}
Applying the assumption of the induction, 
we have 
\begin{align}
\label{CX+ CE+ 1}
[\CX_{(a,c),s}^+, \CE_{(i,k),(j,l)}^t] 
&= 
\begin{cases}
	[\CX_{(i,k),0}^+, \CE_{(i,k),(j,l)}^{t+s}] & \text{ if } (a,c)=(i,k), 
	\\
	- [\CX_{(i,k),0}^+, \CE_{(i+1,k),(j+1,l)}^{t+s}] & \text{ if } (a,c)=(j,l), 
	\\
	[[\CX_{(i-1,k),s}^+, \CX_{(i,k),0}^+], \CE_{(i+1,k),(j,l)}^t] & \text{ if } (a,c)=(i-1,k), 
	\\
	[[\CX_{(i+1,k),s}^+, \CX_{(i,k),0}^+], \CE_{(i+1,k),(j,l)}^t] & \text{ if } (a,c)=(i+1,k), 
	\\
	0 & \text{ otherwise}. 
\end{cases}
\end{align}
We also have 
\begin{align*}
[\CX_{(a,c),s}^+, \CE_{(i,k),(j,l)}^t] 
&= [\CX_{(a,c),s}^+, [\CE_{(i,k),(j-1,l)}^t, \CX_{(j-1,l),0}^+]] 
\\
&=  [[\CX_{(j-1,l),0}^+, \CX_{(a,c),s}^+], \CE_{(i,k),(j-1,l)}^t] 
	+ [[\CX_{(a,c),s}^+, \CE_{(i,k),(j-1,l)}^t], \CX_{(j-1,l),0}^+]. 
\end{align*}
Applying the assumption of the induction, we have 
\begin{align}
\label{CX+ CE+ 2}
[\CX_{(a,c),s}^+, \CE_{(i,k),(j,l)}^t] 
= \begin{cases}
	[[\CX_{(j-1,l),0}^+, \CX_{(j,l),s}^+], \CE_{(i,k),(j-1,l)}^t] & \text{ if } (a,c)=(j,l), 
	\\
	[[ \CX_{(j-1,l),0}^+, \CX_{(j-2,l),s}^+], \CE_{(i,k),(j-1,l)}^t] & \text{ if } (a,c)=(j-2,l), 
	\\
	 [ \CE_{(i-1,k),(j-1,l)}^{t+s}, \CX_{(j-1,l),0}^+] & \text{ if } (a,c)=(i-1,k), 
	 \\
	 - [ \CE_{(i,k),(j,l)}^{t+s}, \CX_{(j-1,l),0}^+] & \text{ if } (a,c)=(j-1,l),
	 \\
	 0 & \text{ otherwise}. 
\end{cases}
\end{align}
By \eqref{CX+ CE+ 1} and \eqref{CX+ CE+ 2}, we have 
\begin{align*}
[\CX_{(a,c),s}^+, \CE_{(i,k),(j,l)}^t] 
&= \begin{cases}
	\CE_{(i-1,k),(j,l)}^{t+s} & \text{ if } (a,c) =(i-1,k), 
	\\
	- \CE_{(i,k),(j+1,l)}^{t+s} & \text{ if } (a,c)= (j,l), 
	\\  
	[\CX_{(i,k),0}^+, \CE_{(i,k),(i+2,k)}^{t+s}] 	& \text{ if } (a,c)=(i,k) =(j-2,l), 
	\\ 
	[[\CX_{(i+1,k),s}^+, \CX_{(i,k),0}^+], \CE_{(i+1,k),(i+3,k)}^t] \hspace{-5em} 
		\\ & \text{ if } (a,c) = (i+1,k) = (j-2,l), 
	\\ 
	[\CX_{(i+1,k),0}^-, \CE_{(i,k),(i+2,k)}^{t+s}] & \text{ if } (a,c) =(i+1,k) = (j-1,l), 
	\\ 0 & \text{ otherwise}. 
\end{cases}
\end{align*} 
By the direct calculations using the relations (L4)-(L6), we also have 
\begin{align*}
[\CX_{(i,k),0}^+, \CE_{(i,k),(i+2,k)}^{t+s}] 
=[[\CX_{(i+1,k),s}^+, \CX_{(i,k),0}^+], \CE_{(i+1,k),(i+3,k)}^t]
= [\CX_{(i+1,k),0}^-, \CE_{(i,k),(i+2,k)}^{t+s}] 
=0. 
\end{align*}
Now we proved \eqref{CX+ CE+}. 

We prove \eqref{CI CE+} by the induction on $(j,l)-(i,k)$. 
In the case where $(j,l)-(i,k)=1$, it is just the relation (L2). 
Assume that $(j,l)-(i,k) >1$. 
We have 
\begin{align*}
[\CI_{(a,c),s}, \CE_{(i,k),(j,l)}^t] 
&= [ \CI_{(a,c),s}, [ \CX_{(i,k),0}^+, \CE_{(i+1,k),(j,l)}^t]] 
\\
&=  [ \CX_{(i,k),0}^+, [\CI_{(a,c),s}, \CE_{(i+1,k),(j,l)}^t]] + [[\CI_{(a,c),s}, \CX_{(i,k),0}^+], \CE_{(i+1,k),(j,l)}^t].  
\end{align*}
Applying the assumption of the induction, we have 
\begin{align*}
[\CI_{(a,c),s}, \CE_{(i,k),(j,l)}^t] 
= \begin{cases}
	[\CX_{(i,k),0}^+, \CE_{(i+1,k),(j,l)}^{t+s}] - [\CX_{(i,k),s}^+,  \CE_{(i+1,k),(j,l)}^t] & \text{ if } (a,c)=(i+1,k), 
	\\ 
	- [ \CX_{(i,k),0}^+, \CE_{(i+1,k),(j,l)}^{t+s}] & \text{ if } (a,c)=(j,l), 
	\\
	[ \CX_{(i,k),s}^+, \CE_{(i+1,k),(j,l)}^t] & \text{ if } (a,c)=(i,k), 
	\\
	0 & \text{ otherwise}. 
\end{cases}
\end{align*} 
Thus, we have \eqref{CI CE+} by applying \eqref{CX+ CE+}. 

We prove \eqref{CX- CE+} by the induction on $\ell =(j,l)-(i,k)$.  
In the case where $\ell=1,2$, we can show \eqref{CX- CE+} by direct calculations. 
Assume that $\ell >2$, 
we have 
\begin{align*}
[\CX_{(a,c),s}^-, \CE_{(i,k),(j,l)}^t] 
&= [\CX_{(a,c),s}^-, [\CX_{(i,k),0}^+, \CE_{(i+1,k),(j,l)}^t]] 
\\ 
&=  [ \CX_{(i,k),0}^+, [ \CX_{(a,c),s}^-, \CE_{(i+1,k),(j,l)}^t ]] 
	+ [ [\CX_{(a,c),s}^-, \CX_{(i,k),0}^+], \CE_{(i+1,k),(j,l)}^t]. 
\end{align*}
Applying the assumption of the induction, we have 
\begin{align*}
&[\CX_{(a,c),s}^-, \CE_{(i,k),(j,l)}^t] 
\\
&= \begin{cases}
	[\CX_{(i,k),0}^+, \CE_{(i+2,k),(j,l)}^{t+s}] & \text{ if } (a,c)=(i+1,k) \text{ and } i +1 \not=m_k, 
	\\
	[\CX_{(i,k),0}^+, -Q_k \CE_{(1,k+1),(j,l)}^{t+s} + \CE_{(1, k+1),(j,l)}^{t+s+1} ] 
		& \text{ if } (a,c)=(i+1,k) \text{ and } i+1 =m_k
	\\
	 [\CX_{(i,k),0}^+, - \CE_{(i+1,k),(j-1,l)}^{t+s}] 
		& \text{ if } (a,c)=(j-1,l) \text{ and } j-1 \not=m_l, 
	\\
	[\CX_{(i,k),0}^+, Q_l \CE_{(i+1,k),(m_l,l)}^{t+s} - \CE_{(i+1,k),(m_l,l)}^{t+s+1} ]
		& \text{ if } (a,c)=(j-1,l) \text{ and } j-1 =m_l,
	\\
	[ - \CI_{(i,k),s} + \CI_{(i+1,k),s}, \CE_{(i+1,k),(j,l)}^t] 
		& \text{ if } (a,c)=(i,k) \text{ and } i \not=m_k, 
	\\
	[ Q_k (\CI_{(m_k,k),s} - \CI_{(1,k+1),s}) - \CI_{(m_k,k),s+1} + \CI_{(1,k+1),s+1}, \CE_{(1,k+1),(j,l)}^t] 
		\hspace{-10em} 
		\\
		& \text{ if }(a,c)=(i,k) \text{ and } i =m_k, 
	\\ 0 & \text{ otherwise}. 
\end{cases}
\end{align*} 
Thus, we have \eqref{CX- CE+} by applying \eqref{CX+ CE+} and \eqref{CI CE+}. 

(\roii) is proven in a similar way. 
(\roiii) is just the relations (L1) and (L2). 
\end{proof} 

% #####
By Lemma \ref{Lemma spanning set g},  
we see that 
$\Fg$ is spanned by 
$\{ \CE_{(i,k)(j,l)}^t \,|\, (i,k), (j,l) \in \vG(\Bm), t \geq 0\}$ 
as a $\QQ(\BQ)$-vector space. 
In fact, we see that it is a basis of $\Fg$ as follows. 
  
% ###### 
\begin{prop}\ 
\label{Prop basis g}
$\{ \CE_{(i,k)(j,l)}^t \,|\, (i,k), (j,l) \in \vG(\Bm), t \geq 0\}$  
gives a basis of $\Fg=\Fg_{\BQ}(\Bm)$. 
\end{prop} 

\begin{proof}
It is enough to show that 
$\{ \CE_{(i,k),(j,l)}^t \,|\, (i,k), (j,l) \in \vG(\Bm), t \geq 0\}$ are linearly independent. 

For $\t \in \QQ(\BQ)$, let $V_\t = \bigoplus_{(j,l) \in \vG(\Bm)} \QQ(\BQ) v_{(j,l)}$ 
be the $\Fg$-module given in \ref{Def V tau}. 
Then, we see that 
\begin{align*}
\CE_{(i,k)(j,l)}^t \cdot v_{(a,c)} 
&= \d_{(a,c)(j,l)}\psi_{(i,k)(j,l)} \t^t v_{(i,k)}, 
\end{align*} 
where we put 
\begin{align*}
\psi_{(i,k)(j,l)} = \begin{cases} 
	\dis \prod_{p=0}^{l-k-1} (- Q_l + \t) & \text{ if } l-k >0, 
	\\
	1 & \text{ otherwise}. 
\end{cases}
\end{align*}
Thus, if 
$\sum_{(i,k),(j,l) \in \vG(\Bm), t \geq 0} r_{(i,k)(j,l)}^t \CE_{(i,k), (j,l)}^t=0$ ($r_{(i,k)(j,l)}^t \in \QQ(\BQ)$), 
we have 
\begin{align*}
\big(\sum_{(i,k),(j,l) \in \vG(\Bm), t \geq 0} r_{(i,k)(j,l)}^t \CE_{(i,k), (j,l)}^t \big) \cdot v_{(a,c)} 
&= \sum_{(i,k) \in \vG(\Bm)}  \psi_{(i,k)(j,l)} \big( \sum_{ t \geq 0} r_{(i,k)(a,c)}^t \t^t \big) v_{(i,k)} 
\\
&=0.
\end{align*} 
Thus, for any $(i,k), (j,l) \in \vG(\Bm)$ and any $\t \in \QQ(\BQ)$, we have 
\begin{align*}
\psi_{(i,k)(j,l)} \big( \sum_{ t \geq 0} r_{(i,k)(j,l)}^t \t^t \big) =0.
\end{align*}
This implies that 
$r_{(i,k)(j,l)}^t =0$ for any $(i,k),(j,l) \in \vG(\Bm)$ and any $t \geq 0$.  
\end{proof}

% #######
\para 
Let $\Fn^+$, $\Fn^-$ and $\Fn^0$ be the Lie subalgebras  of $\Fg$ generated by 
\begin{align*}
& \{ \CX_{(i,k),t}^+ \,|\, (i,k) \in \vG'(\Bm), t \geq 0 \},  
\,  
\{ \CX_{(i,k),t}^- \,|\, (i,k) \in \vG'(\Bm), t \geq 0 \} 
\text{ and }
\\
&\{ \CI_{(j,l),t} \,|\, (j,l) \in \vG(\Bm), t \geq 0\} 
\end{align*} 
respectively. 
Then, we have the following triangular decomposition as a corollary of Proposition \ref{Prop basis g}.

\begin{cor} 
\label{Cor tri decom of g}
We have the triangular decomposition 
\begin{align*}
\Fg = \Fn^- \oplus \Fn^0 \oplus \Fn^+ \quad (\text{as vector spaces}).
\end{align*}
\end{cor}
% ###############################################################################

\para 
\textbf{A current Lie algebra. } 
Let $\QQ[x]$ be the polynomial ring over $\QQ$, 
and 
let $\Fgl_m [x] = \QQ[x] \otimes \Fgl_m$ be the current Lie algebra associated with the general linear Lie algebra 
$\Fgl_m$ over $\QQ$. 
Namely, the Lie bracket on $\Fgl_m [x] $ is defined by 
\begin{align*}
[a\otimes g, b \otimes h] = ab \otimes [g,h] 
\quad (a,b \in \QQ[x], \, g,h \in \Fgl_m).
\end{align*}

Let $E_{i,j} \in \Fgl_m$ ($1 \leq i,j \leq m$) be the elementary matrix having 
$1$ at the $(i,j)$-entry and $0$ elsewhere. 
Put 
$e_i =E_{i,i+1}$, $f_i =E_{i+1,i}$ and  $K_j = E_{j,j}$. 
Then $\QQ[x] \otimes \Fgl_m$ is generated by 
\begin{align*}
x^t \otimes e_i, \, x^t \otimes f_i,   \, x^t \otimes K_j  
\quad (1 \leq i \leq m-1, \, 1 \leq j \leq m, \, t \geq 0). 
\end{align*}

%%%%%%%

\para 
In the case where $r=1$ ($\Bm=m$), 
the Lie algebra $\Fg(m)$ over $\QQ$ 
is generated by $\CX_{i,t}^{\pm}$ and $\CI_{j,t}$ ($1 \leq i \leq m-1$, $1 \leq j \leq m$, $t \geq 0$) 
with the defining relations (L1)-(L6) (for $(i,1) \in \vG(m)$, we denote $(i,1)$ by $i$ simply). 
In this case, the relation (L3) is just 
\begin{align*}
[\CX_{i,t}^+, \CX_{j,s}^-] = \d_{i,j} (\CI_{i,t} - \CI_{i+1,t}).  
\end{align*}
Then, we have the following lemma.  

\begin{lem}
\label{Lem iso g(m) to current}
There exists the isomorphism of Lie algebras 
\begin{align*} 
\Phi: \Fg(m) \ra \Fgl_m[x]
\quad 
( \CX_{i,t}^+ \mapsto x^t \otimes e_i, \, \CX_{i,t}^- \mapsto x^t \otimes f_i, \, \CI_{j,t} \mapsto x^t \otimes K_j). 
\end{align*}
In particular, 
the relations (L1)-(L6) (in the case where $r=1$) 
give a defining relations of $\Fgl_m[x]$ through the isomorphism $\Phi$. 
\end{lem}
\begin{proof}
We can show the well-definedness of the homomorphism $\Phi$ by  checking the defining relations of $\Fg(m)$ directly.  

For $i,j \in \{1,\dots,m\}$ and $t \geq 0$, 
we see that 
$\Phi (\CE_{i,j}^t) = x^t \otimes E_{i,j}$. 
Clearly, $\{ x^t \otimes E_{i,j} \,|\, 1 \leq i,j \leq m, \, t \geq 0\}$ gives a basis of $\Fgl_m [x]$. 
Thus, Proposition \ref{Prop basis g} implies that $\Phi$ is  isomorphic.  
\end{proof}

% #####

\para 
In the case where $r \geq 2$, 
we can regard $\Fg=\Fg_{\BQ}(\Bm)$ as a deformation of the current Lie algebra 
$\QQ(\BQ) \otimes_{\QQ} \Fgl_m[x]$ as follows. 

For $t \geq 0$, 
put 
\begin{align*}
\CY_t = \{ \CX_{(i,k),t}^{\pm}, \CI_{(j,l), t} \,|\,  (i,k) \in \vG'(\Bm), (j,l) \in \vG(\Bm) \}.
\end{align*}

Let $\Fg_t$ be the $\QQ(\BQ)$-subspace of $\Fg$ spanned by 
\begin{align*}
\{ [Y_{t_1}, [Y_{t_2}, \dots,[ Y_{t_{p-1}}, Y_{t_p}] \dots ] 
\,|\, 
Y_{t_b} \in \CY_{t_b}, \, \sum_{b=1}^p t_b \geq t, \, p\geq1
\}.
\end{align*}
Then, we have  the sequence 
\begin{align*}
\Fg= \Fg_0 \supset \Fg_1 \supset \Fg_2 \supset \dots. 
\end{align*} 
By the defining relations (L1)-(L6), we see that 
\begin{align}
\label{filtered g}
[\Fg_s, \Fg_t] \subset \Fg_{s+t} \quad (s,t \geq 0).
\end{align} 

For $t \geq 0$, 
let $\s_t : \Fg_t \ra \Fg_t / \Fg_{t+1}$ be the natural surjection. 
By  \eqref{filtered g}, we can define the structure as a Lie algebra on 
$\gr \Fg = \bigoplus_{t \geq 0} \Fg_t / \Fg_{t+1}$ by 
\begin{align*}
[\s_s(g), \s_{t} (h)] =\s_{s+t} ([g,h]) \quad (g \in \Fg_s, h \in \Fg_t). 
\end{align*}
Then we see that, 
$\gr \Fg$ is generated by 
\begin{align*}
\s_t (\CX_{(i,k),t}^{\pm}),  \, \s_t (\CI_{(j,l),t}) 
\quad  
((i,k) \in \vG'(\Bm), \,  (j,l) \in \vG(\Bm), \, t \geq 0), 
\end{align*}  
and $\gr \Fg$ has a basis 
$\{ \s_t (\CE_{(i,k),(j,l)}^t) \,|\, (i,k), (j,l) \in \vG(\Bm), \, t \geq 0 \}$.

\begin{prop}
\label{Prop glm[x] iso gr g}
There exists the isomorphism of Lie algebras 
\begin{align*}
\Psi: \QQ(\BQ) \otimes_{\QQ} \Fgl_m[x] 
\ra \gr \Fg = \bigoplus_{t \geq 0} \Fg_t / \Fg_{t+1} 
\end{align*}
such that 
\begin{align*}
& x^t \otimes e_{(i,k)}
\mapsto 
\begin{cases} 
\s_t ( \CX_{(i,k),t}^+) & \text{ if } i \not= m_k, 
\\
- Q_k^{-1} \s_t ( \CX_{(m_k,k),t}^+) & \text{ if } i=m_k, 
\end{cases}
\\
& x^t \otimes f_{(i,k)} \mapsto \s_t ( \CX_{(i,k),t}^-), 
\\
& x^t \otimes K_{(j,l)} \mapsto \s_t (\CI_{(j,l),t}), 
\end{align*}
where we use the identification \eqref{identify vG 1 to m} for the indices of generators of $\Fgl_m [x]$.  
\end{prop} 
\begin{proof}
We can show the well-definedness of the homomorphism 
$\Psi$ by checking the defining relations of $\Fgl_m[x]$ directly (see Lemma \ref{Lem iso g(m) to current}). 
We also see that 
\begin{align*}
\Psi (x^t \otimes E_{(i,k),(j,l)}) = \psi_{(i,k)(j,l)} \s_t (\CE_{(i,k),(j,l)}^t), 
\end{align*}
where we put 
\begin{align*}
\psi_{(i,k)(j,l)} = 
	\begin{cases} 
	\dis 	\prod_{p=0}^{l-k-1} (-Q_{k+p}^{-1}) & \text{ if } l-k >0, 
		\\
		1 & \text{ otherwise}.
	\end{cases}
\end{align*} 
Thus, we see that $\Psi$ is isomorphic. 
\end{proof}

% #####

As a corollary of the above proposition, we have the following isomorphism between 
$\QQ(\BQ) \otimes_{\QQ} \Fg(m)$ and $\gr \Fg_{\BQ}(\Bm)$. 
\begin{cor} 
\label{Cor iso g(m) gQ(m)}
There exists the isomorphism of Lie algebras 
\begin{align*} 
\wt{\Psi} : \QQ(\BQ) \otimes_{\QQ} \Fg(m)  \ra \gr \Fg_{\BQ}(\Bm) = \bigoplus_{t \geq 0} \Fg_t / \Fg_{t+1} 
\end{align*}
such that 
\begin{align*}
\CX_{(i,k),t}^+ \mapsto 
	\begin{cases} 
		\s_t (\CX_{(i,k),t}^+) & \text{ if } i \not= m_k, 
		\\ 
		- Q_k^{-1} \s_t (\CX_{(m_k,k),t}^+ & \text{ if } i=m_k, 
	\end{cases} 
\quad  
\CX_{(i,k),t}^- \mapsto \s_t (\CX_{(i,k),t}^-), 
\,\,
\CI_{(j,l),t}  \mapsto \s_t (\CI_{(j,l),t}), 
\end{align*}
where we use the identification \eqref{identify vG 1 to m} for the indices of generators of $\Fg(m)$.  
\end{cor}

%%%%%%%%%%%%%
\para 
We also 
have some relations between the Lie algebra $\Fg_{\BQ}(\Bm)$ and the general linear Lie algebra $\Fgl_m$
as follows. 
Let $\Fgl_{m_1}\oplus \dots \oplus \Fgl_{m_r}$ be a Levi subalgebra of 
$\Fgl_m$ associated with $\Bm=(m_1,\dots, m_r)$. 
Then generates of $\Fgl_{m_1} \oplus \dots \oplus \Fgl_{m_r}$ are given by 
$e_{(i,k)}$, $f_{(i,k)}$ ($1 \leq i \leq m_k-1, 1 \leq k \leq r$) and $K_{(j,l)}$ ($(j,l) \in \vG(\Bm)$), 
where we use the identification \eqref{identify vG 1 to m} for indices.   

\begin{prop}\
\begin{enumerate}
\item 
There exists a surjective homomorphism of Lie algebras 
\begin{align}
\label{surjection gQm to glm}
g : \Fg_{\BQ}(\Bm) \ra \Fgl_m
\end{align}
such that 
\begin{align*} 
&g(\CX_{(i,k),0}^+)
	= \begin{cases} e_{(i,k)} & \text{ if } i \not=m_k, \\ - Q_k e_{(m_k,k)} & \text{ if } i=m_k, \end{cases}
\, \,
g (\CX_{(i,k),0}^-)= f_{(i,k)}, 
\\
& g(\CI_{(j,l),0})= K_{(j,l)}
\text{ and }
g(\CX_{(i,k),t}^{\pm}) =g(\CI_{(j,l),t})=0 \text{ for } t \geq 1. 
\end{align*} 

\item 
There exists an injective homomorphism of Lie algebras 
\begin{align}
\label{injection g Levi to gQm}
\iota : \Fgl_{m_1} \oplus \dots \oplus \Fgl_{m_r} \ra \Fg_{\BQ}(\Bm) 
\end{align}
such that 
$\iota(e_{(i,k)}) =\CX_{(i,k),0}^+$, $\iota(f_{(i,k)}) = \CX_{(i,k),0}^-$ and $\iota(K_{(j,l)})=\CI_{(j,l),0}$. 
\end{enumerate}
\end{prop}
\begin{proof}
We can check the well-definedness of $g$ and $\iota$ by direct calculations. 
Clearly $g$ is surjective. 
Let $ \iota' : \Fgl_{m_1} \oplus \dots \oplus \Fgl_{m_r} \ra \Fgl_m$ be the natural embedding. 
Then, by investigating the image of generators, 
we see  that $\iota' = g \circ \iota$. 
This implies that $\iota$ is injective. 
\end{proof}

%%%%%
\remark 
\label{Remark evaluation gQ}
The surjective homomorphism $g$ in \eqref{surjection gQm to glm} can be regarded as a special case 
of evaluation homomorphisms. 
However, we can not define evaluation homomorphisms for $\Fg_{\BQ}(\Bm)$ in general 
although we can consider $\Fg_{\BQ}(\Bm)$-modules corresponding to some evaluation modules. 
They will be studied in a subsequent paper. 

%%%%%%%%%%%%%%%%%%%%%%%%%%%%%%%%%%%%%%%%%%%%%%%%%%%%%%%%%%%%%%%

%%%%%%%%%%%%%%%%%%%%%%%%%%%%%%%%%%%%%%%%%%%%%%%%%%%%%%%%%%%%%%%

\section{Representations of $\Fg_{\BQ}(\Bm)$} 
\label{S Rep FgQ}
Thanks to the triangular decomposition in Corollary \ref{Cor tri decom of g}, 
we can develop the weight theory to study some representations of $\Fg_{\BQ}(\Bm)$ 
in the usual manner as follows.

\para 
Let $U(\Fg)=U(\Fg_{\BQ}(\Bm))$ be the universal enveloping algebra of the Lie algebra $\Fg_{\BQ}(\Bm)$. 
Then, by Corollary \ref{Cor tri decom of g} together with PBW theorem, 
we have the triangular decomposition 
\begin{align}
\label{tri decom U(g)}
U(\Fg) \cong U(\Fn^-) \otimes U(\Fn^0) \otimes U(\Fn^+). 
\end{align}
Thanks to the triangular decomposition, 
we can develop the weight theory for $U(\Fg)$-modules as follows. 

%
%%%%%
% 

\para 
\textbf{Highest weight modules.} 
For $\la \in P$ and a multiset $\Bvf =(\vf_{(j,l),t} \,|\, (j,l) \in \vG(\Bm), t \geq 1)$ ($\vf_{(j,l),t} \in \QQ(\BQ)$), 
we say that a $U(\Fg)$-modules $M$ is a highest weight modules of highest weight $(\la, \Bvf)$ 
if there exists an element $v_0 \in M$ satisfying the following three conditions: 
\begin{enumerate} 
\item 
$M$ is generated by $v_0$ as a $U(\Fg)$-module, 

\item 
$\CX_{(i,k),t}^+ \cdot v =0$ for all $(i,k) \in \vG'(\Bm)$ and $t \geq 0$, 

\item 
$\CI_{(j,l),0} \cdot v_0 = \lan \la, h_{(j,l)} \ran v_0$  
and 
$\CI_{(j,l),t} \cdot v_0= \vf_{(j,l),t} v_0$ 
for $(j,l) \in \vG(\Bm)$ and $t \geq 1$. 
\end{enumerate} 
If an element $v_0 \in M$ satisfies the above conditions (\roii) and (\roiii), 
we say that $v_0$ is a maximal vector of weight $(\la, \Bvf)$. 
In this case, the submodule $U(\Fg) \cdot v_0$ of $M$ 
is a highest weight module of highest weight $(\la, \Bvf)$. 
If a maximal vector $v_0 \in M$ satisfies the above condition (\roi), 
we say that $v_0$ is a highest weight vector. 

For a highest weight $U(\Fg)$-module $M$ of highest weight $(\la, \Bvf)$ 
with a highest weight vector $v_0 \in M$, 
we have $M=U(\Fn^-) \cdot v_0$ by the triangular decomposition \eqref{tri decom U(g)}. 
Thus, the relation (L2) implies the weight space decomposition 
\begin{align}
\label{wt sp decom h.w. g module}
M = \bigoplus_{\mu \in P \atop \mu \leq \la} M_{\mu} 
\text{ such that } 
\dim_{\QQ(\BQ)} M_{\la} =1, 
\end{align}
where $M_{\mu} = \{ v \in M \,|\, \CI_{(j,l),0} \cdot v = \lan \mu, h_{(j,l)} \ran v \text{ for } (j,l) \in \vG(\Bm) \}$. 

%
%%%%%
%

\para
\textbf{Verma modules.} 
Let $U(\Fn^{\geq 0})$ be the subalgebra of $U(\Fg)$ generated by $U(\Fn^0)$ and $U(\Fn^+)$. 
Then, 
by Proposition \ref{Prop basis g} together with the proof of Lemma \ref{Lemma spanning set g}, 
we see that $U(\Fn^+)$ (resp. $U(\Fn^-)$) is isomorphic to the algebra generated by 
$\{\CX_{(i,k),t}^+ \,|\, (i,k) \in \vG'(\Bm), t \geq 0\}$ 
(resp. $\{\CX_{(i,k),t}^- \,|\, (i,k) \in \vG'(\Bm), t \geq 0\}$) 
with the defining relations (L4)-(L6), 
$U(\Fn^0)$ is isomorphic to the algebra generated by $\{\CI_{(j,l),t} \,|\, (j,l) \in \vG(\Bm), t \geq 0\}$ 
with the defining relations (L1), 
and that 
$U(\Fn^{\geq 0})$ is isomorphic to the algebra generated by 
$\{ \CX_{(i,k) t}^+, \CI_{(j,l) t} \,|\, (i,k) \in \vG'(\Bm), (j,l) \in \vG(\Bm), t \geq 0\}$ 
with the defining relations (L1)-(L6) except (L3). 
Then we have the surjective homomorphism of algebras 
\begin{align}
\label{surj n geq 0 to n 0}
U(\Fn^{\geq 0}) \ra U(\Fn^0) 
\text{ such that } 
\CX_{(i,k),t}^+ \mapsto 0, \CI_{(j,l),t} \mapsto \CI_{(j,l),t}. 
\end{align} 

For $\la \in P$ and a multiset $\Bvf=(\vf_{(j,l),t})$, 
we define a ($1$-dimensional) simple $U(\Fn^0)$-module $\Theta_{(\la, \Bvf)} = \QQ(\BQ) v_0$ by 
\begin{align*} 
\CI_{(j,l),0} \cdot v_0 = \lan \la, h_{(j,l)} \ran v_0, 
\quad 
\CI_{(j,l) t} \cdot v_0 = \vf_{(j,l),t} v_0 
\end{align*}
for $(j,l) \in \vG(\Bm)$ and $t \geq 1$. 
Then we define the Verma module $M(\la,\Bvf)$ as the induced module 
\begin{align*} 
M(\la,\Bvf) = U(\Fg) \otimes_{U(\Fn^{\geq 0})} \Theta_{(\la,\Bvf)}, 
\end{align*} 
where we regard $\Theta_{(\la,\Bvf)}$ as a left $U(\Fn^{\geq 0})$-module 
through the surjection \eqref{surj n geq 0 to n 0}. 

By definitions, the Verma module $M(\la, \Bvf)$ is a highest weight module of highest weight $(\la, \Bvf)$ 
with a highest weight vector $1 \otimes v_0$. 
Then we see that any highest weight module of highest weight $(\la, \Bvf)$ is a quotient of $M(\la, \Bvf)$ 
by the universality of tensor products. 
We also see that $M(\la, \Bvf)$ has the unique simple top 
$L(\la, \Bvf) = M(\la, \Bvf)/ \rad M(\la, \Bvf)$ 
from the weight space decomposition \eqref{wt sp decom h.w. g module}. 

By using the homomorphism 
$\iota : U(\Fgl_{m_1} \oplus \dots \oplus \Fgl_{m_r}) \ra U(\Fg)$ 
induced from \eqref{injection g Levi to gQm}, 
we have a necessary condition for $L(\la, \Bvf)$ to be finite dimensional as follows. 

\begin{prop} 
\label{Prop necessary condition fin. simpe}
For $\la \in P$ and a multiset $\Bvf=(\vf_{(j,l),t})$, 
if $L(\la, \Bvf)$ is finite dimensional, 
then we have $\la \in P^+_{\Bm}$. 
\end{prop} 

\begin{proof} 
Assume that $L(\la, \Bvf)$ is finite dimensional. 
Let $v_0 \in L(\la, \Bvf)$ be a highest weight vector. 
When we regard $L(\la, \Bvf)$ as a $U(\Fgl_{m_1} \oplus \dots \oplus \Fgl_{m_r})$-module 
through the injection $\iota : U(\Fgl_{m_1} \oplus \dots \oplus \Fgl_{m_r}) \ra U (\Fg)$, 
we see that 
$U(\Fgl_{m_1} \oplus \dots \oplus \Fgl_{m_r})$-submodule of $L(\la, \Bvf)$ 
generated by $v_0$ 
is a (finite dimensional) highest weight 
$U(\Fgl_{m_1} \oplus \dots \oplus \Fgl_{m_r})$-module 
of highest weight $\la$.  
Thus, the Lemma follows from the well-known facts for 
$U(\Fgl_{m_1} \oplus \dots \oplus \Fgl_{m_r})$-modules. 
\end{proof}

% 
%%%%%
%

\para 
\textbf{Category $\ZC_{\BQ}(\Bm)$.} 
Let $\ZC_{\BQ}(\Bm)$ (resp. $\ZC_{\BQ}^{\geq 0}(\Bm)$) 
be the full subcategory of $U(\Fg) \cmod$ consisting of $U(\Fg)$-modules 
satisfying the following conditions: 
\begin{enumerate} 
\item 
If $M \in \ZC_{\BQ}(\Bm)$ (resp. $M \in \ZC_{\BQ}^{\geq 0}(\Bm)$),  then $M$ is finite dimensional, 

\item 
If $M \in \ZC_{\BQ}(\Bm)$ (resp. $M \in \ZC_{\BQ}^{\geq 0}(\Bm)$), 
then $M$ has the weight space decomposition 
\begin{align*} 
M = \bigoplus_{\la \in P} M_{\la} 
\quad 
(\text{resp. } M= \bigoplus_{\la \in P_{\geq 0}} M_{\la}), 
\end{align*}
where $M_{\la} = \{ v \in M \,|\, \CI_{(j,l),0} \cdot v = \lan \la, h_{(j,l)} \ran v \text{ for } (j,l) \in \vG(\Bm) \}$, 

\item 
If $M \in \ZC_{\BQ}(\Bm)$ (resp. $M \in \ZC_{\BQ}^{\geq 0}(\Bm)$), 
then all eigenvalues of the action of $\CI_{(j,l),t}$ ($(j,l) \in \vG(\Bm), t \geq 0$) on $M$ belong to $\QQ(\BQ)$. 
\end{enumerate}

By the usual argument, we have the following lemma. 

\begin{lem} 
Any simple object in $\ZC_{\BQ}(\Bm)$ is a highest weight module. 
\end{lem}

By using the surjection $g : U(\Fg) \ra U(\Fgl_m)$ induced from \eqref{surjection gQm to glm},  
we have the following proposition. 

\begin{prop} 
\label{Prop glm fullsub ZCQ(m)}
Let $\ZC_{\Fgl_m}$ be the category of finite dimensional $U(\Fgl_m)$-modules 
which have the weight space decomposition. 
Then, we have the followings. 
\begin{enumerate} 
\item 
$\ZC_{\Fgl_m}$ is a full subcategory of $\ZC_{\BQ}(\Bm)$ through the surjection 
$g : U(\Fg) \ra U(\Fgl_m)$. 

\item 
For $\la \in P^+$, 
the simple highest weight $U(\Fgl_m)$-module $\D_{\Fgl_m}(\la)$ of highest weight $\la$ 
is the simple highest weight $U(\Fg)$-module of highest weight $(\la, \mathbf{0})$ 
through the surjection $g : U(\Fg) \ra U(\Fgl_m)$, 
where $\mathbf{0}$ means $\vf_{(j,l),t}=0$ for all $(j,l) \in \vG(\Bm)$ and $t \geq 1$. 
\end{enumerate}
\end{prop}

%%%%%%%%%%%%%%%%%%%%%%%%%%%%%%%%%%%%%%%%%%%%%%%%%%%%%%%%%%%%%%%

%%%%%%%%%%%%%%%%%%%%%%%%%%%%%%%%%%%%%%%%%%%%%%%%%%%%%%%%%%%%%%%

\section{Algebra $\CU_{q,\BQ}(\Bm)$} 
\label{CUqQ}
In this section, 
we introduce an algebra $\CU_{q,\BQ}(\Bm)$ 
with parameters $q$ and $\BQ=(Q_1,\dots, Q_{r-1})$ 
associated with the Cartan data in the paragraph \ref{Cartan data}. 
Then we study some basic structures of $\CU_{q,\BQ}(\Bm)$. 
In particular, we can regard $\CU_{q,\BQ}(\Bm)$ as a \lq\lq $q$-analogue" 
of the universal enveloping algebra $U(\Fg_{\BQ}(\Bm))$ of the Lie algebra $\Fg_{\BQ}(\Bm)$ 
introduced in the section \S \ref{Section FgQ(Bm)}.

\para 
Put $\AA =\ZZ[\BQ][q,q^{-1}] = \ZZ[q,q^{-1}, Q_1,\dots, Q_{r-1}]$, 
where $q, Q_1,\dots, Q_{r-1}$ are indeterminate elements over $\ZZ$, 
and let $\KK = \QQ(q, Q_1,\dots, Q_{r-1})$ be the quotient field of $\AA$.

% ###################

\begin{definition}
We define the associative algebra 
$\CU=\CU_{q,\BQ}(\Bm)$ 
over  $\KK$ by the following generators and defining relations: 
\begin{description}
\item[Generators] 
$\CX_{(i,k),t}^{\pm}$,  $\CI_{(j,l),t}^{\pm}$, $\CK_{(j,l)}^{\pm}$ ($(i,k) \in \vG'(\Bm)$, $(j,l) \in \vG(\Bm)$, $t \geq 0$). 

\item[Relations] 
\begin{align*}
\\[-2em]
&\tag{R1} 
	\CK_{(j,l)}^+ \CK_{(j,l)}^- = \CK_{(j,l)}^- \CK_{(j,l)}^+ =1, \quad 
	(\CK_{(j,l)}^{\pm})^2 = 1 \pm (q-q^{-1}) \CI_{(j,l),0}^{\mp},
\\
& \tag{R2} 
	[\CK_{(i,k)}^+, \CK_{(j,l)}^+] = [\CK_{(i,k)}^+, \CI_{(j,l),t}^{\s}] = [\CI_{(i,k),s}^{\s}, \CI_{(j,l),t}^{\s'}] =0 \,  \, 
	(\s,\s' \in \{+, -\}), 
\\
& \tag{R3} 
	\CK_{(j,l)}^+ \CX_{(i,k),t}^{\pm} \CK_{(j,l)}^- = q^{\pm a _{(i,k)(j,l)}} \CX_{(i,k),t}^{\pm}, 
\\
\tag{R4} 
	\begin{split}
	 &q^{\pm a_{(i,k)(j,l)}} \CI_{(j,l),0}^{\pm} \CX_{(i,k),t}^{+} - q^{\mp a_{(i,k)(j,l)}} \CX_{(i,k),t}^{+} \CI_{(j,l),0}^{\pm}   
		=  a_{(i,k)(j,l)} \CX_{(i,k),t}^{+}, 
	\\ 
	 &q^{\mp a_{(i,k)(j,l)}} \CI_{(j,l),0}^{\pm} \CX_{(i,k),t}^{-} - q^{\pm a_{(i,k)(j,l)}} \CX_{(i,k),t}^{-} \CI_{(j,l),0}^{\pm}  
		= -  a_{(i,k)(j,l)} \CX_{(i,k),t}^{-},  
	\end{split}
\\
\tag{R5}
	\begin{split} 
	&[\CI_{(j,l),s+1}^{\pm}, \CX_{(i,k),t}^{+}] 
		= q^{ \pm a_{(i,k)(j,l)}} \CI_{(j,l),s}^{\pm} \CX_{(i,k),t+1}^{+} 
			- q^{ \mp a_{(i,k)(j,l)}} \CX_{(i,k),t+1}^{+} \CI_{(j,l),s}^{\pm}, 
	\\
	&[\CI_{(j,l),s+1}^{\pm}, \CX_{(i,k),t}^{-}] 
		= q^{\mp a_{(i,k)(j,l)}} \CI_{(j,l),s}^{\pm} \CX_{(i,k),t+1}^{-} 
			- q^{\pm a_{(i,k)(j,l)}} \CX_{(i,k),t+1}^{-} \CI_{(j,l),s}^{\pm},  
	\end{split}
\\
&\tag{R6} 
	[\CX_{(i,k),t}^+, \CX_{(j,l),s}^-] 
	\\
	& = \d_{(i,k),(j,l)} \begin{cases} 
		\dis 
		\wt{\CK}_{(i,k)}^+ \CJ_{(i,k),s+t} & \text{ if } i \not=m_k, 
		\\ 
		\dis 
		- Q_k \wt{\CK}_{(m_k,k)}^+ \CJ_{(m_k,k),s+t} + \wt{\CK}_{(m_k,k)}^+ \CJ_{(m_k,k),s+t+1} 
			& \text{ if } i=m_k, 
	\end{cases} 
\\
\tag{R7} 
	\begin{split} 
	&[\CX_{(i,k),t}^{\pm}, \CX_{(j,l),s}^{\pm}] =0 \quad \text{ if } (j,l) \not= (i,k), (i \pm 1, k), 
	\\ &
	\CX_{(i,k),t+1}^{\pm} \CX_{(i,k),s}^{\pm} - q^{\pm 2} \CX_{(i,k),s}^{\pm} \CX_{(i,k),t+1}^{\pm} 
		= q^{\pm 2} \CX_{(i,k),t}^{\pm} \CX_{(i,k),s+1}^{\pm} - \CX_{(i,k),s+1}^{\pm} \CX_{(i,k),t}^{\pm},  
	\\ & 
	\CX_{(i,k),t+1}^+ \CX_{(i+1,k),s}^+ - q^{-1} \CX_{(i+1,k),s}^+ \CX_{(i,k),t+1}^+ 
		= \CX_{(i,k),t}^+ \CX_{(i+1,k),s+1}^+ - q \CX_{(i+1,k),s+1}^+ \CX_{(i,k),t}^+, 
	\\ & 
	\CX_{(i+1,k),s}^- \CX_{(i,k),t+1}^- - q^{-1} \CX_{(i,k),t+1}^- \CX_{(i+1,k),s}^- 
		= \CX_{(i+1,k),s+1}^- \CX_{(i,k),t}^- - q \CX_{(i,k),t}^- \CX_{(i+1,k),s+1}^-, 
	\end{split} 
\\
\tag{R8} 
	\begin{split}
	& \CX_{(i \pm 1,k),u}^+ \big( \CX_{(i,k),s}^+ \CX_{(i,k),t}^+ + \CX_{(i,k),t}^+ \CX_{(i,k),s}^+ \big) 
		+ \big( \CX_{(i,k),s}^+ \CX_{(i,k),t}^+ + \CX_{(i,k),t}^+ \CX_{(i,k),s}^+ \big) \CX_{(i \pm 1,k),u}^+ 
	\\ 
	&= (q+q^{-1}) \big( \CX_{(i,k),s}^+ \CX_{(i \pm1,k),u}^+ \CX_{(i,k),t}^+ 
		+ \CX_{(i,k),t}^+ \CX_{(i \pm 1,k),u}^+ \CX_{(i,k),s}^+ \big), 
	\\[0.5em]
	& \CX_{(i \pm 1,k),u}^- \big( \CX_{(i,k),s}^- \CX_{(i,k),t}^- + \CX_{(i,k),t}^- \CX_{(i,k),s}^- \big) 
		+ \big( \CX_{(i,k),s}^- \CX_{(i,k),t}^- + \CX_{(i,k),t}^- \CX_{(i,k),s}^- \big) \CX_{(i \pm 1,k),u}^- 
	\\ 
	&= (q+q^{-1}) \big( \CX_{(i,k),s}^- \CX_{(i \pm1,k),u}^- \CX_{(i,k),t}^- 
		+ \CX_{(i,k),t}^- \CX_{(i \pm 1,k),u}^- \CX_{(i,k),s}^- \big), 
	\end{split}
\end{align*}
\end{description}
where we put $\wt{\CK}_{(i,k)}^+ = \CK_{(i,k)}^+ \CK_{(i+1,k)}^-$, 
	$\wt{\CK}_{(i,k)}^- = \CK_{(i,k)}^- \CK_{(i+1,k)}^+$ and 
\begin{align*}
\CJ_{(i,k),t} = 
\begin{cases}
	\CI_{(i,k),0}^+ - \CI_{(i+1,k),0}^-  + (q-q^{-1}) \CI_{(i,k),0}^+ \CI_{(i+1,k),0}^- 
	& \text{ if } t=0, 
	\\
	\dis q^{-t} \CI_{(i,k),t}^+ - q^t \CI_{(i+1,k),t}^- 
	- (q-q^{-1}) \sum_{b=1}^{t-1} q^{- t + 2 b} \CI_{(i,k),t-b}^+ \CI_{(i+1,k),b}^- 
	& \text{ if } t >0.
	\end{cases} 
\end{align*} 
\end{definition}

\remark 
The relations (R4) follows from the relations (R1) and (R3) in $\CU_{q,\BQ}(\Bm)$. 
Thus, we do not need the relations (R4) as a defining relations of $\CU_{q,\BQ}(\Bm)$. 
However, (R4) does not follows from (R1) and (R3) in the integral forms 
$\CU_{\AA,q,\BQ}^\star (\Bm)$ and $\CU_{\AA,q,\BQ}(\Bm)$ defined below. 
Then, we require the relations (R4) in a defining relations of $\CU_{q,\BQ}(\Bm)$.

\para 
By the relation (R1), for $(i,k) \in\vG'(\Bm)$, we have 
\begin{align}
\wt{\CK}_{(i,k)}^+ \CJ_{(i,k),0} = \frac{\wt{\CK}_{(i,k)}^+ - \wt{\CK}_{(i,k)}^-}{q-q^{-1}}. 
\end{align}
Thus, in the case where $s=t=0$, we can replace the relation (R6) by 
\begin{align}
[\CX_{(i,k),0}^+, \CX_{(j,l),0}^-] 
= 	\d_{(i,k),(j,l)}  
	\begin{cases} 
		\dis \frac{\wt{\CK}_{(i,k)}^+ - \wt{\CK}_{(i,k)}^-}{q-q^{-1}} & \text{ if } i \not= m_k, 
		\\
		\dis - Q_k \frac{\wt{\CK}_{(m_k,k)}^+ - \wt{\CK}_{(m_k,k)}^-}{q-q^{-1}} 
		+ \wt{\CK}_{(m_k,k)}^+ \CJ_{(m_k,k),1} 
		& \text{ if } i=m_k. 
	\end{cases}
\end{align} 
By (R8), if $s=t$, we have 
\begin{align}
\begin{split}
&\CX_{(i \pm 1,k),u}^+ (\CX_{(i,k),t}^+)^2 - (q+q^{-1}) \CX_{(i,k),t}^+ \CX_{(i \pm 1,k),u}^+ \CX_{(i,k),t}^+ 
	+ (\CX_{(i,k),t}^+)^2 \CX_{(i \pm 1,k),u}^+ =0, 
\\
&\CX_{(i \pm 1,k),u}^- (\CX_{(i,k),t}^-)^2 - (q+q^{-1}) \CX_{(i,k),t}^- \CX_{(i \pm 1,k),u}^- \CX_{(i,k),t}^- 
	+ (\CX_{(i,k),t}^-)^2 \CX_{(i \pm 1,k),u}^- =0.  
\end{split}
\end{align}
By (R4) and (R5), we have 
\begin{align}
& [\CI_{(j,l),1}^+, \CX_{(i,k),t}^{\pm}] = [\CI_{(j,l),1}^-, \CX_{(i,k),t}^{\pm}] =\pm a_{(i,k)(j,l)} \CX_{(i,k), t+1}^{\pm}. 
\end{align} 
By the induction on $s$ using the relation (R6), for $s \geq 1$, we can show that 
\begin{align} 
\label{CI CX+}
\begin{split}
&[\CI_{(j,l),s}^{\pm}, \CX_{(i,k),t}^{+}] 
\\&
= a_{(i,k)(j,l)} q^{\pm a_{(i,k)(j,l)} (s-1)} \CX_{(i,k),t+s}^{+} 
	\pm a_{(i,k)(j,l)} (q-q^{-1}) \sum_{p=1}^{s-1} q^{\pm a_{(i,k)(J,l)} (p-1)}\CX_{(i,k),t+p}^+ \CI_{(j,l),s-p}^{\pm} 
\\
&
= a_{(i,k)(j,l)} q^{\mp a_{(i,k)(j,l)} (s-1)} \CX_{(i,k), t+s}^+ 
	\pm a_{(i,k)(j,l)} (q-q^{-1}) \sum_{p=1}^{s-1} q^{\mp a_{(i,k)(j,l)} (p-1)} \CI_{(j,l),s-p}^{\pm} \CX_{(i,k),t+p}^+,  
\end{split} 
\end{align}
and 
\begin{align}
\label{CI CX-}
\begin{split}
&[\CI_{(j,l),s}^{\pm}, \CX_{(i,k),t}^-] 
\\
&= - a_{(i,k)(j,l)} q^{\mp a_{(i,k)(j,l)} (s-1)} \CX_{(i,k),t+s}^- 
	\mp a_{(i,k)(j,l)} (q-q^{-1}) \sum_{p=1}^{s-1} q^{\mp a_{(i,k)(j,l)} (p-1)} \CX_{(i,k),t+p}^- \CI_{(j,l),s-p}^{\pm} 
\\
&= - a_{(i,k)(j,l)} q^{\pm a_{(i,k)(j,l)} (s-1)} \CX_{(i,k),t+s}^- 
	\mp a_{(i,k)(j,l)} (q-q^{-1}) \sum_{p=1}^{s-1}  q^{\pm a_{(i,k)(j,l)} (p-1)} \CI_{(j,l),s-p}^{\pm} \CX_{(i,k),t+p}^-.
\end{split}
\end{align}

% ##################

\para 
Let $\CU^+ = \CU_{q,\BQ}^+(\Bm)$, $\CU^- = \CU^-_{q,\BQ}(\Bm)$ and $\CU^0 = \CU^0_{q,\BQ}(\Bm)$ 
be the subalgebra of $\CU$ generated by 
\begin{align*}
& \{ \CX_{(i,k),t}^+ \,|\, (i,k) \in \vG' (\Bm), \, t \geq 0\}, 
\\
& \{ \CX_{(i,k),t}^- \,|\, (i,k) \in \vG' (\Bm), \, t \geq 0\} \text{ and } 
\\
& \{ \CI_{(j,l), t}^{\pm}, \, \CK_{(j,l)}^{\pm} \,|\, (j,l) \in \vG(\Bm), \, t \geq 0\}
\end{align*}
respectively. 
Then, we have the following triangular decomposition of $\CU$ from the relations 
(R1)-(R8), \eqref{CI CX+} and \eqref{CI CX-}. 

% ######

\begin{prop}
We have 
\begin{align}
\label{tri decor U}
\CU = \CU^- \CU^0 \CU^+.
\end{align} 
\end{prop}

\remark 
We conjecture that 
the multiplication map 
$\CU^- \otimes_{\KK} \CU^0 \otimes_{\KK} \CU^+ \ra \CU$ ($x \otimes y \otimes z \mapsto xyz$) 
gives an isomorphism as vector spaces. 
More precisely, we expect the existence of a PBW type basis of $\CU$ 
(cf. Proposition \ref{Prop basis g} and \eqref{Hom gQ(m) UqQ(m)} 
with Remark \ref{Remark Hom gQ(m) UqQ(m) iso}).

\para 
We have some relations between the algebra $\CU$ and a quantum group 
associated with the general linear Lie algebra as follows. 

Let $U_q(\Fgl_m)$ be the quantum group associated with the general linear Lie algebra $\Fgl_m$ over $\KK$. 
Namely, $U_q(\Fgl_m)$ is an associative algebra over $\KK$ generated by 
$e_i, f_i$ ($1 \leq i \leq m-1$) and $K_j^{\pm}$ ($1 \leq j \leq m$) with the following defining relations: 
\begin{align*}
& \tag{Q1} 
K_i^+ K_j^+ = K_j^+ K_i^+, \quad K_i^+ K_i^- = K_i^- K_i^+ =1, 
\\
& \tag{Q2} 
K_j^+ e_i K_j^- = q^{a_{ij}} e_i, \quad K_j^+ f_i K_j^- = q^{- a_{ij}} f_i, 
\text{ where } a_{ij} = \lan \a_i, h_j \ran, 
\\
& \tag{Q3} 
e_i f_j - f_j e_i = \d_{i,j} \frac{K_i^+ K_{i+1}^- - K_i^- K_{i+1}^+}{q-q^{-1}}, 
\\
& \tag{Q4} 
e_{i \pm 1} e_i^2 - (q+q^{-1}) e_i e_{i \pm 1} e_i + e_i^2 e_{i \pm 1} =0, 
\quad e_i e_j = e_j e_i \, (|i-j|\geq 2), 
\\
& \tag{Q5} 
f_{i \pm 1} f_i^2 - (q+q^{-1}) f_i f_{i \pm 1} f_i + f_i^2 f_{i \pm 1} =0, 
\quad f_i f_j = f_j f_i \, (|i-j|\geq 2). 
\end{align*}

Let 
$U_q(\Fgl_{m_1} \oplus \dots \oplus \Fgl_{m_r}) \cong U_q(\Fgl_{m_1}) \otimes \dots \otimes U_q (\Fgl_{m_r})$ 
be the Levi subalgebra of $U_q(\Fgl_m)$ 
associated with the Levi subalgebra $\Fgl_{m_1} \oplus \dots \oplus \Fgl_{m_r}$ of $\Fgl_m$. 
Then generators of $U_q(\Fgl_{m_1} \oplus \dots \oplus \Fgl_{m_r})$ are given by 
$e_{(i,k)}, f_{(i,k)}$ $(1 \leq i \leq m_k-1, 1 \leq k \leq r)$ and $K^{\pm}_{(j,l)}$ ($(j,l) \in \vG(\Bm)$), 
where we use the identification \eqref{identify vG 1 to m} for indices.   

% #####

\begin{prop}\
\label{Proposition relation with quantum groups}
\begin{enumerate}
\item 
There exits a surjective homomorphism of algebras 
\begin{align}
\label{surjection UqQm to Uqglm}
g : \CU_{q,\BQ}(\Bm)  \ra U_q(\Fgl_m)
\end{align} 
such that 
\begin{align*} 
&g (\CX_{(i,k),0}^+) = \begin{cases} e_{(i,k)} & \text{ if } i \not= m_k, \\ - Q_k e_{(m_k,k)} & \text{ if } i=m_k, \end{cases}  
\quad 
g (\CX_{(i,k),0}^-) = f_{(i,k)}, 
\\
&g (\CK_{(j,l)}^{\pm}) = K_{(j,l)}^{\pm} 
\text{ and } 
g(\CX_{(i,k),t}^{\pm})=g(\CI^{\pm}_{(j,l),t}) =0 \text{ for } t \geq 1.
\end{align*}

\item 
There exists an injective homomorphism of algebras 
\begin{align} 
\label{injection Levi to U}
\iota : U_q(\Fgl_{m_1} \oplus \dots \oplus \Fgl_{m_r}) \ra \CU_{q,\BQ}(\Bm)
\end{align} such that 
$\iota(e_{(i,k)})= \CX_{(i,k),0}^+$, $\iota (f_{(i,k)}) = \CX_{(i,k),0}^-$ and $\iota (K_{(j,l)}^{\pm}) = \CK_{(j,l)}^{\pm}$. 
\end{enumerate}
\end{prop}

\begin{proof}
We can check the well-definedness of $g$ and $\iota$ by direct calculations. 
Clearly $g$ is surjective. 
Let $ \iota' : U_q(\Fgl_{m_1} \oplus \dots \oplus \Fgl_{m_r}) \ra U_q(\Fgl_m)$ be the natural embedding. 
Then, by investigating the image of generators, 
we see  that $\iota' = g \circ \iota$. 
This implies that $\iota$ is injective. 
\end{proof}

%%%%%
\remark 
\label{Remark evaluation}
The surjective homomorphism $g$ in \eqref{surjection UqQm to Uqglm} can be regarded as a special case 
of evaluation homomorphisms. 
However, we can not define evaluation homomorphisms for $\CU_{q,\BQ}(\Bm)$ in general 
although we can consider $\CU_{q,\BQ}(\Bm)$-modules corresponding to some evaluation modules. 
They will be studied in a subsequent paper.

%%%%%%
\para 
\label{CUAqQstar}
Let $\CU^\star_{\AA}= \CU^\star_{\AA, q,\BQ}(\Bm)$ 
be the $\AA$-subalgebra of $\CU_{q,\BQ}(\Bm)$ generated by 
\begin{align*} 
\{ \CX_{(i,k),t}^{\pm}, \, \CI_{(j,l),t}^{\pm}, \, \CK_{(j,l)}^{\pm} \,|\,  
(i,k) \in \vG'(\Bm), \, (j,l) \in \vG(\Bm), \, t \geq 0\}. 
\end{align*} 
Then, $\CU_{\AA}^\star$ is an associative algebra over $\AA$ 
generated by the same generators with the defining relations (R1)-(R8). 
We regard $\QQ(\BQ)$ as an $\AA$-module through the ring homomorphism 
$\AA \ra \QQ(\BQ)$ ($q \mapsto 1$), 
and we consider the specialization $\QQ(\BQ) \otimes_{\AA} \CU_{\AA}^\star$ 
using this ring homomorphism.  
Let $\mathfrak{J}$ be the ideal of $\QQ(\BQ) \otimes_{\AA} \CU^\star_{\AA}$ 
generated by 
\begin{align}
\label{ideal J}
\{\CK_{(j,l)}^+ -1, \,  \CI_{(j,l),t}^+ - \CI_{(j,l),t}^- \,|\, (i,l) \in \vG(\Bm), \,  t \geq 0 \}.
\end{align} 

Let  $U(\Fg_{\BQ}(\Bm))$ be the universal enveloping algebra of 
the Lie algebra $\Fg_{\BQ}(\Bm)$ 
defined in Definition \ref{Def gQ(m)}. 
Then we can check that 
there exists a surjective homomorphism of algebras 
\begin{align}
\label{Hom gQ(m) UqQ(m)}
U(\Fg_{\BQ}(\Bm)) \ra \QQ(\BQ) \otimes_{\AA} \CU^\star_{\AA, q, \BQ}(\Bm) / \mathfrak{J} 
\end{align} 
such that $\CX_{(i,k),t}^{\pm} \mapsto \CX_{(i,k),t}^{\pm}$ and 
$\CI_{(j,l),t} \mapsto \CI_{(j,l),t}^+ (=\CI_{(j,l),t}^-)$. 

\remark 
\label{Remark Hom gQ(m) UqQ(m) iso}
We conjecture that the homomorphism \eqref{Hom gQ(m) UqQ(m)} is isomorphic. 
Then we may regard $\CU_{q,\BQ}(\Bm)$ as a $q$-analogue of $U(\Fg_{\BQ}(\Bm))$. 

We also remark that 
we have $(\CK_{(j,l)}^+)^2 =1$ in $\CU^\star_{\AA}$ by the relation (R1). 
On the other hand, 
there exists an algebra automorphism of $\CU$ such that 
$\CK_{(j,l)}^{\pm} \mapsto - \CK_{(j,l)}^{\pm}$ and the other generators send to the same generators. 
Thus, the choice of  signs for $\CK_{(j,l)}^+$ in \eqref{ideal J}
will not cause any troubles. 

%%%%%%%%%%%%%%%%%%%%%%%%%%

\para 
\label{A-form}
The final of this section, 
we define the $\AA$-form of $\CU$ taking the divided powers. 

For $(i,k) \in \vG'(\Bm)$ and $t, d \in \ZZ_{\geq 0}$, 
put 
\begin{align*}
\CX_{(i,k),t}^{\pm (d)} = \frac{(\CX_{(i,k),t}^{\pm})^d}{[d]!} \in \CU. 
\end{align*}
For $(j,l) \in \vG(\Bm)$ and $d \in \ZZ_{\geq 0}$, put 
\begin{align*}
\left[ \begin{matrix} \CK_{(j,l)} ; 0 \\ d \end{matrix} \right] 
= \prod_{b=1}^d \frac{\CK_{(j,l)}^+  q^{ - b +1} - \CK_{(j,l)}^{-} q^{b-1}}{q^b - q^{-b}} \in \CU. 
\end{align*}

Let 
$\CU_{\AA} = \CU_{\AA,q,\BQ}(\Bm)$ be the $\AA$-subalgebra of $\CU$ 
generated by all 
$\CX_{(i,k),t}^{\pm (d)}$, $\CI_{(j,l),t}^{\pm}$, $\CK_{(j,l)}^{\pm}$ 
and $\left[ \begin{smallmatrix} \CK_{(j,l)} ; 0 \\ d \end{smallmatrix} \right] $.

%%%%%%%%%%%%%%%%%%%%%%%%%%%%%%%%%%%%%%%%%%%%%%%%%%%%%%%%%%%%%%%

%%%%%%%%%%%%%%%%%%%%%%%%%%%%%%%%%%%%%%%%%%%%%%%%%%%%%%%%%%%%%%%
\section{Representations of $\CU_{q,\BQ}(\Bm)$} 
\label{Section Rep CUqQ}
Thanks to the triangular decomposition \eqref{tri decor U} of $\CU=\CU_{q,\BQ}(\Bm)$, 
we can develop the weight theory to study $\CU$-modules in the usual manner as follows. 

\para 
\textbf{Highest weight modules.} 
For $\la \in P$ and a multiset $\Bvf=(\vf_{(j,l),t}^{\pm}  \,|\, (j,l) \in \vG(\Bm), t \geq 1)$ 
($\vf_{(j,l),t}^{\pm} \in \KK$), 
we say that a $\CU$-module $M$ is a highest weight module 
of highest weight $(\la, \Bvf)$ 
if there exists an element $v_0 \in M$ satisfying the following three conditions: 
\begin{enumerate}
\item 
$M$ is generated by $v_0$ as a $\CU$-module,

\item 
$\CX_{(i,k),t}^+ \cdot v_0 =0$ for all $(i,k) \in \vG'(\Bm)$ and $t \geq 0$, 

\item 
$\CK^+_{(j,l)} \cdot v_0 = q^{\lan \la, h_{(j,l)} \ran} v_0$ 
and 
$\CI_{(j,l),t}^{\pm} \cdot v_0 = \vf_{(j,l),t}^{\pm} v_0$ 
for $(j,l) \in \vG(\Bm)$ and $t \geq 1$. 
\end{enumerate}
If an element $v_0 \in M$ satisfies the above conditions (\roii) and (\roiii), 
we say that $v_0$ is a maximal vector of weight $(\la, \Bvf)$. 
In this case, the submodule $\CU \cdot v_0$  of $ M$ is a highest weight module of highest weight $(\la, \Bvf)$. 
If a maximal vector $v_0 \in M$ satisfies also the above condition (\roi), 
we say that $v_0$ is a highest weight vector. 

If $v_0 \in M$ is a maximal vector of weight $(\la, \Bvf)$, 
for $(j,l) \in \vG(\Bm)$, 
we have 
\begin{align*}
\CI_{(j,l),0}^{\pm} \cdot v = q^{\mp \la_{(j,l)}} [\la_{(j,l)}] v , \text{ where } \la_{(j,l)} = \lan \la, h_{(j,l)} \ran
\end{align*}
by the relation (R1). 

For a highest weight $\CU$-module $M$ of highest weight $(\la, \Bvf)$ 
with a highest weight vector $v_0 \in M$, 
we have $M = \CU^- \cdot v_0$ by the triangular decomposition \eqref{tri decor U}. 
Thus, the relation (R3) implies the weight space decomposition 
\begin{align}
\label{weight sp decor}
M = \bigoplus_{\mu \in P \atop \mu \leq \la} M_{\mu} 
\text{ such that }  
\dim_{\KK} M_{\la} =1, 
\end{align} 
where 
$M_{\mu} = \{ v \in M \,|\, \CK_{(j,l)}^+ \cdot v = q^{\lan \mu, h_{(j,l)} \ran} v \text{ for } (j,l) \in \vG(\Bm) \}$. 

%%%%%%

\para 
\textbf{Verma modules.} 
Let $\wt{\CU}^0$ be the associative algebra over $\KK$ generated by 
$\CI_{(j,l),t}^{\pm}$ and $\CK_{(j,l)}^{\pm}$ for all $(j,l) \in \vG(\Bm)$ and $t \geq 0$ 
with the defining relations (R1) and (R2). 
We also define the associative algebra $\wt{\CU}^{\geq 0}$ generated by 
$\CX_{(i,k),t}^+$, $\CI_{(j,l),t}^{\pm}$ and $\CK_{(j,l)}^{\pm}$ 
for all $(i,k) \in \vG'(\Bm)$, $(j,l) \in \vG(\Bm)$ and $t \geq 0$ 
with the defining relations (R1)-(R8) except (R6). 
Then we have the homomorphism of algebras 
\begin{align} 
\label{hom wt U geq 0 to U}
\wt{\CU}^{\geq 0} \ra \CU 
\text{ such that } \CX_{(i,k),t}^+ \mapsto \CX_{(i,k),t}^+, \CI_{(j,l),t}^{\pm} \mapsto \CI_{(j,l),t}^{\pm}, 
\end{align} 
and  the surjective homomorphism of algebras 
\begin{align} 
\label{surj wt U geq 0 to wt U0}
\wt{\CU}^{\geq 0} \ra \wt{\CU}^0 
\text{ such that }  \CX_{(i,k),t}^+ \mapsto 0, \CI_{(j,l)}^{\pm} \mapsto \CI_{(j,l),t}^{\pm}, 
\CK_{(j,l)}^{\pm} \mapsto \CK_{(j,l)}^{\pm}.
\end{align} 

For $\la \in P$ and a multiset $\Bvf =(\vf_{(j,l),t}^{\pm})$,  
we define a ($1$-dimensional) simple $\wt{\CU}^0$-module 
$\Theta_{(\la, \Bvf)} = \KK v_0$ by 
\begin{align*}
\CK_{(j,l)}^+ \cdot v_0 = q^{\lan \la, h_{(j,l)} \ran} v_0, \quad 
\CI_{(j,l),t}^{\pm} \cdot v_0 = \vf_{(j,l),t}^{\pm} v_0 
\end{align*}
for $(j,l) \in \vG(\Bm)$ and $t \geq 1$. 
Then we define the Verma module $M(\la, \Bvf)$ as the induced module 
\begin{align*}
M(\la, \Bvf) = \CU \otimes_{\wt{\CU}^{\geq 0}} \Theta_{(\la, \Bvf)}, 
\end{align*}
where we regard $\Theta_{(\la, \Bvf)}$ (resp. $\CU$) 
as a left (resp. right) $\wt{\CU}^{\geq 0}$-module through the homomorphism 
\eqref{surj wt U geq 0 to wt U0}  (resp. \eqref{hom wt U geq 0 to U}). 

By definitions, 
the Verma module $M(\la, \Bvf)$ is a highest weight module of highest weight $(\la, \Bvf)$ 
with a highest weight vector $1 \otimes v_0$.  
Then  we see that 
any highest weight module of highest weight $(\la, \Bvf)$ is a quotient of $M(\la, \Bvf)$ 
by the universality of tensor products. 
We also see that 
$M(\la, \Bvf)$ has the unique simple top $L(\la,\Bvf) = M(\la, \Bvf)/ \rad M(\la,\Bvf)$ 
from the weight space decomposition \eqref{weight sp decor}. 

By using the homomorphism $\iota : U_q(\Fgl_{m_1} \oplus \dots \oplus\Fgl_{m_r}) \ra \CU$ 
in \eqref{injection Levi to U}, 
we have the following necessary condition for $L(\la,\Bvf)$ to be finite dimensional  
in a similar way as in the proof of Proposition \ref{Prop necessary condition fin. simpe}. 

\begin{prop}
For $\la \in P$ and a multiset $\Bvf=(\vf_{(j,l),t}^{\pm})$, 
if $L(\la, \Bvf)$ is finite dimensional, 
then we have 
$\la \in P^+_{\Bm}$. 
\end{prop}

%%%%%%%% 
\para 
\textbf{Category $\ZC_{q,\BQ}(\Bm)$.} 
Let $\ZC_{q,\BQ}(\Bm)$ 
(resp. $ \ZC_{q, \BQ}^{\geq 0}(\Bm)$) 
be the full subcategory of $\CU \cmod$ consisting of 
$\CU$-modules satisfying the following conditions: 
\begin{enumerate}
\item 
If $M \in \ZC_{q,\BQ}(\Bm)$ (resp. $M \in \ZC_{q,\BQ}^{\geq 0}(\Bm)$), then $M$ is finite dimensional, 
\item 
If $M \in \ZC_{q,\BQ}(\Bm)$ (resp. $M \in \ZC_{q,\BQ}^{\geq 0}(\Bm)$), 
then $M$ has the weight space decomposition 
\begin{align*}
M = \bigoplus_{\la \in P} M_{\la} 
\quad 
(\text{resp. } M= \bigoplus_{ \la \in P_{\geq 0}} M_{\la}), 
\end{align*}
where $M_{\la} =\{ v \in M \,|\, \CK_{(j,l)}^+ \cdot m = q^{\lan \la, h_{(j,l)} \ran} v$ for $(j,l) \in \vG(\Bm)\}$, 
\item 
If $M \in \ZC_{q,\BQ}(\Bm)$ (resp. $M \in \ZC_{q,\BQ}^{\geq 0}(\Bm)$), then 
all eigenvalues of the action of $\CI_{(j,l),t}^{\pm}$ ($(j,l) \in \vG(\Bm), t \geq 0$) on $M$ 
belong to $\KK$. 
\end{enumerate}

By the usual argument, we have the following lemma. 

\begin{lem} 
Any simple object in $\ZC_{q,\BQ}(\Bm)$ is a highest weight module. 
\end{lem}

%%%%%%%
By using the surjection $g : \CU_{q,\BQ}(\Bm) \ra U_q(\Fgl_m)$ in \eqref{surjection UqQm to Uqglm}, 
we have the following proposition. 
\begin{prop}
\label{Prop ZCUqgl subset ZCqQ}
Let $\ZC_{U_q(\Fgl_m)}$ be the category of finite dimensional $U_q(\Fgl_m)$-modules which have the weight space decomposition. 
Then we have the followings. 
\begin{enumerate}
\item 
$\ZC_{U_q(\Fgl_m)}$ is a full subcategory of $\ZC_{q,\BQ}(\Bm)$ through the surjection \eqref{surjection UqQm to Uqglm}. 

\item 
For $\la \in P^+$, 
the simple highest weight $U_q(\Fgl_m)$-module $\D_{U_q(\Fgl_m)}(\la)$ of highest weight $\la$ 
is the simple  highest weight $\CU$-module of highest weight $(\la,\mathbf{0})$ 
through the surjection \eqref{surjection UqQm to Uqglm}, 
where 
$\mathbf{0}$ means $\vf_{(j,l),t}^{\pm}=0$ for all $(j,l) \in \vG(\Bm)$ and $t \geq 1$. 
\end{enumerate}
\end{prop}
%%%%%%%%%%%%%%%%%%%%%%%%%%%%%%%%%%%%%%%%%%%%%%%%%%%%%%%%%%%%%%%

%%%%%%%%%%%%%%%%%%%%%%%%%%%%%%%%%%%%%%%%%%%%%%%%%%%%%%%%%%%%%%%

\section{Review of cyclotomic $q$-Schur algebras} 
\label{Review CSA}

In this section, we recall the definition and some fundamental properties of the cyclotomic $q$-Schur algebra 
$\Sc_{n,r}(\Bm)$ introduced in \cite{DJM98}. See \cite{DJM98} and \cite{Mat04} for details. 

\para 
Let $R$ be a commutative ring, 
and we take parameters $q, Q_0,Q_1, \dots, Q_{r-1} \in R$ such that $q$ is invertible in $R$. 
The Ariki-Koike algebra $\He_{n,r}$ associated with the complex reflection group 
$\FS_n \ltimes (\ZZ/ r \ZZ)^n$ 
is the associative algebra with $1$ over $R$ generated by 
$T_0, T_1, \dots, T_{n-1}$ with the following defining relations: 
\begin{align*}
& (T_0 - Q_0)(T_0-Q_1) \dots (T_0-Q_{r-1})=0, 
\quad (T_i -q )(T_i+q^{-1}) =0 \quad  ( 1 \leq i \leq n-1),
\\
& T_0 T_1 T_0 T_1=T_1 T_0 T_1 T_0, 
\quad T_i T_{i+1} T_i = T_{i+1} T_i T_{i+1} \quad  ( 1 \leq i \leq n-2), 
\\
& T_i T_j = T_j T_i  \quad  (|i-j| \geq 2). 
\end{align*}

The subalgebra of $\He_{n,r}$ generated by $T_1, \dots, T_{n-1}$ 
is isomorphic to the Iwahori-Hecke algebra $\He_{n}$ associated with the symmetric group $\FS_n$ of degree $n$. 
For $w \in \FS_n$, 
we denote by $\ell(w)$ the length of $w$, 
and denote by $T_w$ the standard basis of $\He_{n}$ corresponding to $w$. 

% #####

\para 
Put $L_1 =T_0$ and $L_i=T_{i-1} L_{i-1} T_{i-1}$ for $i=2, \dots, n$. 
These elements $L_1, \dots, L_n$ are called Jucys-Murphy elements of $\He_{n,r}$ 
(see \cite{Mat08} for properties of Jucys-Murphy elements). 
The following lemma is well-known, 
and one can easily check them from defining relations of $\He_{n,r}$. 

\begin{lem}
\label{Lemma commute L T}
We have the following. 
\begin{enumerate}
\item 
	$L_i$ and $L_j$ commute with each other for any $1 \leq i,j \leq n$. 

\item 
	$T_i$ and $L_j$ commute with each other if $j \not= i, i+1$. 

\item 
	$T_i$ commutes with both $L_i L_{i+1}$ and $L_i + L_{i+1}$ for any $ 1 \leq i \leq n-1$. 

\item 
	$L_{i+1}^t T_i =(q-q^{-1}) \sum_{s=0}^{t-1} L_{i+1}^{t-s} L_i^s + T_i L_i^t $ 
	for any $ 1 \leq i \leq n-1$ and $t \geq 1$. 

\item 
	$L_i^t T_i = - (q-q^{-1}) \sum_{s=1}^t   L_i^{t-s} L_{i+1}^s + T_i L_{i+1}^t $ 
	for any $1 \leq i \leq n-1$ and $t \geq 1$. 
\end{enumerate}
\end{lem}

% #####

\para 
Let $\Bm=(m_1, \dots, m_r) \in \ZZ_{>0}^r$ 
be an $r$-tuple of  positive integers. 
Put 
\begin{align*}
\vL_{n,r}(\Bm) = 
	\left\{ \mu=(\mu^{(1)}, \mu^{(2)}, \dots, \mu^{(r)}) \Bigg| 
		\begin{matrix} 
			\mu^{(k)} = (\mu_1^{(k)}, \dots, \mu_{m_k}^{(k)}) \in \ZZ_{\geq 0}^{m_k} 
			\\
			\sum_{k=1}^r \sum_{i=1}^{m_k} \mu_i^{(k)} =n 
		\end{matrix} 
	\right\}. 
\end{align*}
We also put 
\begin{align*}
\vL_{n,r}^+(\Bm) = \{ \mu \in \vL_{n,r}(\Bm) \,|\, \mu_1^{(k)} \geq \mu_2^{(k)} \geq \dots \geq \mu_{m_k}^{(k)} \geq  0 \text{ for each }k=1,\dots,r \}.
\end{align*}
We regard $\vL_{n,r}(\Bm)$ as a subset of weight lattice $P= \bigoplus_{(i,k) \in \vG(\Bm)} \ZZ \ve_{(i,k)}$  
by the injection 
$\vL_{n,r}(\Bm) \ra P  $
such that 
$\mu \mapsto \sum_{(i,k)\in \vG(\Bm)} \mu_i^{(k)} \ve_{(i,k)}$.  
Then we see that 
$\vL^+_{n,r}(\Bm) = \vL_{n,r}(\Bm) \cap P^+_{\Bm}$. 

For $\mu \in \vL_{n,r}(\Bm)$, put 
\begin{align}
\label{Def m mu}
m_{\mu} = \Big( \sum_{w \in \FS_{\mu}} q^{\ell(w)} T_w \Big) 
		\Big( \prod_{k=1}^{r-1} \prod_{i=1}^{a_k} (L_i - Q_k) \Big), 
\end{align}
where $\FS_{\mu}$ is the Young subgroup of $\FS_n$ with respect to $\mu$,  
and $a_k= \sum_{j=1}^{k} |\mu^{(j)}|$. 
The following fact is well known: 
\begin{align}
\label{ m mu T}
	m_{\mu} T_w = q^{\ell(w)} m_{\mu} \text{ if } w \in \FS_{\mu}. 
\end{align}

The cyclotomic $q$-Schur algebra $\Sc_{n,r}(\Bm)$ associated with $\He_{n,r}$ is defined by 
\begin{align}
\Sc_{n,r}(\Bm) = \End_{\He_{n,r}} \Big( \bigoplus_{\mu \in \vL_{n,r}(\Bm)} m_{\mu} \He_{n,r} \Big). 
\end{align} 
For convenience in the later arguments, put 
$m_{\mu} =0 \text{ for } \mu \in P \setminus \vL_{n,r}(\Bm)$. 

%%%%%%%%

\para 
Put 
$\wt{\vL}_{n,r}^+(\Bm) = \vL_{n,r}^+((n, \dots, n, m_r))$. 
It is clear that 
$\wt\vL_{n,r}^+(\Bm) = \vL_{n,r}^+(\Bm)$ 
if $m_k \geq n$ for all $k=1,\dots, r-1$. 
In the case where $m_k < n$ for some $k <r$, 
$\vL_{n,r}^+(\Bm)$ is a proper subset of $\wt\vL_{n,r}^+(\Bm)$. 

In \cite{DJM98} 
(see also \cite{Mat04} for the case where $m_k <n$ for some $k$), 
it is proven that 
$\Sc_{n,r}(\Bm)$ is a cellular algebra with respect to the poset $(\wt\vL_{n,r}^+, \geq)$. 
For $\la \in \wt\vL^+_{n,r} (\Bm)$, 
let $\D(\la)$ be the Weyl (cell) module corresponding to $\la$ constructed in \cite{DJM98} 
(see also \cite{Mat04} and \cite[Lemma 1.18]{W-3}). 
By the general theory of cellular algebras given in \cite{GL},  
$\{ \D(\la) \,|\, \la \in \wt{\vL}_{n,r}^+(\Bm)\}$ gives a complete set of isomorphism classes of simple $\Sc_{n,r}(\Bm)$-modules 
if $\Sc_{n,r}(\Bm)$ is semi-simple. 
It is also proven, in \cite{DJM98}, that 
$\Sc_{n,r}(\Bm)$ is a quasi-hereditary algebra 
such that $\{\D(\la) \,|\, \la \in \vL_{n,r}^+(\Bm)\}$ gives a complete set of standard modules 
if $R$ is a field and $m _k \geq n$ for all $k=1,\dots, r-1$. 

From the construction of $\D(\la)$ in \cite{DJM98}, 
$\D(\la)$ has a basis indexed by the set of  semi-standard tableaux. 
Since we use them in the later argument, 
we recall the definition of semi-standard tableaux 
from \cite{DJM98}. 

For $\la \in \wt\vL_{n,r}^+(\Bm)$, 
the diagram $[\la]$ of $\la$ is 
the set 
\[
[\la]= 
\{(i,j,k) \in \ZZ^3 \,|\, 1 \leq i \leq m_k, \, 1 \leq j \leq \la_i^{(k)}, \, 1 \leq k \leq r\}.
\] 
For $x=(i,j,k) \in [\la]$, put 
\[
\res (x) = q^{2(j-i)} Q_{k-1}. 
\]
For $\la \in \wt\vL_{n,r}^+(\Bm)$ and $\mu \in \vL_{n,r}(\Bm)$, 
a tableau of shape $\la$ with weight $\mu$ 
is a map 
\[
T :\nobreak [\la] \ra \{(a,c) \in \ZZ \times \ZZ \,|\, a \geq 1, 1 \leq c \leq r\}
\] 
such that 
$\mu_i^{(k)} = \sharp \{x \in [\la] \,|\, T(x) = (i,k)\}$. 
We define the order on $\ZZ \times \ZZ$ 
by 
$(a,c) \geq (a',c')$ if either $c > c'$, or $c=c'$ and $a \geq a'$.
For a tableau $T$ of shape $\la$ with weight $\mu$, 
we say that $T$ is  semi-standard if $T$ satisfies the following conditions: 
\begin{enumerate}
\item If $T((i,j,k)) =(a,c)$, then $k \leq c$,
\item $T((i,j,k)) \leq T((i,j+1,k))$ if $(i,j+1,k) \in [\la]$, 
\item $T((i,j,k)) < T((i+1,j,k))$ if $(i+1,j,k) \in [\la]$. 
\end{enumerate}
For $\la \in \wt\vL_{n,r}^+(\Bm)$, $\mu \in \vL_{n,r}(\Bm)$, 
we denote by 
$\CT_0(\la,\mu)$ 
the set of semi-standard tableaux of shape $\la$ with weight $\mu$. 
Then, from the cellular basis of $\Sc_{n,r}(\Bm)$ in \cite{DJM98}, 
we see that 
$\D(\la)$ has the basis 
\[
\{\vf_T \,|\, T \in \CT_0(\la,\mu) \text{ for some } \mu \in \vL_{n,r}(\Bm)\}. 
\] 
(See \cite{DJM98} for the definition of $\vf_T$.)

%%%%%%%%%%%%%%%%%%%%%%%%%%%%%%%%%%%%%%%%%%%%%%%%%%%%%%%%%%%%%%%

%%%%%%%%%%%%%%%%%%%%%%%%%%%%%%%%%%%%%%%%%%%%%%%%%%%%%%%%%%%%%%%

\section{Generators of cyclotomic $q$-Schur algebras} 

In this section, we define some generators of the cyclotomic $q$-Schur algebra, 
and we obtain some relations among  them 
which will be used  to obtain the homomorphism from $\CU_{q,\BQ}(\Bm)$ 
in the next section. 

\para 
A partition $\la$ is a non-increasing sequence $\la =(\la_1,\la_2,\dots )$ of non-negative integers. 
For a partition $\la=(\la_1, \la_2, \dots)$, we denote by $\ell(\la)$ the length of $\la$ 
which is the maximal integer $l$ such that $\la_l \not=0$. 
If $\sum_{i=1}^{\ell(\la)} \la_i =n$, we denote it by $\la \vdash n$. 
For a integer $k$ and a partition $\la \vdash n$ such that $\ell(\la)\leq k$, 
put 
\begin{align*} 
\FS_k \cdot \la
= \big\{ (\mu_1, \mu_2,\dots, \mu_k) \in \ZZ_{\geq 0}^k  \bigm| 
		\mu_i = \la_{\s(i)}, \, \s  \in \FS_k \big\}. 
\end{align*}

%%%%%%
\para 
For integers $t,k >0$,  
we  define the symmetric polynomials 
$\Phi_t^{\pm}(x_1, \dots,s_k) \in R [x_1, \dots,x_k]^{\FS_k}$ 
of degree $t$ with variables $x_1, \dots, x_k$ as 
\begin{align}
\label{Def Phi}
\Phi_t^{\pm} (x_1, \dots,x_k) 
	= \sum_{\la \vdash t \atop \ell(\la) \leq k} (1 - q^{\mp 2})^{\ell (\la) -1} 
		\Fm_{\la} (x_1, \dots, x_k), 
\end{align}
where 
$\Fm_{\la}(x_1, \dots, x_k) 
	=\sum_{\mu =(\mu_1,\mu_2, \dots, \mu_k) \in \FS_k \cdot \la} x_1^{\mu_1} x_2^{\mu_2} \dots x_k^{\mu_k}$ 
is the monomial symmetric polynomial associated with the partition $\la$.
For convenience, 
we also define 
\begin{align}
\Phi_0^{\pm} (x_1, \dots, x_k) = q^{\mp  k  \pm 1}[k]. 
\end{align}

From the definition, we have 
\begin{align} 
\label{Phi t=1}
\Phi_1^{\pm} (x_1,\dots,x_k) = x_1 + x_2 + \dots + x_k 
\text{ and } 
\Phi_t^{\pm}(x_1)=x_1^t. 
\end{align}
The polynomials $\Phi_{t}^{\pm}(x_1,\dots,x_k)$ satisfy the following recursive relations 
which will be used for calculations of some relations between generators of $\Sc_{n,r}(\Bm)$ in later.  

\begin{lem} 
For $t \geq 0$, we have 
\begin{align}
\label{recursive relation of Phi}
\begin{split}
\Phi_{t+1}^{\pm} (x_1,\dots, x_k) 
&= \sum_{s=1}^k \Phi_t^{\pm}(x_1,\dots,x_s) x_s 
	- q^{\mp 2} \sum_{s=1}^{k-1} \Phi_t^{\pm} (x_1,\dots,x_s) x_{s+1} 
\\
&= x_1^{t+1} 
	+ \sum_{s=2}^k \big( \Phi_t^{\pm} (x_1,\dots,x_s) x_s - q^{\mp 2} \Phi_t^{\pm} (x_1,\dots, x_{s-1})x_{s}\big) 
\end{split}
\end{align}
and 
\begin{align}
\label{recursive relation of Phi 2}
\begin{split}
&\Phi_{t+1}^{\pm} (x_1,x_2, \dots, x_k) - \Phi_{t+1}^{\pm} (x_2,\dots,x_k) 
\\
&= x_1 \big( \Phi_t^{\pm} (x_1,x_2, \dots, x_k) - q^{\mp 2} \Phi_t^{\pm} (x_2,\dots,x_k) \big). 
\end{split}
\end{align}
\end{lem} 
\begin{proof}
In the case where $t=0$, we can check the statements by direct calculations. 

Assume that $t \geq 1$. 
From the definition, 
we have 
\begin{align*}
\Phi_{t+1}^{\pm} (x_1,\dots, x_k) 
&= \sum_{\la \vdash t+1 \atop \ell(\la) \leq k} (1- q^{\mp 2})^{\ell (\la) -1} 
	\sum_{\mu \in \FS_k \la} x_1^{\mu_1} x_2^{\mu_2} \dots x_k^{\mu_k} 
\\
&= \sum_{s=1}^k \sum_{\la \vdash t+1 \atop \ell (\la) \leq s} (1- q^{\mp 2})^{\ell (\la) -1} 
	\sum_{\mu \in \FS_s \la \atop \mu_s \not= 0} x_1^{\mu_1} x_2^{\mu_2} \dots x_s^{\mu_s} 
\\
&= \sum_{s=1}^k \sum_{\la \vdash t+1 \atop \ell (\la) \leq s} (1- q^{\mp 2})^{\ell (\la) -1} 
	\sum_{\mu \in \FS_s \la \atop \mu_s =1} x_1^{\mu_1} x_2^{\mu_2} \dots x_s^{\mu_s} 
	\\ & \hspace{1em} 
	+ \sum_{s=1}^k \sum_{\la \vdash t+1 \atop \ell (\la) \leq s} (1- q^{\mp 2})^{\ell (\la) -1} 
	\sum_{\mu \in \FS_s \la \atop \mu_s \geq 2} x_1^{\mu_1} x_2^{\mu_2} \dots x_s^{\mu_s} 
\\
&= \sum_{s=1}^k \sum_{\la \vdash t \atop \ell (\la) \leq s} (1- q^{\mp 2})^{\ell (\la) } 
	\sum_{\mu \in \FS_s \la \atop \mu_s = 0} x_1^{\mu_1} x_2^{\mu_2} \dots x_{s-1}^{\mu_{s-1}} x_s  
	\\ & \hspace{1em} 
	+ \sum_{s=1}^k \sum_{\la \vdash t \atop \ell (\la) \leq s} (1- q^{\mp 2})^{\ell (\la) -1} 
	\sum_{\mu \in \FS_s \la \atop \mu_s \not=0} x_1^{\mu_1} x_2^{\mu_2} \dots x_s^{\mu_s} x_s 
\\
&= \sum_{s=1}^k 
	\big( \sum_{\la \vdash t \atop \ell (\la) \leq s} (1-q^{\mp 2})^{\ell (\la) -1} 
	\sum_{\mu \in \FS_s \la} x_1^{\mu_1} x_2^{\mu_2} \dots x_s^{\mu_s}
	\big)  x_s 
	\\
	& \hspace{1em} 
	- q^{\mp 2} \sum_{s=2}^k 
		\big( \sum_{\la \vdash t \atop \ell (\la) \leq s} (1-q^{\mp2})^{\ell (\la) -1} 
		\sum_{\mu \in \FS_s \la \atop \mu_s=0} x_1^{\mu_1} x_2^{\mu_2} \dots x_{s-1}^{\mu_{s-1}} 
		\big)  x_s 
\\
&= \sum_{s=1}^k \Phi_t^{\pm}(x_1,\dots,x_s) x_s 
	- q^{\mp 2} \sum_{s=1}^{k-1} \Phi_t^{\pm} (x_1,\dots,x_s) x_{s+1}.  
\end{align*}
We can easily check the second equality of \eqref{recursive relation of Phi}. 

We prove \eqref{recursive relation of Phi 2} by the induction on $t$. 
In the case where $t=1$, 
we can check \eqref{recursive relation of Phi 2} directly by using the relation \eqref{recursive relation of Phi} 
together with \eqref{Phi t=1}. 
Assume that $t >1$. 
By \eqref{recursive relation of Phi}, we have 
\begin{align*}
&\Phi_{t+1}^{\pm} (x_1,x_2, \dots, x_k) - \Phi_{t+1}^{\pm} (x_2, \dots, x_k)  
\\
&= 
	\big( \sum_{s=1}^k \Phi_t^{\pm} (x_1,\dots, x_s) x_s 
		- q^{\mp 2} \sum_{s=1}^{k-1} \Phi_t^{\pm} (x_1, \dots, x_s) x_{s+1} \big) 
	\\ & \hspace{1em} 
	- \big( \sum_{s=2}^k \Phi_t^{\pm} (x_2,\dots, x_s) x_s 
		- q^{\mp 2} \sum_{s=2}^{k-1} \Phi_t^{\pm} (x_2, \dots, x_s) x_{s+1} \big) 
\\
&= \Phi_t^{\pm}(x_1)x_1 - q^{\mp 2} \Phi_t^{\pm} (x_1) x_2 
	+ \sum_{s=2}^{k} \big( \Phi_t^{\pm} (x_1,\dots, x_s) - \Phi_t^{\pm} (x_2,\dots, x_s) \big) x_s 
	\\
	& \hspace{1em} 
	- q^{\mp 2} \sum_{s=2}^{k-1}  \big( \Phi_t^{\pm} (x_1,\dots, x_s) - \Phi_t^{\pm} (x_2,\dots, x_s) \big) x_{s+1}. 
\end{align*}
Applying the assumption of the induction, 
we have 
\begin{align*}
&\Phi_{t+1}^{\pm} (x_1,x_2, \dots, x_k) - \Phi_{t+1}^{\pm} (x_2, \dots, x_k)  
\\
&= x_1 \Phi_{t-1}^{\pm}(x_1) x_1 - q^{\mp 2} x_1 \Phi_{t-1}^{\pm}(x_1)x_2 
	\\ & \hspace{1em} 
	+ \sum_{s=2}^k x_1 \big( \Phi_{t-1}^{\pm}(x_1,x_2, \dots, x_s) - q^{\mp 2} \Phi_{t-1}^{\pm} (x_2, \dots, x_s) \big) x_s 
	\\ & \hspace{1em} 
	- q^{\mp 2} \sum_{s=2}^{k-1} 
		x_1 \big( \Phi_{t-1}^{\pm}(x_1,x_2, \dots, x_s) - q^{\mp 2} \Phi_{t-1}^{\pm} (x_2, \dots, x_s) \big) x_{s+1}
\\
&= x_1 \big\{ 
	\big( \sum_{s=1}^k \Phi_{t-1}^{\pm} (x_1,x_2, \dots, x_s) x_s 
		- q^{\mp 2} \sum_{s=1}^{k-1} \Phi_{t-1}^{\pm} (x_1,x_2, \dots, x_s) x_{s+1} \big) 
	\\ & \hspace{3em} 
	- q^{\mp 2} \big( \sum_{s=2}^k \Phi_{t-1}^{\pm} (x_2, \dots, x_s) x_s 
		- q^{\mp 2} \sum_{s=2}^{k-1} \Phi_{t-1}^{\pm} (x_2, \dots, x_s)x_{s+1} \big) \big\}. 
\end{align*}
Applying the relation \eqref{recursive relation of Phi},  
we obtain \eqref{recursive relation of Phi 2}. 
\end{proof}

% #####

\remark 
\label{Remark def Phi}
At first, 
the author defined the polynomials $\Phi_t^{\pm} (x_1,\dots,x_k)$ 
by using the relations  \eqref{recursive relation of Phi} inductively. 
The definition of $\Phi_t^{\pm}(x_1,\dots,x_k)$ as in \eqref{Def Phi} was suggested by Tatsuyuki Hikita. 

% #####

\para 
For $\mu \in \vL_{n,r}(\Bm)$ and $(j,l) \in \vG(\Bm)$, 
put 
\begin{align*}
N^{\mu}_{(j,l)} = \sum_{c=1}^{l-1} |\mu^{(c)}| + \sum_{p=1}^j \mu_p^{(l)}.
\end{align*}

For $(j,l) \in \vG(\Bm)$ and an integer $t \geq 0$, 
we define the elements $\CK_{(j,l)}^{\pm}$ and $ \CI_{(j,l),t}^{\pm}$   of $\Sc_{(n,r)}(\Bm)$ by 
\begin{align*}
& \CK_{(j,l)}^{\pm} (m_\mu) = q^{\pm \mu_j^{(l)}} m_{\mu},
\\
& \CI_{(j,l),t}^{+}(m_\mu) 
	=  
	\begin{cases} 
		q^{ t-1} m_{\mu} \Phi_t^{+} (L_{N^{\mu}_{(j,l)}}, L_{N^{\mu}_{(j,l)} -1}, \dots, L_{N^\mu_{(j,l)}- \mu_j^{(l)}+1}) 
		& \text{ if } \mu_j^{(l)} \not=0, 
		\\
		0 & \text{ if } \mu_j^{(l)}=0, 
	\end{cases}
\\
& \CI_{(j,l),t}^{-}(m_\mu) 
	= 
	\begin{cases} 
		 q^{-t+1} m_{\mu} \Phi_t^{-} (L_{N^{\mu}_{(j,l)}}, L_{N^{\mu}_{(j,l)}-1}, \dots, L_{N^{\mu}_{(j,l)}- \mu_j^{(l)}+1}) 
		& \text{ if } \mu_j^{(l)} \not=0, 
		\\
		0 & \text{ if } \mu_j^{(l)}=0, 
	\end{cases}
\end{align*}
for each  $\mu \in \vL_{n,r}(\Bm)$. 

It is clear that the definitions of $\CK_{(j,l)}^{\pm}$ are well-defined. 
For $\mu \in \vL_{n,r}(\Bm)$ and $(j,l) \in \vG(\Bm)$ such that $\mu_j^{(l)} \not=0$,  
we see that 
$\Phi_t^{\pm}(L_{N^{\mu}_{(j,l)}}, L_{N^{\mu}_{(j,l)} -1}, \dots, L_{N^\mu_{(j,l)}- \mu_j^{(l)}+1}) $  
commute with $T_w$ for any $w \in \FS_{\mu}$ by Lemma \ref{Lemma commute L T} 
since 
$\Phi_t^{\pm}(L_{N^{\mu}_{(j,l)}}, \dots, L_{N^\mu_{(j,l)}- \mu_j^{(l)}+1}) $ 
is a symmetric polynomials with variables 
$L_{N^{\mu}_{(j,l)}}, L_{N^{\mu}_{(j,l)} -1}, \dots, L_{N^\mu_{(j,l)}- \mu_j^{(l)}+1}$. 
Thus, 
$\Phi_t^{\pm}(L_{N^{\mu}_{(j,l)}}, \dots, L_{N^\mu_{(j,l)}- \mu_j^{(l)}+1}) $ 
commute with $m_\mu$, 
and the definitions of $\CI_{(j,l),t}^{\pm}$ are well-defined. 

% #####

The following lemma is immediate from definitions. 

\begin{lem}
\label{Lemma commute K H}
For $(i,k), (j,l) \in \vG(\Bm)$ and $s,t \geq 0$, 
we have the following relations. 
\begin{enumerate}
\item 
$\CK_{(j,l)}^+ \CK_{(j,l)}^- = \CK_{(j,l)}^- \CK_{(j,l)}^+ =1$. 

\item 
$[\CK_{(i,k)}^+, \CK_{(j,l)}^+] = [\CK_{(i,k)}^+, \CI_{(j,l),t}^{\s}] = [\CI_{(i,k),s}^{\s}, \CI_{(j,l),t}^{\s'}] =0$ 
($\s,\s' \in \{+, -\}$). 
\end{enumerate}
\end{lem}

We also have the following lemma by direct calculations.  

\begin{lem} 
\label{Lemma CK+- 2}
For $(j,l) \in \vG(\Bm)$, we have 
\begin{align*}
(\CK_{(j,l)}^\pm)^2 = 1 \pm  (q-q^{-1}) \CI_{(j,l),0}^{\mp}. 
\end{align*}
\end{lem}

% ################

\para 
For $(i,k) \in \vG'(\Bm)$ and an integer $t \geq 0$, 
we define the element $\wt{\CK}_{(i,k)}^{\pm}$ and $\CJ_{(j,l),t} $ of $\Sc_{n,r}(\Bm)$ by  
\begin{align*}
\wt{\CK}_{(i,k)}^{\pm} = \CK_{(i,k)}^{\pm} \CK_{(i+1,k)}^{\mp}
\end{align*}
and 
\begin{align*}
\CJ_{(i,k),t} = 
\begin{cases}
	\CI_{(i,k),0}^+ - \CI_{(i+1,k),0}^-  + (q-q^{-1}) \CI_{(i,k),0}^+ \CI_{(i+1,k),0}^- 
	& \text{ if } t=0, 
	\\
	\dis q^{-t} \CI_{(i,k),t}^+ - q^t \CI_{(i+1,k),t}^- 
	- (q-q^{-1}) \sum_{b=1}^{t-1} q^{- t + 2 b} \CI_{(i,k),t-b}^+ \CI_{(i+1,k),b}^- 
	& \text{ if } t >0.
	\end{cases} 
\end{align*} 

By Lemma \ref{Lemma CK+- 2}, we have the following corollary. 

\begin{cor}
\label{Cor CJ0}
For $(i,k) \in \vG'(\Bm)$, we have 
\begin{align*}
\CJ_{(i,k),0} = \CI_{(i,k),0}^+ - (\CK_{(i,k)}^-)^2 \CI_{(i+1,k),0}^-. 
\end{align*}
\end{cor}

% ###########

\para 
For $N \in \ZZ_{\geq 0}$ and $\mu \in \ZZ_{>0}$, 
put 
\begin{align*} 
&[T;N,\mu]^+ = 
	\begin{cases} 
		\dis 1 + \sum_{h=1}^{\mu -1} q^h T_{N+1} T_{N+2} \dots T_{N+h} 
		& \text{ if } N+\mu \leq n, 
		\\
		0 & \text{ otherwise,}
	\end{cases} 
\\
&[T;N,\mu]^- = 
	\begin{cases}
		\dis 1 + \sum_{h=1}^{\mu - 1} q^h T_{N-1} T_{N-2} \dots T_{N-h} 
		& \text{ if } n \geq N \geq \mu, 
		\\
		0 & \text{ otherwise.} 
	\end{cases} 
\end{align*} 
For convenience, we put $[ T; N,0]^{\pm}=0$ for any $N \in \ZZ_{\geq 0}$. 

For  $N, \mu \in \ZZ_{\geq 0}$ and $d \in \ZZ_{>0}$,  
put 
\begin{align*}
&\left[ \begin{smallmatrix} T; N, \mu \\ d \end{smallmatrix} \right]^+ 
= [T; N+ (d-1), \mu - (d -1)]^+ \dots [T; N+1, \mu -1]^+ [T; N, \mu]^+,  
\\
&\left[ \begin{smallmatrix} T; N, \mu \\ d \end{smallmatrix} \right]^- 
= [T; N-(d-1), \mu -(d-1)]^- \dots [T; N-1, \mu -1]^- [T; N, \mu]^-.
\end{align*}
We also put 
$\left[ \begin{smallmatrix} T; N, \mu \\ 0 \end{smallmatrix} \right]^+ 
= \left[ \begin{smallmatrix} T; N, \mu \\ 0 \end{smallmatrix} \right]^- 
=1$ 
for any $N, \mu \in \ZZ_{\geq 0}$. 

For $N \in \ZZ_{\geq 0}$ and $d \in \ZZ_{>0}$, put 
\begin{align*}
& (T; N,d)^+ = 
	\begin{cases} 
	\dis 1 + \sum_{h=1}^{d-1} q^h T_{N+d-h} T_{N+d- (h-1)} \dots T_{N+d-2} T_{N+d-1} 
	& \text{ if } N+d \leq n, 
	\\
	0 & \text{ otherwise,}
	\end{cases}
\\
& (T;N,d)^- = 
	\begin{cases} 
	\dis 1 + \sum_{h=1}^{d-1} q^h T_{N-d +h} T_{N-d +(h-1)} \dots T_{N-d+2} T_{N-d+1} 
	& \text{ if } n \geq N \geq d, 
	\\
	0 & \text{ otherwise.} 
	\end{cases} 
\end{align*}
We also put 
\begin{align*}
(T;N,d)^{\pm} ! =(T;N,d)^{\pm} (T;N,d-1)^{\pm} \dots (T;N,1)^{\pm}. 
\end{align*}
%
%%%%%
%

The following lemma follows from Lemma \ref{Lemma commute L T} immediately. 

\begin{lem}
\label{Lemma Li commute [T; N mu]}
For $N, \mu \in \ZZ_{\geq 0}$, we have the following. 
\begin{enumerate} 
\item 
$L_i$ commute with $[T;N,\mu]^+$ unless $ N + \mu \geq i \geq N+1$. 

\item 
$L_i$ commute with $[T; N,\mu]^-$ unless $N \geq i \geq N - \mu +1$. 
\end{enumerate} 
\end{lem}

%
%%%%%
%

\begin{lem}
\label{Lemma [T;N,c] com rel}
We have the following. 
\begin{enumerate} 
\item  
For $N, \mu \in \ZZ_{\geq 0}$ such that $ N+\mu \leq n$ and $\mu \geq 3$, 
we have 
\begin{align*}
&(q^{\mu -2} T_{N+2} T_{N+3} \dots T_{N+\mu-1}) (q^{\mu-1} T_{N+1} T_{N+2} \dots T_{N+\mu-1}) 
\\
&= (q^{\mu-1} T_{N+1} T_{N+2} \dots T_{N+\mu-1}) (q^{\mu -2} T_{N+1} T_{N+2} \dots T_{N+\mu-2}). 
\end{align*}

\item  
For $N, \mu \in \ZZ_{\geq 0}$  such that $N \geq \mu \geq 3$, 
we have 
\begin{align*} 
&(q^{\mu-2} T_{N-2} T_{N-3} \dots T_{N-\mu +1}) (q^{\mu -1} T_{N-1} T_{N-2} \dots T_{N-\mu+1}) 
\\
&= (q^{\mu -1} T_{N-1} T_{N-2} \dots T_{N-\mu+1}) (q^{\mu-2} T_{N-1} T_{N-2} \dots T_{N-\mu +2}). 
\end{align*}

\item For $N,\mu,c \in \ZZ_{\geq 0}$ such that $\mu \geq c \geq 1$, we have 
\begin{align*} 
&[T; N+1, c]^+ (q^\mu T_{N+1} T_{N+2} \dots T_{N+\mu}) 
=(q^\mu T_{N+1} T_{N+2} \dots T_{N+\mu})  [T; N,c]^+, 
\\
&[T; N-1,c]^- (q^{\mu} T_{N-1} T_{N-2} \dots T_{N-\mu}) 
= (q^{\mu} T_{N-1} T_{N-2} \dots T_{N-\mu})  [ T;N,c]^-. 
\end{align*}
\end{enumerate}
\end{lem} 
\begin{proof} 
(\roi) and (\roii) follows from the defining relations of $\He_{n,r}$. 
We can prove (\roiii) by the induction on $c$. 
\end{proof}

%
%%%%%
%

\begin{lem} 
\label{Lemma [T;N,mu ,, d]}
For $N,\mu \in \ZZ_{\geq 0}$ and $d \in \ZZ_{>0}$, we have 
\begin{align*} 
&\left[ \begin{smallmatrix} T;N,\mu \\ d \end{smallmatrix} \right]^+ 
= \begin{cases} 
	(T;N,d)^+ \Big( \left[ \begin{smallmatrix} T; N, d-1 \\ d-1 \end{smallmatrix} \right]^+ 
		\\ \dis  \quad 
		+ \sum_{h=1}^{\mu - d} (q^h T_{N+d} T_{N+d+1} \dots T_{N+d+h-1}) 
		\left[ \begin{smallmatrix} T; N, d +h -1 \\ d-1 \end{smallmatrix} \right]^+
		\Big) 
		 &\text{ if } \mu \geq d, 
		 \\
		 0 & \text{ if } \mu <d, 
	\end{cases} 
\\
& \left[ \begin{smallmatrix} T; N,\mu \\ d \end{smallmatrix} \right]^- 
= \begin{cases}  
	(T; N,d)^- \Big( \left[ \begin{smallmatrix} T; N, d-1 \\ d-1 \end{smallmatrix} \right]^- 
		\\ \dis \quad 
		+ \sum_{h=1}^{\mu -d} (q^h T_{N-d} T_{N-d-1} \dots T_{N-d -h +1}) 
		\left[ \begin{smallmatrix} T; N, d + h-1 \\ d-1 \end{smallmatrix} \right]^- \Big) 
		 & \text{ if } \mu \geq d, 
	\\
		0 & \text{ if } \mu <d.
	\end{cases} 
\end{align*}
\end{lem}

\begin{proof} 
In the case where $\mu <d$,  
we see that 
$\left[ \begin{smallmatrix} T; N,\mu \\ d \end{smallmatrix} \right]^{\pm}=0$ 
from the definitions. 

First, we prove that, 
if $\mu >d$, 
\begin{align} 
\label{[T;N,mu d]+ induction}
\left[ \begin{smallmatrix} T; N, \mu \\ d \end{smallmatrix} \right]^+  
= \left[ \begin{smallmatrix} T; N, \mu -1 \\ d \end{smallmatrix} \right]^+ 
	+ (T; N,d)^+ (q^{\mu-d} T_{N+d} T_{N+d+1} \dots T_{N+\mu-1}) 
		\left[ \begin{smallmatrix} T; N, \mu -1 \\ d -1\end{smallmatrix} \right]^+ 
\end{align}
by the induction on $d$. 
In the case where $d=1$, 
it is clear by definitions. 
Assume that $d >1$, then we have 
\begin{align*} 
\left[ \begin{smallmatrix} T; N, \mu \\ d \end{smallmatrix} \right]^+ 
&= [T; N+(d-1), \mu - (d-1)]^+ \left[ \begin{smallmatrix} T; N, \mu \\ d -1 \end{smallmatrix} \right]^+. 
\end{align*} 
Applying the assumption of the induction, we have 
\begin{align*} 
&\left[ \begin{smallmatrix} T; N, \mu \\ d \end{smallmatrix} \right]^+ 
\\
&= \Big\{ [T; N+(d-1), \mu -d]^+ + (q^{\mu -d} T_{N+d} T_{N+d+1} \dots T_{N+\mu-1}) \Big\} 
	\\
	& \quad \times 
	\Big\{ \left[ \begin{smallmatrix} T; N, \mu -1 \\ d -1\end{smallmatrix} \right]^+ 
		+ (T;N,d-1)^+ (q^{\mu - d +1} T_{N+d-1}T_{N+d} \dots T_{N+\mu -1}) 
			\left[ \begin{smallmatrix} T; N, \mu -1 \\ d -2\end{smallmatrix} \right]^+ \Big\}. 
\end{align*} 
Then, by using Lemma \ref{Lemma Li commute [T; N mu]} and Lemma \ref{Lemma [T;N,c] com rel}, 
we see that 
\begin{align*} 
&\left[ \begin{smallmatrix} T; N, \mu \\ d \end{smallmatrix} \right]^+ 
\\&= [ T; N+d-1, \mu -d]^+  \left[ \begin{smallmatrix} T; N, \mu -1 \\ d -1\end{smallmatrix} \right]^+ 
	+ (q^{\mu-d} T_{N+d} T_{N+d+1} \dots T_{N+\mu-1}) 
		\left[ \begin{smallmatrix} T; N, \mu -1 \\ d -1 \end{smallmatrix} \right]^+ 
	\\ & \quad 
	+ (T; N,d-1)^+ (q^{\mu-d+1} T_{N+d-1} T_{N+d} \dots T_{N+\mu-1}) 
		[T; N + d-2, \mu-d]^+ \left[ \begin{smallmatrix} T; N, \mu -1 \\ d -2 \end{smallmatrix} \right]^+ 
	\\ & \quad 
	+ (T; N, d-1)^+ ( q^{\mu-d+1} T_{N+d-1} T_{N+d} \dots T_{N+\mu-1}) 
		(q^{\mu-d} T_{N+d-1} T_{N+d} \dots T_{N+\mu-2}) 
		\\ & \qquad \times
		\left[ \begin{smallmatrix} T; N, \mu -1 \\ d -2 \end{smallmatrix} \right]^+. 
\end{align*} 
Note that 
\begin{align*} 
[T; N+d-2, \mu-d]^+ + q^{\mu-d} T_{N+d-1} T_{N+d} \dots T_{N+\mu-2} 
	= [T;N+d-2, \mu-d+1]^+
\end{align*} 
and 
$[T;N+d-2,\mu-d+1]^+ \left[ \begin{smallmatrix} T; N, \mu -1 \\ d -2 \end{smallmatrix} \right]^+ 
=\left[ \begin{smallmatrix} T; N, \mu -1 \\ d -1 \end{smallmatrix} \right]^+$, 
we have 
\begin{align*} 
&\left[ \begin{smallmatrix} T; N, \mu \\ d \end{smallmatrix} \right]^+ 
\\
&= \left[ \begin{smallmatrix} T; N, \mu -1 \\ d \end{smallmatrix} \right]^+ 
	\\ & \quad 
	+\big( 1 + (T; N, d-1)^+ (q T_{N+d-1}) \big) 
	(q^{\mu-d} T_{N+d} T_{N+d+1} \dots T_{N+\mu-1}) 
	\left[ \begin{smallmatrix} T; N, \mu -1 \\ d-1 \end{smallmatrix} \right]^+. 
\end{align*} 
By definition, we see that 
$1 + (T;N,d-1)^+ (q T_{N+d-1})= (T;N,d)^+$. 
Thus, we have \eqref{[T;N,mu d]+ induction}. 

%%%%%%

Next, we prove that 
\begin{align} 
\label{[T;N,d,,d]^+ induction}
\left[ \begin{smallmatrix} T; N,d \\ d \end{smallmatrix} \right]^+ 
= (T; N,d)^+ \left[ \begin{smallmatrix} T; N,d -1 \\ d -1\end{smallmatrix} \right]^+ 
\end{align} 
by the induction on $d$. 
In the case where $d=1$, it is clear from definitions. 
Assume that $d>1$. 
Note that 
$\left[ \begin{smallmatrix} T; N,d \\ d \end{smallmatrix} \right]^+ 
	= \left[ \begin{smallmatrix} T; N,d \\ d -1 \end{smallmatrix} \right]^+$, 
by \eqref{[T;N,mu d]+ induction}, we have 

\begin{align*} 
\left[ \begin{smallmatrix} T; N,d \\ d \end{smallmatrix} \right]^+ 
&= \left[ \begin{smallmatrix} T; N,d -1 \\ d -1 \end{smallmatrix} \right]^+ 
	+ (T;N,d-1)^+ (q T_{N+d-1} ) \left[ \begin{smallmatrix} T; N,d -1\\ d -2 \end{smallmatrix} \right]^+ 
\\
&= \big( 1 + (T; N, d-1)^+ (q T_{N+d-1}) \big) 
	\left[ \begin{smallmatrix} T; N,d -1 \\ d -1\end{smallmatrix} \right]^+ 
\\
&= (T;N,d)^+ \left[ \begin{smallmatrix} T; N,d -1\\ d -1\end{smallmatrix} \right]^+. 
\end{align*}

%%%%%%

Next we prove that, if $\mu \geq d$, 
\begin{align} 
\label{[T;N,mu,,d]^+ (T;N,d)^+ **} 
\begin{split} 
&\left[ \begin{smallmatrix} T; N,\mu \\ d \end{smallmatrix} \right]^+
\\
&=(T;N,d)^+ \Big( \left[ \begin{smallmatrix} T; N, d-1 \\ d-1 \end{smallmatrix} \right]^+ 
		+ \sum_{h=1}^{\mu - d} (q^h T_{N+d} T_{N+d+1} \dots T_{N+d+h-1}) 
		\left[ \begin{smallmatrix} T; N, d +h -1 \\ d-1 \end{smallmatrix} \right]^+\Big)
\end{split}
\end{align} 
by the induction on $\mu -d$. 
In the case where $\mu=d$, it is just \eqref{[T;N,d,,d]^+ induction}. 
Assume that $\mu >d$. 
By applying the assumption of the induction to the right-hand side of 
\eqref{[T;N,mu d]+ induction}, we have \eqref{[T;N,mu,,d]^+ (T;N,d)^+ **}. 

It is similar for $\left[ \begin{smallmatrix} T; N,\mu \\ d \end{smallmatrix} \right]^-$. 
\end{proof}

We have the following corollary which will be used in Theorem \ref{Thm UqQ to Sc}
to consider the divided powers in cyclotomic $q$-Schur algebras.

\begin{cor} 
\label{Cor [T;N,mu ,, d]}
For $N, \mu \in \ZZ_{\geq 0}$ and $d \in \ZZ_{>0}$, 
there exist  the elements 
$\mathfrak{H}^{\pm} (N,\mu,d) \in \He_{n,r}$ 
such that 
\begin{align*}
\left[ \begin{smallmatrix} T; N,\mu \\ d \end{smallmatrix} \right]^{\pm} 
= (T;N,d)^{\pm}! \mathfrak{H}^{\pm}(N,\mu,d).  
\end{align*}
\end{cor}

\begin{proof} 
Note that 
$T_{N+d} T_{N+d+1} \dots T_{N+d+h-1}$ (resp. $T_{N-d} T_{N-d-1} \dots T_{N-d-h+1}$) 
commute with $(T;N,d-1)^+!$ (resp. $(T;N,d-1)^-!$), 
then we can prove the corollary 
by the induction on $d$ using Lemma \ref{Lemma [T;N,mu ,, d]}. 
\end{proof}

%
%%%%%% ##############################################
%

\para 
For $(i,k) \in \vG'(\Bm)$, we define the elements $\CX_{(i,k),0}^+$ and $\CX_{(i,k),0}^-$ of $\Sc_{n,r}(\Bm)$ by 
\begin{align*}
&\CX_{(i,k),0}^+ (m_{\mu}) 
	= q^{- \mu_{i+1}^{(k)}+1} m_{\mu+\a_{(i,k)}} 
		[T;N_{(i,k)}^{\mu}, \mu_{i+1}^{(k)}]^+, 
\\
&\CX_{(i,k),0}^- (m_{\mu})  
	= q^{ - \mu_i^{(k)} +1} m_{\mu - \a_{(i,k)}} h^{\mu}_{-(i,k)} [T; N_{(i,k)}^\mu, \mu_i^{(k)}]^- 
\end{align*} 
for each $\mu \in \vL_{n,r}(\Bm)$, 
where 
we put $\mu_{m_k+1}^{(k)} = \mu_1^{(k+1)}$ if $i=m_k$, 
and 
\begin{align*}
h^{\mu}_{-(i,k)} 
	= \begin{cases} 1 & \text{ if } i \not= m_k ,\\ L_{N^{\mu}_{(m_k,k)}} - Q_k & \text{ if } i=m_k. \end{cases} 
\end{align*}
Note that $m_{\mu \pm \a_{(i,k)}}=0$  if $\mu \pm \a_{(i,k)} \not \in \vL_{n,r}(\Bm)$. 

By \cite[Lemma 6.10]{W}, 
the definitions of $\CX_{(i,k),0}^{\pm}$ are well-defined. 
(The elements $\CX_{(i,k),0}^{\pm}$ are denoted by $\vf_{(i,k)}^{\pm}$ in \cite{W}.)  

For $(i,k) \in \vG'(\Bm)$ and an integer $t>0$, 
we define the elements $\CX_{(i,k),t}^{\pm}$ of $\Sc_{n,r}(\Bm)$ inductively by 
\begin{align}
\label{Def CX t}
\begin{split}
&\CX_{(i,k),t}^{+} 
	=  \CI_{(i,k),1}^{+} \CX_{(i,k),t-1}^{+} - \CX_{(i,k),t-1}^{+} \CI_{(i,k),1}^{+}, 
\\
&\CX_{(i,k),t}^{-} 
	= - \big( \CI_{(i,k),1}^{-} \CX_{(i,k),t-1}^{-} - \CX_{(i,k),t-1}^{-} \CI_{(i,k),1}^{-}\big). 
\end{split}
\end{align}

% #####

\begin{lem}
\label{Lemma K+ X K-}
For $(i,k) \in \vG'(\Bm)$, $(j,l) \in \vG(\Bm)$ and $t \geq 0$, we have 
\begin{align*}
\CK_{(j,l)}^+ \CX_{(i,k),t}^{\pm} \CK_{(j,l)}^- 
	= q^{ \pm a _{(i,k)(j,l)}} \CX_{(i,k),t}^{\pm}. 
\end{align*}
\end{lem}

\begin{proof}
We see the statement in the case where $t=0$ from the definitions directly. 
Then we can prove the statement by the induction on $t$ 
using \eqref{Def CX t} together with Lemma \ref{Lemma commute K H}. 
\end{proof}

% #####

We can describe the elements $\CX_{(i,k),t}^{\pm}$ of $\Sc_{n,r}(\Bm)$ precisely as follows. 

\begin{lem}
\label{Lemma Xpmt} 
For $(i,k) \in \vG'(\Bm)$, $t \geq 0$ and $\mu \in \vL_{n,r}(\Bm)$, 
we have the followings. 
\begin{enumerate}
\item 
$\CX_{(i,k),t}^+(m_\mu)  
= q^{ - \mu_{i+1}^{(k)}+1} m_{\mu + \a_{(i,k)}} L_{N^{\mu}_{(i,k)} +1}^t [T; N_{(i,k)}^\mu, \mu_{i+1}^{(k)}]^+. $

\item 
$\CX_{(i,k),t}^- (m_{\mu}) 
=  q^{- \mu_i^{(k)}+1}  m_{\mu - \a_{(i,k)}} L_{N^{\mu}_{(i,k)}}^t h^{\mu}_{-(i,k)} [T; N_{(i,k)}^\mu, \mu_i^{(k)}]^-.  $
\end{enumerate}
\end{lem}

\begin{proof}
We prove (\roi). 
We can easily show that 
$\CX_{(i,k),t}^+ (m_\mu)=0$ 
if $\mu_{i+1}^{(k)} =0$ by the induction on $t$ using \eqref{Def CX t}. 
Assume that $\mu_{i+1}^{(k)} \not=0$. 
If $t=0$, then it is just the definition of $\CX_{(i,k),0}^+$. 
We prove the equation for $t >0$ by the induction on $t$. 
Note that $(\mu+\a_{(i,k)})_i^{(k)} = \mu_i^{(k)}+1$ 
and $N^{\mu+\a_{(i,k)}}_{(i,k)} = N^{\mu}_{(i,k)} +1$, 
by the assumption of the induction, we have 
\begin{align*}
\CI_{(i,k),1}^+ \CX_{(i,k),t-1}^+ (m_{\mu}) 
= & q^{- \mu_{i+1}^{(k)}+1} m_{\mu + \a_{(i,k)}} 
	\\ & \times 
	( L_{N^{\mu}_{(i,k)}+1} + L_{N^{\mu}_{(i,k)}} + L_{N^{\mu}_{(i,k)} -1} + \dots + L_{N^{\mu}_{(i,k)} - \mu_i^{(k)} +1})
	\\ & \times
	L^{t-1}_{N^{\mu}_{(i,k)}+1} 
	[T; N_{(i,k)}^\mu, \mu_{i+1}^{(k)}]^+.  
\end{align*} 
On the other hand, we have 
\begin{align*}
\CX_{(i,k),t-1}^+ \CI_{(i,k),1}^+ (m_\mu) 
= & \d_{(\mu_i^{(k)} \not=0)} q^{- \mu_{i+1}^{(k)}+1}  m_{\mu+\a_{(i,k)}} 
	L^{t-1}_{N^{\mu}_{(i,k)}+1} [T;N_{(i,k)}^\mu, \mu_{i+1}^{(k)}]^+ 
	\\ &  \times 
	( L_{N^{\mu}_{(i,k)}} + L_{N^{\mu}_{(i,k)} -1} + \dots + L_{N^{\mu}_{(i,k)} - \mu_i^{(k)} +1}). 
\end{align*}
Thus, by \eqref{Def CX t} and Lemma \ref{Lemma Li commute [T; N mu]}, 
we have (\roi). 
(\roii) is similar. 
\end{proof}

% #####

\begin{prop}
\label{Proposition CX CX}
For $(i,k),(j,l) \in \vG'(\Bm)$ and $s,t \geq 0$, 
we have the following relations. 

\begin{enumerate} 
\item 
$[\CX_{(i,k),t}^{\pm}, \CX_{(j,l),s}^{\pm}] =0$ 
if $(j,l) \not= (i,k), (i \pm 1,k)$.

\item 
$\CX_{(i,k),t+1}^{\pm} \CX_{(i,k),s}^{\pm} - q^{\pm 2} \CX_{(i,k),s}^{\pm} \CX_{(i,k),t+1}^{\pm} 
	= q^{\pm 2} \CX_{(i,k),t}^{\pm} \CX_{(i,k),s+1}^{\pm} - \CX_{(i,k),s+1}^{\pm} \CX_{(i,k),t}^{\pm}$. 
\item 
\begin{align*} 
	& \CX_{(i,k),t+1}^+ \CX_{(i+1,k),s}^+ - q^{-1} \CX_{(i+1,k),s}^+ \CX_{(i,k),t+1}^+ 
		= \CX_{(i,k),t}^+ \CX_{(i+1,k),s+1}^+ - q \CX_{(i+1,k),s+1}^+ \CX_{(i,k),t}^+, 
	\\ & 
	\CX_{(i+1,k),s}^- \CX_{(i,k),t+1}^- - q^{-1} \CX_{(i,k),t+1}^- \CX_{(i+1,k),s}^- 
		= \CX_{(i+1,k),s+1}^- \CX_{(i,k),t}^- - q \CX_{(i,k),t}^- \CX_{(i+1,k),s+1}^-. 
\end{align*} 
\end{enumerate} 
\end{prop}
\begin{proof}
(\roi) follows from Lemma \ref{Lemma Xpmt} using Lemma \ref{Lemma commute L T}.

We prove (\roii). 
We may assume that $t \geq s$ by multiplying $-1$ to both sides if necessary. 
We prove 
\begin{align}
\label{CXt+1+ CXs+ - CXs+ CXt+1+}
\CX_{(i,k),t+1}^+ \CX_{(i,k),s}^+ - q^2 \CX_{(i,k),s}^+ \CX_{(i,k),t+1}^+ 
= q^2 \CX_{(i,k),t}^+ \CX_{(i,k),s+1}^+ - \CX_{(i,k),s+1}^+ \CX_{(i,k),t}^+. 
\end{align} 
Put $N=N_{(i,k)}^\mu$. 
By Lemma \ref{Lemma Xpmt} together with Lemma \ref{Lemma Li commute [T; N mu]}, 
for $\mu \in \vL_{n,r}(\Bm)$, 
we have 
\begin{align}
\label{CX+CX+}
\begin{split}  
&\CX_{(i,k),t+1}^+ \CX_{(i,k),s}^+ (m_{\mu}) 
\\
&= q^{- 2 \mu_{i+1}^{(k)}+3} m_{\mu + 2 \a_{(i,k)}} L_{N+1}^s L_{N+2}^{t+1} 
	[T; N+1, \mu_{i+1}^{(k)}-1]^+ [T; N, \mu_{i+1}^{(k)}]^+. 
\end{split} 
\end{align}
Thus, we may assume that $\mu_{i+1}^{(k)} \geq 2$ since $m_{\mu + 2 \a_{(i,k)}}=0$ if $\mu_{i+1}^{(k)}<2$. 
By the induction on $\mu_{i+1}^{(k)}$, 
we can show that 
\begin{align}
\label{T_{N+1} [T]+[T]+}  
T_{N+1} [T; N+1, \mu_{i+1}^{(k)}-1]^+ [T; N, \mu_{i+1}^{(k)}]^+ 
= q [T; N+1, \mu_{i+1}^{(k)}-1]^+ [T; N, \mu_{i+1}^{(k)}]^+. 
\end{align}
%%%%%%
We also have, 
by Lemma \ref{Lemma commute L T}, 
\begin{align} 
\label{LN+1s LN+2t+1}
\begin{split} 
L_{N+1}^s L_{N+2}^{t+1} 
&= (L_{N+1} L_{N+2})^s (T_{N+1} L_{N+1} T_{N+1}) L_{N+2}^{t-s} 
\\
&= T_{N+1} (L_{N+1} L_{N+2})^s L_{N+1} \big\{ L_{N+1}^{t-s} T_{N+1} 
	+ (q-q^{-1}) \sum_{p=1}^{t-s} L_{N+1}^{t-s-p} L_{N+2}^p \big\} 
\\
&= T_{N+1} L_{N+1}^{t+1} L_{N+2}^s T_{N+1} 
 + (q-q^{-1}) T_{N+1} \sum_{p=1}^{t-s} L_{N+1}^{t-p+1} L_{N+2}^{s +p}. 
\end{split}
\end{align}
Then, 
\eqref{CX+CX+} by using \eqref{ m mu T}, \eqref{T_{N+1} [T]+[T]+} and \eqref{LN+1s LN+2t+1}, 
we have 
\begin{align*} 
&\CX_{(i,k),t+1}^+ \CX_{(i,k),s}^+ (m_{\mu}) 
\\
&= q^2 q^{- 2 \mu_{i+1}^{(k)}+3} m_{\mu + 2 \a_{(i,k)}} L_{N+1}^{t+1} L_{N+2}^{s} 
	[T; N+1, \mu_{i+1}^{(k)}-1]^+ [T; N, \mu_{i+1}^{(k)}]^+ 
	\\ & \quad 
	+ q (q-q^{-1}) q^{- 2 \mu_{i+1}^{(k)}+3} m_{\mu + 2 \a_{(i,k)}} 
	\sum_{p=1}^{t-s}  L_{N+1}^{t-p+1} L_{N+2}^{s+p} 
	[T; N+1, \mu_{i+1}^{(k)}-1]^+ [T; N, \mu_{i+1}^{(k)}]^+ 
\\
&= q^2 \CX_{(i,k),s}^+ \CX_{(i,k),t+1}^+ (m_{\mu}) 
	\\ & \quad 
	+ q (q-q^{-1}) q^{- 2 \mu_{i+1}^{(k)}+3} m_{\mu + 2 \a_{(i,k)}} 
	\sum_{p=1}^{t-s}  L_{N+1}^{t-p+1} L_{N+2}^{s+p} 
	[T; N+1, \mu_{i+1}^{(k)}-1]^+ [T; N, \mu_{i+1}^{(k)}]^+. 
\end{align*}
Similarly, we have 
\begin{align*} 
&q^2 \CX_{(i,k),t}^+ \CX_{(i,k),s+1}^+ (m_{\mu}) 
\\
&= q^{- 2 \mu_{i+1}^{(k)}+3} m_{\mu + 2 \a_{(i,k)}} T_{N+1} L_{N+1}^{s+1} L_{N+2}^{t}  T_{N+1} 
	[T; N+1, \mu_{i+1}^{(k)} -1]^+ [T; N, \mu_{i+1}^{(k)}]^+ 
\\
&=  q^{- 2 \mu_{i+1}^{(k)}+3} m_{\mu + 2 \a_{(i,k)}} L_{N+1}^t L_{N+2}^{s+1} 
		[T; N+1, \mu_{i+1}^{(k)}-1]^+ [T; N, \mu_{i+1}^{(k)}]^+ 
	\\ & \quad 
	+ q (q-q^{-1})  q^{- 2 \mu_{i+1}^{(k)}+3} m_{\mu + 2 \a_{(i,k)}} 
		\sum_{p=1}^{t-s} L_{N+1}^{t-p +1} L_{N+2}^{s+p} 
		 [T; N+1, \mu_{i+1}^{(k)}-1]^+ [T; N, \mu_{i+1}^{(k)}]^+ 
\\
&= \CX_{(i,k),s+1}^+ \CX_{(i,k),t}^+ (m_{\mu}) 
	\\ & \quad 
	+ q (q-q^{-1})  q^{- 2 \mu_{i+1}^{(k)}+3} m_{\mu + 2 \a_{(i,k)}} 
		\sum_{p=1}^{t-s} L_{N+1}^{t-p +1} L_{N+2}^{s+p} 
		 [T; N+1, \mu_{i+1}^{(k)}-1]^+ [T; N, \mu_{i+1}^{(k)}]^+. 
\end{align*}
Thus, we have \eqref{CXt+1+ CXs+ - CXs+ CXt+1+}. 
Another case of (\roii) is proven in a similar way. 

%%%%%

We prove (\roiii). 
Put $N= N_{(i,k)}^\mu$. 
In the case where $\mu_{i+1}^{(k)}=0$, 
by Lemma \ref{Lemma Xpmt} together with  Lemma \ref{Lemma Li commute [T; N mu]}, 
we see that  
\begin{align} 
\label{CXi+t+1 CXi+1+s mu i+1 =1}
\begin{split}
&(\CX_{(i,k),t+1}^+ \CX_{(i+1,k),s}^+ - q^{-1} \CX_{(i+1,k),s}^+ \CX_{(i,k),t+1}^+) (m_{\mu}) 
\\
&= (\CX_{(i,k),t}^+ \CX_{(i+1,k),s+1}^+ - q \CX_{(i+1,k),s+1}^+ \CX_{(i,k),t}^+) (m_{\mu}) 
\\
&= q^{- \mu_{i+2}^{(k)}+1} m_{\mu + \a_{(i,k)} + \a_{(i+1,k)}} L_{N+1}^{s+t+1} 
	[T; N, \mu_{i+2}^{(k)}]^+. 
\end{split} 
\end{align}
Assume that $\mu_{i+1}^{(k)} \not=0$. 
By Lemma \ref{Lemma Xpmt} together with  Lemma \ref{Lemma Li commute [T; N mu]}, 
we have 
\begin{align}
\label{CXi+t+1 CXi+1+s}
\begin{split}
&(\CX_{(i,k),t+1}^+ \CX_{(i+1,k),s}^+ - q^{-1} \CX_{(i+1,k),s}^+ \CX_{(i,k),t+1}^+) (m_{\mu}) 
\\
&= 
q^{- \mu_{i+1}^{(k)} - \mu_{i+2}^{(k)} +1} m_{\mu + \a_{(i,k)} + \a_{(i+1,k)}} 
	 L_{N+1}^{t+1} (q^{\mu_{i+1}^{(k)}} T_{N+1} T_{N+2} \dots T_{N+\mu_{i+1}^{(k)}}) 
	 L_{N+ \mu_{i+1}^{(k)}+1}^s 
	 \\ & \qquad \times 
	[T; N+\mu_{i+1}^{(k)}, \mu_{i+2}^{(k)}]^+ 
\end{split}
\end{align} 
and 
\begin{align} 
\label{CXi+t CXi+1+s+1}
\begin{split} 
&(\CX_{(i,k),t}^+ \CX_{(i+1,k),s+1}^+ - q \CX_{(i+1,k),s+1}^+ \CX_{(i,k),t}^+) (m_{\mu})  
\\
&= - (q-q^{-1}) q^{ - \mu_{i+1}^{(k)} - \mu_{i+2}^{(k)}+2} m_{\mu+\a_{(i,k)} + \a_{(i+1,k)}} 
	L_{N+1}^t L_{N+\mu_{i+1}^{(k)}+1}^{s+1} 
	\\ & \qquad \times 
	[T; N,\mu_{i+1}^{(k)}]^+ 
	[T; N+\mu_{i+1}^{(k)}, \mu_{i+2}^{(k)}]^+ 
	\\ & \quad 
	+ q^{- \mu_{i+1}^{(k)} - \mu_{i+2}^{(k)} +1} m_{\mu + \a_{(i,k)} + \a_{(i+1,k)}} 
	 L_{N+1}^{t} (q^{\mu_{i+1}^{(k)}} T_{N+1} T_{N+2} \dots T_{N+\mu_{i+1}^{(k)}}) 
	 L_{N+ \mu_{i+1}^{(k)}+1}^{s+1} 
	 \\ & \qquad \times 
	[T; N+\mu_{i+1}^{(k)}, \mu_{i+2}^{(k)}]^+ .  
\end{split} 
\end{align}
By the induction on $\mu_{i+1}^{(k)}$ using Lemma \ref{Lemma commute L T}, we can prove that  
\begin{align*} 
&(T_{N+1} T_{N+2} \dots T_{N+\mu_{i+1}^{(k)}}) L_{N+\mu_{i+1}^{(k)}+1} 
\\
&= L_{N+1} (T_{N+1} T_{N+2} \dots T_{N+\mu_{i+1}^{(k)}})
	+ \d_{(\mu_{i+1}^{(k)} \geq 2)} (q-q^{-1}) L_{N+2} (T_{N+2} T_{N+3} \dots T_{N+\mu_{i+1}^{(k)}}) 
	\\ & \quad 
	+ (q-q^{-1}) \sum_{p=1}^{\mu_{i+1}^{(k)}-2} (T_{N+1} T_{N+2} \dots T_{N+p}) 
		L_{N+p+2} (T_{N+p+2} T_{N+p+3} \dots  T_{N+\mu_{i+1}^{(k)}} ) 
	\\ & \quad 
	+ (q-q^{-1}) (T_{N+1} T_{N+2} \dots T_{N+\mu_{i+1}^{(k)}-1}) L_{N+\mu_{i+1}^{(k)} +1}. 
\end{align*}
By using Lemma \ref{Lemma commute L T} and \eqref{ m mu T}, 
this equation implies 
\begin{align} 
\label{Lt T T L}
\begin{split} 
&m_{\mu + \a_{(i,k)} + \a_{(i+1,k)}} L_{N+1}^t (q^{\mu_{i+1}^{(k)}} T_{N+1} T_{N+2} \dots T_{N+\mu_{i+1}^{(k)}}) 
	L_{N+\mu_{i+1}^{(k)}+1}  
\\
&= m_{\mu + \a_{(i,k)} + \a_{(i+1,k)}} L_{N+1}^{t+1} 
	(q^{\mu_{i+1}^{(k)}} T_{N+1} T_{N+2} \dots T_{N+\mu_{i+1}^{(k)}}) 
	\\ & \quad 
	+q (q-q^{-1}) m_{\mu + \a_{(i,k)} + \a_{(i+1,k)}} L_{N+1}^t 
	[T; N, \mu_{i+1}^{(k)}]^+  L_{N+\mu_{i+1}^{(k)}+1}. 
\end{split} 
\end{align}
Thus, \eqref{CXi+t CXi+1+s+1} and \eqref{Lt T T L} imply 
\begin{align} 
\label{CXi+t CXi+1+s+1 (2)}
\begin{split} 
&(\CX_{(i,k),t}^+ \CX_{(i+1,k),s+1}^+ - q \CX_{(i+1,k),s+1}^+ \CX_{(i,k),t}^+) (m_{\mu})  
\\
&= q^{- \mu_{i+1}^{(k)} - \mu_{i+2}^{(k)}+1} m_{\mu + \a_{(i,k)} + \a_{(i+1,k)}} L_{N+1}^{t+1} 
	(q^{\mu_{i+1}^{(k)}} T_{N+1} T_{N+2} \dots T_{N+\mu_{i+1}^{(k)}})
	 L_{N+\mu_{i+1}^{(k)}+1}^s 
	 \\ & \qquad \times 
	[T; N+ \mu_{i+1}^{(k)}, \mu_{i+2}^{(k)}]^+.  
\end{split} 
\end{align}
By \eqref{CXi+t+1 CXi+1+s mu i+1 =1}, \eqref{CXi+t+1 CXi+1+s} and \eqref{CXi+t CXi+1+s+1 (2)}, 
we have 
\begin{align*}
\CX_{(i,k),t+1}^+ \CX_{(i+1,k),s}^+ - q^{-1} \CX_{(i+1,k),s}^+ \CX_{(i,k),t+1}^+
= \CX_{(i,k),t}^+ \CX_{(i+1,k),s+1}^+ - q \CX_{(i+1,k),s+1}^+ \CX_{(i,k),t}^+. 
\end{align*} 
Another case of (\roiii) is proven in a similar way. 
\end{proof}

% #######################################

\begin{prop}
\label{Prop q-Serre relations}
For $(i,k) \in \vG'(\Bm)$ and $s,t,u \geq 0$, we have the followings. 
\begin{enumerate} 
\item 
\begin{align*}
&\CX_{(i \pm 1,k),u}^+ \big( \CX_{(i,k),s}^+ \CX_{(i,k),t}^+ + \CX_{(i,k),t}^+ \CX_{(i,k),s}^+ \big) 
		+ \big( \CX_{(i,k),s}^+ \CX_{(i,k),t}^+ + \CX_{(i,k),t}^+ \CX_{(i,k),s}^+ \big) \CX_{(i \pm 1,k),u}^+ 
	\\ 
	&= (q+q^{-1}) \big( \CX_{(i,k),s}^+ \CX_{(i \pm1,k),u}^+ \CX_{(i,k),t}^+ 
		+ \CX_{(i,k),t}^+ \CX_{(i \pm 1,k),u}^+ \CX_{(i,k),s}^+ \big). 
\end{align*} 

\item 
\begin{align*} 
	& \CX_{(i \pm 1,k),u}^- \big( \CX_{(i,k),s}^- \CX_{(i,k),t}^- + \CX_{(i,k),t}^- \CX_{(i,k),s}^- \big) 
		+ \big( \CX_{(i,k),s}^- \CX_{(i,k),t}^- + \CX_{(i,k),t}^- \CX_{(i,k),s}^- \big) \CX_{(i \pm 1,k),u}^- 
	\\ 
	&= (q+q^{-1}) \big( \CX_{(i,k),s}^- \CX_{(i \pm1,k),u}^- \CX_{(i,k),t}^- 
		+ \CX_{(i,k),t}^- \CX_{(i \pm 1,k),u}^- \CX_{(i,k),s}^- \big). 
\end{align*}
\end{enumerate} 
\end{prop}

\begin{proof}
By Lemma \ref{Lemma Xpmt} together with Lemma \ref{Lemma Li commute [T; N mu]}, 
we have 
\begin{align}
\label{CX+u CX+s CX+t -CX+s CX+u CX+t}
\begin{split} 
&(\CX_{(i+1,k),u}^+ \CX_{(i,k),s}^+ \CX_{(i,k),t}^+ - q \CX_{(i,k),s}^+ \CX_{(i+1,k),u}^+ \CX_{(i,k),t}^+) (m_{\mu}) 
\\
&= - \d_{(\mu_{i+1}^{(k)}=1)} q^{- \mu_{i+2}^{(k)}+2} m_{\mu+2 \a_{(i,k)} + \a_{(i+1,k)}} 
	L_{N+1}^t L_{N+2}^{s+u} [T; N+1, \mu_{i+2}^{(k)}]^+ 
	\\
	&\quad  - \d_{(\mu_{i+1}^{(k)} \geq 2)} q^{- 2 \mu_{i+1}^{(k)} - \mu_{i+2}^{(k)}+4} 
		m_{\mu+2 \a_{(i,k)} + \a_{(i+1,k)}} L_{N+1}^t L_{N+2}^s 
		\\ & \qquad \times 
		( q^{\mu_{i+1}^{(k)}-1} T_{N+2} T_{N+3} \dots T_{N+\mu_{i+1}^{(k)}}) [T; N, \mu_{i+1}^{(k)}]^+ 
		L_{N+ \mu_{i+1}^{(k)}+1}^u  [T; N+\mu_{i+1}^{(k)}, \mu_{i+2}^{(k)}]^+ 
\end{split}
\end{align}
and 
\begin{align*} 
&(\CX_{(i,k),s}^+ \CX_{(i,k),t}^+ \CX_{(i+1,k),u}^+ - q^{-1} \CX_{(i,k),s}^+ \CX_{(i+1,k),u}^+ \CX_{(i,k),t}^+)(m_{\mu}) 
\\
&=q^{-2 \mu_{i+1}^{(k)} - \mu_{i+2}^{(k)} +2} m_{\mu+2 \a_{(i,k)} + \a_{(i+1,k)}} L_{N+1}^t L_{N+2}^s 
	\\ & \quad \times 
	[T; N+1, \mu_{i+1}^{(k)}]^+ 
		(q^{\mu_{i+1}^{(k)}} T_{N+1} T_{N+2} \dots T_{N+ \mu_{i+1}^{(k)}} )
		L_{N+ \mu_{i+1}^{(k)}+1}^u [T; N+\mu_{i+1}^{(k)}, \mu_{i+2}^{(k)}]^+. 
\end{align*}
Applying Lemma \ref{Lemma [T;N,c] com rel} (\roiii), 
we have 
\begin{align} 
\label{CX+s CX+t CX+u -CX+s CX+u CX+t}
\begin{split}
&(\CX_{(i,k),s}^+ \CX_{(i,k),t}^+ \CX_{(i+1,k),u}^+ - q^{-1} \CX_{(i,k),s}^+ \CX_{(i+1,k),u}^+ \CX_{(i,k),t}^+)(m_{\mu}) 
\\
&= \d_{(\mu_{i+1}^{(k)}=1)} q^{- \mu_{i+2}^{(k)}+1} m_{\mu+ 2 \a_{(i,k)} + \a_{i+1,k)}} L_{N+1}^t L_{N+2}^s 
	T_{N+1} L_{N+2}^u  [T; N+\mu_{i+1}^{(k)}, \mu_{i+2}^{(k)}]^+ 
	\\ & \quad 
	+ \d_{\mu_{i+1}^{(k)} \geq 2} 
	q^{- 2 \mu_{i+1}^{(k)} - \mu_{i+2}^{(k)} +3} m_{\mu + 2 \a_{(i,k)} + \a_{(i+1,k)}} 
	L_{N+1}^t L_{N+2}^s T_{N+1} 
	\\ & \qquad \times 
	( q^{\mu_{i+1}^{(k)}-1} T_{N+2} T_{N+3} \dots T_{N+\mu_{i+1}^{(k)}})
		[T; N; \mu_{i+1}^{(k)}]^+ 
		L_{N+ \mu_{i+1}^{(k)}+1}^u [T; N+\mu_{i+1}^{(k)}, \mu_{i+2}^{(k)}]^+. 
\end{split}
\end{align}
We see that 
\begin{align*} 
&m_{\mu+2 \a_{(i,k)} + \a_{(i+1,k)}}(L_{N+1}^t L_{N+2}^s + L_{N+1}^s L_{N+2}^t)T_{N+1} 
\\
&= m_{\mu+2 \a_{(i,k)} + \a_{(i+1,k)}} T_{N+1} (L_{N+1}^t L_{N+2}^s + L_{N+1}^s L_{N+2}^t)
\\
&= q m_{\mu+2 \a_{(i,k)} + \a_{(i+1,k)}} (L_{N+1}^t L_{N+2}^s + L_{N+1}^s L_{N+2}^t)
\end{align*}
by Lemma \ref{Lemma commute L T} and \eqref{ m mu T}. 
Then 
\eqref{CX+u CX+s CX+t -CX+s CX+u CX+t} and \eqref{CX+s CX+t CX+u -CX+s CX+u CX+t} imply 
\begin{align*} 
&\CX_{(i+1,k),u}^+ (\CX_{(i,k),s}^+ \CX_{(i,k),t}^+ + \CX_{(i,k),t}^+ \CX_{(i,k),s}^+)  
	+ (\CX_{(i,k),s}^+ \CX_{(i,k),t}^+ + \CX_{(i,k),t}^+ \CX_{(i,k),s}^+) \CX_{(i+1,k),u}^+ 
\\
&= (q+q^{-1}) (\CX_{(i,k),s}^+ \CX_{(i+1,k),u}^+ \CX_{(i,k),t}^+ + \CX_{(i,k),t}^+ \CX_{(i+1,k),u}^+ \CX_{(i,k),s}^+). 
\end{align*}
The other cases of the proposition are proven in a similar way. 
\end{proof}

% ######################################

By direct calculations, we have the following lemma. 

\begin{lem}
\label{Lemma CI 0 CX+-}
For $(i,k) \in \vG'(\Bm)$, $(j,l) \in \vG(\Bm)$, $t \geq 0$, we have the followings. 
\begin{enumerate}
\item 
$ q^{\pm a_{(i,k)(j,l)}} \CI_{(j,l),0}^{\pm} \CX_{(i,k),t}^{+} - q^{\mp a_{(i,k)(j,l)}} \CX_{(i,k),t}^{+} \CI_{(j,l),0}^{\pm}  
	=  a_{(i,k)(j,l)} \CX_{(i,k),t}^{+}$. 

\item 
$ q^{\mp a_{(i,k)(j,l)}} \CI_{(j,l),0}^{\pm} \CX_{(i,k),t}^{-} - q^{\pm a_{(i,k)(j,l)}} \CX_{(i,k),t}^{-} \CI_{(j,l),0}^{\pm}   
	=  - a_{(i,k)(j,l)} \CX_{(i,k),t}^{-}$. 
\end{enumerate}
\end{lem}

We also have the following proposition. 

\begin{prop}
\label{Proposition IX - XI}
For $(i,k) \in \vG'(\Bm)$, $(j,l) \in \vG(\Bm)$, $s, t \geq0$, we have the followings. 
\begin{enumerate}
\item 
$[\CI_{(j,l),s+1}^{\pm}, \CX_{(i,k),t}^{+}] 
	= q^{ \pm a_{(i,k)(j,l)}} \CI_{(j,l),s}^{\pm} \CX_{(i,k),t+1}^{+} 
		- q^{ \mp a_{(i,k)(j,l)}} \CX_{(i,k),t+1}^{+} \CI_{(j,l),s}^{\pm}.$ 

\item 
$[\CI_{(j,l),s+1}^{\pm}, \CX_{(i,k),t}^{-}] 
	= q^{\mp a_{(i,k)(j,l)}} \CI_{(j,l),s}^{\pm} \CX_{(i,k),t+1}^{-} 
		- q^{\pm a_{(i,k)(j,l)}} \CX_{(i,k),t+1}^{-} \CI_{(j,l),s}^{\pm}$. 
\end{enumerate}
\end{prop}

\begin{proof}
By Lemma \ref{Lemma Xpmt} together with Lemma \ref{Lemma commute L T}, 
we see that 
\begin{align*}
[\CI_{(j,l),s}^{\s}, \CX_{(i,k),t}^{\s'}] =0 
\text{ if } (j,l) \not=(i,k), (i+1,k), 
\end{align*}
where $\s,\s' \in \{ +,-\}$. 
Thus, it is enough to prove the cases where $(j,l)=(i,k) $ or $(j,l)=(i+1,k)$.  
We prove 
\begin{align}
\label{CI+ CX+ - CX+ CI+}
[\CI_{(i,k),s+1}^+, \CX_{(i,k),t}^+] = q \CI_{(i,k),s}^+ \CX_{(i,k),t+1}^+ - q^{-1} \CX_{(i,k),t+1}^+ \CI_{(i,k),s}^+. 
\end{align}  
For $\mu \in \vL_{n,r}(\Bm)$, put $N=N^{\mu}_{(i,k)}$. 
Then, by Lemma \ref{Lemma Xpmt} together with Lemma \ref{Lemma Li commute [T; N mu]}, 
we have 
\begin{align*}
&(\CI_{(i,k),s+1}^+ \CX_{(i,k),t}^+ - \CX_{(i,k),t}^+ \CI_{(i,k),s+1}^+) (m_{\mu}) 
\\
&= q^{ s - \mu_{i+1}^{(k)} +1} m_{\mu+\a_{(i,k)}} 
	\\ & \quad \times 
	\big( \Phi_{s+1}^+ (L_{N+1}, L_N, L_{N-1}, \dots, L_{N- \mu_i^{(k)}+1} ) 
		- \Phi_{s+1}^+ (L_N, L_{N-1}, \dots, L_{N- \mu_i^{(k)}+1}) \big) 
	\\ & \quad \times 
	L_{N+1}^t [T; N, \mu_{i+1}^{(k)}]^+. 
\end{align*}
By \eqref{recursive relation of Phi 2}, we have 
\begin{align*}
&(\CI_{(i,k),s+1}^+ \CX_{(i,k),t}^+ - \CX_{(i,k),t}^+ \CI_{(i,k),s+1}^+)(m_{\mu}) 
\\
&=   q^{ s - \mu_{i+1}^{(k)} +1} m_{\mu+\a_{(i,k)}} 
	\\ & \quad \times 
	L_{N+1} \big( \Phi_s^+ (L_{N+1}, L_N, \dots, L_{N-\mu_i^{(k)}+1}) 
				- q^{-2} \Phi_s^+ (L_N, L_{N-1}, \dots, L_{N- \mu_i^{(k)}+1}) \big) 
	\\ & \quad \times 
	L_{N+1}^t [T; N, \mu_{i+1}^{(k)}]^+
\\
&=  q^{ (s-1) - \mu_{i+1}^{(k)} +1} m_{\mu+\a_{(i,k)}} 
	\\ & \quad \times \Big\{ 
	q  \Phi_s^+ (L_{N+1}, L_N, \dots, L_{N-\mu_i^{(k)}+1}) 
	 L_{N+1}^{t+1} [T; N, \mu_{i+1}^{(k)}]^+
	 \\ & \hspace{3em} 
	 - q^{-1} L_{N+1}^{t+1} [T; N, \mu_{i+1}^{(k)}]^+
	 \Phi_s^+ (L_N, L_{N-1}, \dots, L_{N- \mu_i^{(k)}+1})\Big\} 
\\
&=( q \CI_{(i,k),s}^+ \CX_{(i,k),t+1}^+ - q^{-1} \CX_{(i,k),t+1}^+ \CI_{(i,k),s}^+ )(m_{\mu}).
\end{align*}
Now we proved \eqref{CI+ CX+ - CX+ CI+}. 
Other cases are proven in a similar way.
\end{proof}

%%%%%%%%%%%%%%%%%%%%%%%%%%%%%%%%%%%%%%%%%%%%%%%%%%%%%%%%%%%

\begin{prop}
\label{Proposition CX+ CX- - CX- CX+ (i,k) (j,l)}
For $(i,k),(j,l) \in \vG'(\Bm)$ such that $(i,k) \not= (j,l)$ and $s,t \geq 0$, 
we have 
\begin{align*}
[ \CX_{(i,k),t}^+, \CX_{(j,l),s}^-]=0.
\end{align*}
\end{prop}

\begin{proof} 
By Lemma \ref{Lemma Xpmt}, 
for $\mu \in \vL_{n,r}(\Bm)$, 
we have 
\begin{align*}
&\CX_{(i,k),t}^+ \CX_{(j,l),s}^- (m_\mu) 
\\
&=  q^{- \mu_j^{(l)} - (\mu - \a_{(j,l)})_{i+1}^{(k)} +2}   m_{\mu + \a_{(i,k)} - \a_{(j,l)} }  
	\\  & \quad \times 
	L^t_{N_{(i,k)}^{\mu- \a_{(j,l)}} +1} [T; N_{(i,k)}^{\mu-\a_{(j,l)}}, (\mu-\a_{(j,l)})_{i+1}^{(k)}]^+ 
	L^s_{N^\mu_{(j,l)}} h^\mu_{-(j,l)} [T; N_{(j,l)}^\mu, \mu_j^{(l)}]^- 
\end{align*}
and 
\begin{align*}
&\CX_{(j,l),s}^- \CX_{(i,k),t}^+ (m_{\mu}) 
\\
&= q^{ - \mu_{i+1}^{(k)} - (\mu + \a_{(i,k)})_{j}^{(l)} +2}  m_{\mu + \a_{(i,k)} - \a_{(j,l)}} 
	\\ & \quad \times 
	L^s_{N^{\mu +\a_{(i,k)}}_{(j,l)}} h^{\mu + \a_{(i,k)}}_{-(j,l)}  [T; N_{(j,l)}^{\mu + \a_{(i,k)}}, (\mu+\a_{(i,k)})_j^{(l)}]^-  
	L^t_{N^\mu_{(i,k)}+1} [T; N_{(i,k)}^\mu, \mu_{i+1}^{(k)}]^+.
\end{align*}
Since $(i,k) \not= (j,l)$, we have 
\begin{align*}
& N^\mu_{(i,k)} = N^{\mu - \a_{(j,l)}}_{(i,k)}, \quad N^\mu_{(j,l)} = N^{\mu + \a_{(i,k)}}_{(j,l)}, 
\\
& (\mu - \a_{(j,l)})_{i+1}^{(k)} = 
	\begin{cases} \mu_{i+1}^{(k)} & \text{ if } (j,l) \not=(i+1,k), \\ \mu_{i+1}^{(k)} -1 & \text{ if } (j,l)=(i+1,k), \end{cases} 
\\
& (\mu + \a_{(i,k)})_{j}^{(l)} = 
	\begin{cases} \mu_{j}^{(l)} & \text{ if } (j,l) \not=(i+1,k), \\ \mu_{j}^{(l)} -1 & \text{ if } (j,l)=(i+1,k).  \end{cases} 
\\
& h^{\mu}_{-(j,l)} = h^{\mu + \a_{(i,k)}}_{-(j,l)} = 
	\begin{cases} 
	1 & \text{ if } j \not= m_j, 
	\\
	L_{N^{\mu}_{(m_l,l)}} - Q_l & \text{ if } j = m_l. 
	\end{cases}
\end{align*}
Then, by Lemma \ref{Lemma Li commute [T; N mu]}, 
we have  
\begin{align*}
[T; N_{(i,k)}^{\mu-\a_{(j,l)}}, (\mu-\a_{(j,l)})_{i+1}^{(k)}]^+  L^s_{N^\mu_{(j,l)}} h^\mu_{-(j,l)}
= L^s_{N^\mu_{(j,l)}} h^\mu_{-(j,l)} [T; N_{(i,k)}^{\mu-\a_{(j,l)}}, (\mu-\a_{(j,l)})_{i+1}^{(k)}]^+ 
\end{align*}
and 
\begin{align*}
[T; N_{(j,l)}^{\mu + \a_{(i,k)}}, (\mu+\a_{(i,k)})_j^{(l)}]^-  L^t_{N^\mu_{(i,k)}+1} 
= L^t_{N^\mu_{(i,k)}+1} [T; N_{(j,l)}^{\mu + \a_{(i,k)}}, (\mu+\a_{(i,k)})_j^{(l)}]^-. 
\end{align*}
Thus, in order to prove the proposition, 
it is enough to show that 
\begin{align}
\label{TT = TT}
\begin{split} 
&[T; N_{(i,k)}^{\mu-\a_{(j,l)}}, (\mu-\a_{(j,l)})_{i+1}^{(k)}]^+ [T; N_{(j,l)}^\mu, \mu_j^{(l)}]^- 
\\
&= [T; N_{(j,l)}^{\mu + \a_{(i,k)}}, (\mu+\a_{(i,k)})_j^{(l)}]^-  [T; N_{(i,k)}^\mu, \mu_{i+1}^{(k)}]^+. 
\end{split}
\end{align}
If $(j,l) \not=(i+1,k)$, we see easily that \eqref{TT = TT} holds since the product is cummutative in each side. 
In the case where $(j,l)=(i+1,k)$, we can prove that \eqref{TT = TT} by the induction on $\mu_{i+1}^{(k)}$. 
Now we proved the proposition. 
\end{proof}

\remark 
There is an error in the proof of \cite[Proposition 6.11 (\roi)]{W} (see the case where $(j,l) = (i+1,k)$). 
The above proof also gives a fixed proof of \cite[Proposition 6.11 (\roi)]{W} as a special case. 
\\
%%%%%%%%%%%%%%%%%%%%%%%%%%%%%%%%%%%%%%%%%%%%%%%%%%%%%%%%%%%%%%%

We prepare some technical lemmas. 

% #####

\begin{lem}
\label{Lemma m mu L T}
For $\mu \in \vL_{n,r}(\Bm)$ and $(i,k) \in \vG(\Bm)$, we have the followings. 
\begin{enumerate}
\item 
For $t \geq 0$ and $1 \leq p \leq \mu_i^{(k)}$, we have 
\begin{align*}
m_{\mu} L_{N^\mu_{(i,k)}}^t [T; N_{(i,k)}^\mu, p]^- 
= q^{ 2 p -2} m_{\mu} \Phi^+_t (L_{N^\mu_{(i,k)}}, L_{N^\mu_{(i,k)}-1}, \dots, L_{N^\mu_{(i,k)}-p+1}). 
\end{align*}

\item 
For $t \geq 0$ and $1 \leq p \leq \mu_{i+1}^{(k)}$, we have  
\begin{align*}
m_{\mu} L_{N^\mu_{(i,k)}+1}^t [T; N_{(i,k)}^\mu, p]^+ 
= m_{\mu} \Phi^-_t  (L_{N^\mu_{(i,k)}+1}, L_{N^\mu_{(i,k)}+2}, \dots, L_{N^\mu_{(i,k)}+p}). 
\end{align*}
\end{enumerate}
\end{lem}

\begin{proof} 
In the case where $t=0$, 
we have (\roi) and (\roii) from \eqref{ m mu T}. 

We prove (\roi) for $t>0$.  
Put $N= N_{(i,k)}^{\mu}$. 
For $1 \leq h \leq \mu_i^{(k)}-1$, 
by the induction on $h$ together with Lemma \ref{Lemma commute L T} and \eqref{ m mu T}, 
we can show that 
\begin{align}
\label{m_mu L_N^t (T_N-1 dots T_N-h)}
\begin{split}
&m_\mu L_N^t (T_{N-1} T_{N-2} \dots T_{N-h}) 
\\
&= m_\mu \big\{ 
	(q-q^{-1}) q^{h-1} L_N^t 
	+ \sum_{s=2}^h (q-q^{-1}) q^{h-s} L_N^{t-1} (T_{N-1} T_{N-2} \dots T_{N-s+1}) L_{N-s+1} 
	\\
	& \hspace{3em} 
	+ L_N^{t-1} (T_{N-1} T_{N-2} \dots T_{N-h}) L_{N-h} \big\}. 
\end{split}
\end{align}
We prove that 
\begin{align}
\label{m_mu L_N^t ( T_N-1 dots T_N-h) by Phi}
\begin{split}
&m_{\mu} L_N^t (T_{N-1} T_{N-2} \dots T_{N-h}) 
\\
&= m_{\mu} \big( q^h \Phi_t^+ (L_N, L_{N-1}, \dots, L_{N-h}) 
				- q^{h-2} \Phi_t^+ (L_N,L_{N-1}, \dots, L_{N-h+1}) \big) 
\end{split}
\end{align}
by the induction on $t$. 
In the case where $t=1$, 
by \eqref{m_mu L_N^t (T_N-1 dots T_N-h)} together with \eqref{ m mu T}, we have 
\begin{align*}
&m_{\mu} L_N (T_{N-1} T_{N-2} \dots T_{N-h}) 
\\
&= m_{\mu} \big\{ (q-q^{-1})q^{h-1} L_{N} + \sum_{s=2}^h (q-q^{-1}) q^{h-s} q^{s-1} L_{N-s+1} + q^h L_{N-h} \big\} 
\\
&= m_\mu \big( q^h \Phi_1^+(L_N,L_{N-1}, \dots, L_{N-h}) - q^{h-2} \Phi_1^+ (L_N, L_{N-1}, \dots, L_{N-h+1}) \big). 
\end{align*}
Assume that $t >1$. 
Applying the assumption of the induction to \eqref{m_mu L_N^t (T_N-1 dots T_N-h)}, 
we have 
\begin{align*}
&m_\mu L_N^t (T_{N-1}T_{N-2} \dots T_{N-h}) 
\\
&= m_{\mu} \big\{ (q-q^{-1}) q^{h-1} L_N^t 
	\\ & \hspace{3em} 
	+ \sum_{s=2}^h (q-q^{-1}) q^{h-s} 
		\big( q^{s-1} \Phi_{t-1}^+(L_N,L_{N-1}, \dots, L_{N-s+1})  
		\\ & \hspace{13em} - q^{s-3} \Phi_{t-1}^+ (L_N,L_{N-1}, \dots, L_{N-s+2}) \big) L_{N-s+1} 
	\\ & \hspace{3em} 
	+ \big( q^h \Phi_{t-1}^+ (L_N, L_{N-1}, \dots, L_{N-h}) - q^{h-2} \Phi_{t-1}^+ (L_N, L_{N-1}, \dots, L_{N-h+1}) \big) L_{N-h} 
	\big\} 
\end{align*}
Put $s'=s-1$, we have 
\begin{align*}
&m_\mu L_N^t (T_{N-1}T_{N-2} \dots T_{N-h}) 
\\
&= m_{\mu} \Big\{ 
	q^h \Big( L_N^t + \sum_{s=1}^h \big( \Phi_{t-1}^+ (L_N,L_{N-1}, \dots, L_{N-s'}) L_{N-s'} 
		\\ & \hspace{10em} 
			- q^{-2} \Phi_{t-1}(L_N, L_{N-1}, \dots, L_{N-s'+1}) L_{N-s'} \big) \Big) 
	\\ 
	& \hspace{4em} 
	- q^{h-2} \Big( L_N^t + \sum_{s=1}^{h-1} \big( \Phi_{t-1}^+ (L_N,L_{N-1}, \dots, L_{N-s'}) L_{N-s'} 
		\\ & \hspace{13em} 
			- q^{-2} \Phi_{t-1}(L_N, L_{N-1}, \dots, L_{N-s'+1}) L_{N-s'} \big) \Big) \Big\}. 
\end{align*}
Applying \eqref{recursive relation of Phi} to the right-hand side, we have \eqref{m_mu L_N^t ( T_N-1 dots T_N-h) by Phi}. 
Thanks to \eqref{m_mu L_N^t ( T_N-1 dots T_N-h) by Phi}, we have 
\begin{align*}
&m_{\mu} L_{N}^t [T; N, p]^- 
\\
&=m_\mu L_N^t (1 + \sum_{h=1}^{p-1} q^h T_{N-1}T_{N-2} \dots T_{N-h}) 
\\
&= m_\mu \big\{ 
	\Phi_t^+(L_N) 
	+ \sum_{h=1}^{p-1} \big( q^{2 h} \Phi_{t}^+(L_{N},L_{N-1},\dots, L_{N-h}) 	
		\\ & \hspace{12em} 					
		- q^{ 2 h -2} \Phi_t^+ (L_N,L_{N-1},\dots, L_{N-h+1}) \big) 	\big\} 
\\
&= q^{2p-2} m_\mu \Phi_t^+ (L_N,L_{N-1}, \dots, L_{N-p}). 
\end{align*}
Now we obtained (\roi). 

%%%%%%

For $t>0$ and $1 \leq h \leq \mu_{i+1}^{(k)} -1$, 
by the induction on $h$ using Lemma \ref{Lemma commute L T} and \eqref{ m mu T},  
we can show that 
\begin{align}
\label{m_mu L_n+1^t T_N+1 dots T_N+h} 
\begin{split}
&m_{\mu} L_{N+1}^t (T_{N+1} T_{N+2} \dots T_{N+h}) 
\\
&= q^{-h} m_{\mu} L_{N+1}^{t-1} 
	\big\{ (1-q^2) \big( 1 + \sum_{s=1}^{h-1} q^s T_{N+1} T_{N+2} \dots T_{N+s} \big) 
		\\ & \hspace{8em} 
			+ q^h T_{N+1} T_{N+2} \dots T_{N+h} 
	\big\} L_{N+h+1}. 
\end{split}
\end{align}

We prove (\roii) by the induction on $t$. 
We have already proved (\roii) in the case where $t=0$. 

Assume that $t > 0$. 
By \eqref{m_mu L_n+1^t T_N+1 dots T_N+h}, we have 
\begin{align*}
&m_{\mu} L_{N +1}^t [T; N, p]^+ 
\\
&=m_{\mu} L_{N+1}^t \big( 1 + \sum_{h=1}^{p-1} q^h T_{N+1} T_{N+2} \dots T_{N+h} \big) 
\\
&= m_{\mu} L_{N+1}^{t-1} 
	\Big\{ L_{N+1} 
		+ \sum_{h=1}^{p-1} \big\{ (1-q^2) (1+ \sum_{s=1}^{h-1} q^s T_{N+1} T_{N+2} \dots T_{N+s}) 
			\\ & \hspace{12em} 
			+ q^h T_{N+1} T_{N+2} \dots T_{N+h} \big\} L_{N+h+1} \Big\} 
\\
&= m_{\mu} L_{N+1}^{t-1} \Big\{ 
	\sum_{h=1}^{p} [T;N,h]^+  L_{N+h} 
	- q^2 \sum_{h=1}^{p-1} [T;N,h]^+  L_{N+h+1} \Big\}. 
\end{align*}
Applying the assumption of the induction, we have 
\begin{align*}
&m_{\mu} L_{N +1}^t [T; N, p]^+ 
\\
&=  m_{\mu} \Big\{ \sum_{h=1}^{p} \Phi_{t-1}^- (L_{N+1}, L_{N+2}, \dots , L_{N+h}) L_{N+h} 
	\\ & \hspace{5em} 
	- q^2 \sum_{h=1}^{p-1} \Phi_{t-1}^- (L_{N+1}, L_{N+2}, \dots, L_{N+h}) L_{N+h+1} \Big\}. 
\end{align*}
Applying \eqref{recursive relation of Phi}, we have 
\begin{align*}
m_{\mu} L_{N +1}^t [T; N, p]^+ 
=m_{\mu} \Phi_t^- (L_{N+1}, L_{N+2}, \dots, L_{N+p}). 
\end{align*}
\end{proof}

% #####
\begin{lem}
\label{Lemma m mu L T etc}
For $\mu \in \vL_{n,r}(\Bm)$ and $(i,k) \in \vG'(\Bm)$, put $N=N_{(i,k)}^{\mu}$. Then we have the followings. 
\begin{enumerate}
\item  
If $\mu_i^{(k)} \not=0$, we have 
\begin{align*}
& m_{\mu } L_N^t [T; N-1, \mu_{i+1}^{(k)}+1]^+ [T; N, \mu_i^{(k)}]^- 
\\
&= q^{2 \mu_i^{(k)} - 2} m_{\mu} \Phi_t^+ (L_N, L_{N-1}, \dots, L_{N- \mu_i^{(k)} +1}) 
	 \\ & \quad 
	+ \d_{(\mu_{i+1}^{(k)} \not=0)} m_{\mu} L_N^t \big( [T; N+1, \mu_{i}^{(k)}+1]^- -1 \big) [T; N, \mu_{i+1}^{(k)}]^+ 
\end{align*}

\item 
If $\mu_i^{(k)} \not=0$, we have 
\begin{align*}
&m_{\mu} L_N^t [T; N-1, \mu_{i+1}^{(k)}+1]^+  L_N [T; N, \mu_i^{(k)}]^- 
\\
&= q^{2 \mu_{i}^{(k)} -2} m_{\mu} \Phi_{t+1}^+ (L_{N}, L_{N-1}, \dots, L_{N-\mu_i^{(k)} +1})
	\\ & \hspace{1em} 
	- \d_{(\mu_{i+1}^{(k)} \not=0)} (q-q^{-1}) q^{2 \mu_i^{(k)} -1}  
		m_{\mu} \Phi_{t}^+ (L_{N}, L_{N-1}, \dots, L_{N-\mu_i^{(k)} +1}) 
		\\ & \hspace{3em} \times 
			\Phi_{1}^- (L_{N+1}, L_{N+2}, \dots, L_{N+\mu_{i+1}^{(k)}})
	\\ & \hspace{1em} 
	+ m_\mu  L_N^t L_{N+1} \big( [T; N-1, \mu_{i+1}^{(k)}+1]^+ -1 \big) [T; N, \mu_i^{(k)}]^- 
\end{align*}

\item 
If $\mu_{i+1}^{(k)} \not=0$, we have 
\begin{align*}
&m_{\mu} [T; N+1, \mu_i^{(k)}+1]^- L_{N+1}^t [T; N, \mu_{i+1}^{(k)}]^+ 
\\
&= (1 + \d_{(t \not=0)} (q^{ 2 \mu_i^{(k)}} -1) )  m_{\mu} \Phi_t^- (L_{N+1}, L_{N+2}, \dots, L_{N+\mu_{i+1}^{(k)}}) 
	\\ & \hspace{1em} 
	+ \d_{(\mu_i^{(k)} \not=0)} (q-q^{-1}) \sum_{b=1}^{t-1} m_{\mu}  q^{2 \mu_i^{(k)}-1} 
		\Phi_{t-b}^+ (L_N, L_{N-1}, \dots, L_{N- \mu_i^{(k)}+1}) 
		\\ & \hspace{3em} \times 
		\Phi_b^- (L_{N+1}, L_{N+2}, \dots, L_{N+\mu_{i+1}^{(k)}})
	\\ & \hspace{1em} 
	+ m_{\mu} L_N^t \big( [T; N+1, \mu_i^{(k)}+1]^- -1 \big) [T; N, \mu_{i+1}^{(k)}]^+ 
\end{align*}

\item 
If $\mu_{i+1}^{(k)} \not=0$, we have 
\begin{align*}
&m_{\mu} L_{N+1} [T; N+1, \mu_i^{(k)}+1]^-  L_{N+1}^t [T; N, \mu_{i+1}^{(k)}]^+  
\\
&= (1 + \d_{(t \not=0)} (q^{ 2 \mu_i^{(k)}} -1) ) 
	 m_{\mu} \Phi_{t+1}^- (L_{N+1}, L_{N+2}, \dots, L_{N+\mu_{i+1}^{(k)}}) 
	\\ & \hspace{1em} 
	+ \d_{(\mu_i^{(k)} \not=0)} (q-q^{-1}) \sum_{b=1}^{t-1} m_{\mu}  q^{2 \mu_i^{(k)}-1} 
		\Phi_{t-b}^+ (L_N, L_{N-1}, \dots, L_{N- \mu_i^{(k)}+1}) 
		\\ & \hspace{3em} \times 
		\Phi_{b+1}^- (L_{N+1}, L_{N+2}, \dots, L_{N+\mu_{i+1}^{(k)}})
	\\ & \hspace{1em} 
	+ m_{\mu} L_N^t L_{N+1} \big( [T; N+1, \mu_i^{(k)}+1]^- -1 \big) [ T; N, \mu_{i+1}^{(k)}]^+ 
\end{align*}
\end{enumerate} 
\end{lem}

\begin{proof} 
By the induction on $\mu_{i+1}^{(k)}$, we can prove that 
\begin{align}
\label{[T;N-1,mu]+[T;N,mu]-}
\begin{split} 
&[T; N-1, \mu_{i+1}^{(k)}+1]^+ [T; N, \mu_i^{(k)}]^-  
\\
&=  [T; N, \mu_i^{(k)}]^- 
	+ \d_{(\mu_{i+1}^{(k)} \not=0)} \big( [T; N+1, \mu_i^{(k)}+1]^- -1 \big) [T; N, \mu_{i+1}^{(k)}]^+
\end{split} 
\end{align}

Thus we have 
\begin{align*}
& m_{\mu } L_N^t [T; N-1, \mu_{i+1}^{(k)}+1]^+ [T; N, \mu_i^{(k)}]^- 
\\	
& =  m_{\mu } L_N^t \Big\{ [T; N, \mu_i^{(k)}]^- 
	+ \d_{(\mu_{i+1}^{(k)} \not=0)} \big( [T; N+1, \mu_i^{(k)}+1]^- -1 \big) [T; N, \mu_{i+1}^{(k)}]^+ \Big\}. 
\end{align*}
Applying Lemma \ref{Lemma m mu L T} (\roi), we have (\roi). 

% #####

We prove (\roii). 
By Lemma \ref{Lemma commute L T}, we have 
\begin{align*}
&[T; N-1, \mu_{i+1}^{(k)}+1]^+  L_N 
\\
&= L_N + L_{N+1} \big( [T; N-1, \mu_{i+1}^{(k)} +1]^+ -1 \big) 
	- \d_{(\mu_{i+1}^{(k)} \not=0)} q (q-q^{-1}) L_{N+1} [T; N, \mu_{i+1}^{(k)}]^+. 
\end{align*}
Thus, we have 
\begin{align*}
&m_{\mu} L_N^t  [T; N-1, \mu_{i+1}^{(k)}+1]^+  L_N [T; N, \mu_i^{(k)}]^- 
\\
&= m_{\mu}  L_N^{t+1} [T; N, \mu_i^{(k)}]^- 
	+ m_\mu  L_N^t L_{N+1} \big( [T; N-1, \mu_{i+1}^{(k)} +1]^+ -1 \big)  [T; N, \mu_i^{(k)}]^- 
	\\ & \hspace{1em} 
	- \d_{(\mu_{i+1}^{(k)} \not=0)} q (q-q^{-1}) m_{\mu} L_N^tL_{N+1} 
	[T; N, \mu_{i+1}^{(k)}]^+ [T; N, \mu_i^{(k)}]^-. 
\end{align*}
Applying \eqref{ m mu T}, Lemma \ref{Lemma Li commute [T; N mu]}, 
Lemma \ref{Lemma m mu L T} and \eqref{[T;N-1,mu]+[T;N,mu]-}, 
we have (\roii).

We prove (\roiii). 
By Lemma \ref{Lemma commute L T}, we have 
\begin{align*}
 [T; N+1, \mu_i^{(k)}+1]^-  L_{N+1}^t 
&=  L_{N+1}^t + L_N^t \big( [T; N+1, \mu_i^{(k)}+1]^- -1 \big) 
	\\ & \quad 
	+ \d_{(\mu_i^{(k)} \not=0)} q (q-q^{-1}) \sum_{b=1}^t L_N^{t-b} L_{N+1}^b [T; N,\mu_i^{(k)}]^-. 
\end{align*}

Thus, we have 
\begin{align*}
&m_{\mu} [T; N+1, \mu_i^{(k)}+1]^- L_{N+1}^t [T; N, \mu_{i+1}^{(k)}]^+ 
\\
&= m_{\mu} L_{N+1}^t  [T; N, \mu_{i+1}^{(k)}]^+ 
	+ m_{\mu} L_N^t  \big( [T; N+1, \mu_i^{(k)}+1]^- -1 \big) [T; N, \mu_{i+1}^{(k)}]^+ 
	\\ & \hspace{1em} 
	+ \d_{(\mu_i^{(k)} \not=0)} q (q-q^{-1}) \sum_{b=1}^t m_{\mu}  L_N^{t-b} L_{N+1}^b 
		[T; N,\mu_i^{(k)}]^-  [T; N, \mu_{i+1}^{(k)}]^+ 
\\
&= m_{\mu} L_{N+1}^t [T; N, \mu_{i+1}^{(k)}]^+ 
	\\ & \hspace{1em} 
	+ \d_{(\mu_{i}^{(k)} \not=0)} \d_{(t \not=0)} q (q-q^{-1}) m_{\mu} L_{N+1}^t 
		[T; N, \mu_i^{(k)}]^- [T; N, \mu_{i+1}^{(k)}]^+ 
	\\ & \hspace{1em} 
	+ \d_{(\mu_i^{(k)} \not=0)} q (q-q^{-1}) \sum_{b=1}^{t-1} m_{\mu}  L_N^{t-b} L_{N+1}^b 
		[T; N, \mu_i^{(k)}]^- [T; N, \mu_{i+1}^{(k)}]^+ 
	\\ & \hspace{1em} 
	+ m_{\mu} L_N^t  \big( [T; N+1, \mu_i^{(k)}+1]^- -1 \big) [T; N, \mu_{i+1}^{(k)}]^+ 
\end{align*}
Applying \eqref{ m mu T}, Lemma \ref{Lemma Li commute [T; N mu]} 
and Lemma \ref{Lemma m mu L T}, 
we have (\roiii).

We prove (\roiv). 
By Lemma \ref{Lemma commute L T}, we have 
\begin{align*}
&m_{\mu} L_{N+1} [T; N+1, \mu_i^{(k)}+1]^- L_{N+1}^t  [T; N, \mu_{i+1}^{(k)}]^+ 
\\
&= m_{\mu} L_{N+1}^{t+1} [T; N, \mu_{i+1}^{(k)}]^+ 
	+ m_{\mu} L_N^t L_{N+1}\big( [T; N+1, \mu_i^{(k)}+1]^- -1 \big) [T; N, \mu_{i+1}^{(k)}]^+ 
	\\ & \hspace{1em} 
	+ \d_{(\mu_i^{(k)} \not=0)} q (q-q^{-1}) \sum_{b=1}^t m_{\mu}  L_N^{t-b} L_{N+1}^{b+1} 
		[T; N, \mu_i^{(k)}]^- [T; N, \mu_{i+1}^{(k)}]^+ 
\end{align*}
Applying \eqref{ m mu T}, Lemma \ref{Lemma Li commute [T; N mu]} 
and Lemma \ref{Lemma m mu L T}, 
we have (\roiv).
\end{proof}

% #####

\begin{prop}
\label{Proposition CX+ CX- - CX- CX+}
For $(i,k) \in \vG'(\Bm)$ and $s,t \geq 0$, we have 
\begin{align*}
[\CX_{(i,k),t}^+, \CX_{(i,k),s}^-] 
= \begin{cases} 
	\dis 
	\wt{\CK}_{(i,k)}^+ \CJ_{(i,k),s+t} 
		& \text{ if } i \not=m_k, 
	\\ \dis 
	- Q_k \wt{\CK}_{(m_k,k)}^+ \CJ_{(m_k,k),s+t} + \wt{\CK}_{(m_k,k)}^+ \CJ_{(m_k,k),s+t+1} 
		& \text{ if } i=m_k. 
	\end{cases} 
\end{align*}
\end{prop}

\begin{proof} 
Assume that $s=0$ and $t \geq 0$. 
For $\mu \in \vL_{n,r}(\Bm)$, put $N=N^{\mu}_{(i,k)}$.  
By Lemma \ref{Lemma Xpmt}, we have 
\begin{align}
\label{CX+t CX-0}
\begin{split} 
&\CX_{(i,k),t}^+ \CX_{(i,k),0}^- (m_{\mu}) 
\\
&=  \d_{(\mu_i^{(k)}\not=0)} q^{- \mu_i^{(k)} - \mu_{i+1}^{(k)} +1} m_{\mu } 
	L_N^t [T; N-1,\mu_{i+1}^{(k)}+1]^+  h_{-(i,k)}^{\mu} [T; N, \mu_i^{(k)}]^- 
\end{split}
\end{align}
and 
\begin{align}
\label{CX-0 CX+t} 
\begin{split}
&\CX_{(i,k),0}^- \CX_{(i,k),t}^+ (m_\mu) 
\\
&= \d_{(\mu_{i+1}^{(k)} \not=0)} q^{- \mu_i^{(k)} - \mu_{i+1}^{(k)}+1} 
		m_{\mu} h_{-(i,k)}^{\mu + \a_{(i,k)}} [T; N+1, \mu_i^{(k)}+1]^-  L_{N+1}^t [T; N, \mu_{i+1}^{(k)}]^+. 
\end{split}
\end{align}

Assume that $i\not=m_k$. 
By \eqref{CX+t CX-0}  and \eqref{CX-0 CX+t} together with Lemma \ref{Lemma m mu L T etc}, we have 
\begin{align*}
& (\CX_{(i,k),t}^+ \CX_{(i,k),0}^-  - \CX_{(i,k),0}^- \CX_{(i,k),t}^+)(m_{\mu}) 
\\
&= q^{-\mu_i^{(k)} - \mu_{i+1}^{(k)}+1} m_{\mu} \Big\{
		\d_{(\mu_i^{(k)} \not=0)} q^{ 2 \mu_i^{(k)}-2} \Phi_t^+ (L_N, L_{N-1}, \dots, L_{N-\mu_i^{(k)}+1}) 
	\\ & \hspace{2em} 
	- \d_{(\mu_{i+1}^{(k)} \not=0)} 
		(1 + \d_{(t\not=0)}( q^{2 \mu_i^{(k)}} -1)) \Phi_t^- (L_{N+1}, L_{N+2}, \dots, L_{N+\mu_{i+1}^{(k)}}) 
	\\ & \hspace{2em} 
	- \d_{(\mu_i^{(k)} \not=0)} \d_{(\mu_{i+1}^{(k)} \not=0)} (q-q^{-1}) \sum_{b=1}^{t-1} q^{ 2 \mu_i^{(k)}-1}
		 \Phi_{t-b}^+ (L_N, L_{N-1}, \dots, L_{N- \mu_i^{(k)}+1}) 
		 \\ & \hspace{3em} \times 
		 \Phi_b^- (L_{N+1}, L_{N+2}, \dots, L_{N+\mu_{i+1}^{(k)}}) \Big\} 
\\
&= q^{\mu_i^{(k)} - \mu_{i+1}^{(k)}} m_{\mu} 
	\Big\{ \d_{(\mu_i^{(k)} \not=0)} q^{-t} q^{t-1} \Phi_t^+ (L_N, L_{N-1}, \dots, L_{N- \mu_i^{(k)}+1}) 
	\\ & \hspace{2em} 
	 - \d_{(\mu_{i+1}^{(k)} \not=0)} 
	 	( q^{- 2 \mu_i^{(k)}} + \d_{(t \not=0)} ( 1 - q^{- 2 \mu_i^{(k)}}))
	 	 q^t q^{ -t +1}\Phi_t^- (L_{N+1}, L_{N+2}, \dots, L_{N+\mu_{i+1}^{(k)}}) 
	 \\ & \hspace{2em} 
	 - \d_{(\mu_i^{(k)} \not=0)} \d_{(\mu_{i+1}^{(k)} \not=0)} (q-q^{-1}) \sum_{b=1}^{t-1} q^{-t + 2 b}  
		 q^{t-b-1}\Phi_{t-b}^+ (L_N, L_{N-1}, \dots, L_{N- \mu_i^{(k)}+1}) 
		 \\ & \hspace{3em} \times 
		 q^{- b +1} \Phi_b^- (L_{N+1}, L_{N+2}, \dots, L_{N+\mu_{i+1}^{(k)}}) \Big\} 
\\
&=  \wt{\CK}_{(i,k)}^+ \CJ_{(i,k),t} (m_{\mu}). 
\end{align*}
Thus, we have 
$[\CX_{(i,k),t}^+, \CX_{(i,k),0}^-]  = \wt{\CK}_{(i,k)}^+ \CJ_{(i,k),t} $ 
if $i \not=m_k$. 
(Note Corollary \ref{Cor CJ0} in the case where $t=0$.) 

In a similar way, by \eqref{CX+t CX-0}  and \eqref{CX-0 CX+t} together with Lemma \ref{Lemma m mu L T etc},
we also have 
$[\CX_{(m_k,k),t}^+, \CX_{(m_k,k),0}^-] 
= - Q_k \wt{\CK}_{(m_k,k)}^+ \CJ_{(m_k,k),s+t} + \wt{\CK}_{(m_k,k)}^+ \CJ_{(m_k,k),s+t+1} $ 
if $i=m_k$. 
Now we proved the proposition in the case where $s=0$ and $t \geq 0$. 

Finally, we prove the proposition by the induction on $s$. 
In the case where $s=0$, 
we have already proved. 
Assume that $s >0$, 
by \eqref{Def CX t}, 
we have 
\begin{align*}
[\CX_{(i,k),t}^+, \CX_{(i,k),s}^-] 
= & \CX_{(i,k),t}^+   (-  \CI_{(i,k),1}^- \CX_{(i,k),s-1}^- + \CX_{(i,k),s-1}^- \CI_{(i,k),1}^-)   
	\\ & -   (-  \CI_{(i,k),1}^- \CX_{(i,k),s-1}^- + \CX_{(i,k),s-1}^- \CI_{(i,k),1}^-)    \CX_{(i,k),t}^+. 
\end{align*}
Applying Proposition \ref{Proposition IX - XI} together with Lemma \ref{Lemma CI 0 CX+-}, 
we have 
\begin{align*}
[\CX_{(i,k),t}^+, \CX_{(i,k),s}^-] 
&=  - \CI_{(i,k),1}^- \CX_{(i,k),t}^+ \CX_{(i,k),s-1}^- + \CX_{(i,k),t+1}^+ \CX_{(i,k),s-1}^- 
		+ \CX_{(i,k),t}^+ \CX_{(i,k),s-1}^- \CI_{(i,k),1}^-  
	\\ 
	& \quad + \CI_{(i,k),1}^- \CX_{(i,k),s-1}^- \CX_{(i,k),t}^+ 
		- \CX_{(i,k),s-1}^- \CX_{(i,k),t}^+ \CI_{(i,k),1}^- - \CX_{(i,k),s-1}^- \CX_{(i,k),t+1}^+ 
\\
&= [\CX_{(i,k),t+1}^+, \CX_{(i,k),s-1}^-] 
	\\ & \quad 
	- \CI_{(i,k),1}^- [ \CX_{(i,k),t}^+, \CX_{(i,k),s-1}^-] 
	+ [ \CX_{(i,k),t}^+, \CX_{(i,k),s-1}^-]  \CI_{(i,k),1}^-. 
\end{align*}
Then, by the assumption of the induction together with Lemma \ref{Lemma commute K H}, 
we have the proposition. 
\end{proof}

%%%%%%%%%%%%%%%%%%%%%%%%%%%%%%%

\begin{lem}
\label{Lemma wtKJ0}
For $(i,k) \in \vG'(\Bm)$, we have the followings. 
\begin{enumerate}
\item 
If $(q-q^{-1})$ is invertible in $R$, we have 
\begin{align*}
\wt{\CK}_{(i,k)}^+ \CJ_{(i,k),0} 
= \frac{ \wt{\CK}_{(i,k)}^+ - \wt{\CK}_{(i,k)}^-}{q-q^{-1}}. 
\end{align*}

\item 
If $q=1$, we have 
\begin{align*}
\wt{\CK}_{(i,k)}^+ \CJ_{(i,k),0} 
	= \CI_{(i,k),0}^+ - \CI_{(i+1,k),0}^-. 
\end{align*}
\end{enumerate}
\end{lem}

\begin{proof}
For $\mu \in \vL_{n,r}(\Bm)$, by the definitions together with Corollary \ref{Cor CJ0}, 
we have 
\begin{align*}
\wt{\CK}_{(i,k)}^+ \CJ_{(i,k),0} (m_{\mu}) 
&= \wt{\CK}_{(i,k)}^+ \big( \CI_{(i,k),0}^+ - (\CK_{(i,k)}^-)^2 \CI_{(i+1,k),0}^- \big) (m_{\mu}) 
\\
&= q^{\mu_i^{(k)} - \mu_{i+1}^{(k)}} 
	( q^{- \mu_i^{(k)}} [ \mu_i^{(k)}] - q^{ - 2 \mu_i^{(k)}} q^{\mu_{i+1}^{(k)}} [\mu_{i+1}^{(k)}])  m_{\mu} 
\\
&= [ \mu_i^{(k)} - \mu_{i+1}^{(k)}] m_{\mu}.
\end{align*}
If $(q-q^{-1})$ is invertible in $R$, we have 
\begin{align*}
[\mu_{i}^{(k)} - \mu_{i+1}^{(k)}] m_{\mu} 
&= \frac{q^{\mu_{i}^{(k)} - \mu_{i+1}^{(k)}} - q^{ - \mu_{i}^{(k)} + \mu_{i+1}^{(k)}}}{q-q^{-1}} m_{\mu} 
\\
&= \frac{\wt{\CK}_{(i,k)}^+ - \wt{\CK}_{(i,k)}^-}{q-q^{-1}} (m_{\mu}).
\end{align*}
Thus, we have (\roi). 

If $q=1$, we have 
\begin{align*}
[\mu_{i}^{(k)} - \mu_{i+1}^{(k)}] m_{\mu}  
&= (\mu_i^{(k)} - \mu_{i+1}^{(k)})  m_{\mu} 
\\
&= (\CI_{(i,k),0}^+ - \CI_{(i+1),0}^-)(m_{\mu}). 
\end{align*}
Thus, we have (\roii). 
\end{proof}
%%%%%%%%%%%%%%%%%%%%%%%%%%%%%%%%%%%%%%%%%%%%%%%%%%%%%%%%%%%%%%%%%%%%

%%%%%%%%%%%%%%%%%%%%%%%%%%%%%%%%%%%%%%%%%%%%%%%%%%%%%%%%%%%%%%%%%%%%

In the case where  $q=1$, we have the following lemma. 
\begin{lem}
\label{Lemma K I at q=1}
Assume that $q=1$. Then, for $(j,l) \in \vG(\Bm)$ and $t \geq 0$,   we have the followings. 
\begin{enumerate}
\item 
$\CK_{(j,l)}^{\pm} =1$. 

\item 
$\CI_{(j,l),t}^+ = \CI_{(j,l),t}^-$. 
\end{enumerate}
\end{lem}

\begin{proof}
If $q=1$, we see that 
\begin{align}
\Phi_t^{\pm} (x_1, \dots, x_k)  =	 x_1^t + x_2^t + \dots + x_k^t, 
\end{align}
in particular we have $\Phi_t^+(x_1,\dots, x_k)= \Phi_t^-(x_1,\dots,x_k)$. 
Thus, we have the lemma from the definitions. 
\end{proof}

%%%%%%%%%%%%%%%%%%%%%%%%%%%%%%%%%%%%%%%%%%%%%%%%%%%%%%%%%%%

%%%%%%%%%%%%%%%%%%%%%%%%%%%%%%%%%%%%%%%%%%%%%%%%%%%%%%%%%%%

\section{The cyclotomic $q$-Schur algebra as a quotient of $\CU_{q,\BQ}(\Bm)$}

Let $\wt{\BQ}=(Q_0,Q_1,\dots,Q_{r-1})$ be an $r$-tuple of indeterminate elements over $\ZZ$, 
and $\QQ(\wt{\BQ})=\QQ(Q_0,Q_1, \dots,Q_{r-1})$ be the quotient field of 
$\ZZ[\wt{\BQ}]= \ZZ[Q_0,Q_1,\dots,Q_{r-1}]$. 
Put $\wt{\AA}=\ZZ[q,q^{-1}, Q_0,Q_1,\dots,Q_{r-1}]$, 
and let $\wt{\KK} =\QQ(q,Q_0,Q_1,\dots,Q_{r-1})$ be the quotient field of $\wt{\AA}$, 
where $q$ is indeterminate over $\ZZ$. 
Put 
\begin{align*}
&\Fg_{\wt{\BQ}}(\Bm) = \QQ(\wt{\BQ}) \otimes_{\QQ(\BQ)} \Fg_{\BQ}(\Bm), 
\\
&\CU_{q,\wt{\BQ}}(\Bm) = \wt{\KK} \otimes_{\KK} \CU_{q,\BQ}(\Bm) 
\text{ and }
\CU_{\wt{\AA}, q , \wt{\BQ}}(\Bm) = \wt{\AA} \otimes_{\AA} \CU_{\AA, q, \BQ}(\Bm).
\end{align*} 
We define a full subcategory $\ZC_{\wt{\BQ}}(\Bm)$ and $\ZC_{\wt{\BQ}}^{\geq 0} (\Bm)$ 
(resp. $\ZC_{q,\wt{\BQ}}(\Bm)$ and $\ZC_{q,\wt{\BQ}}^{\geq 0}(\Bm)$) 
of $U(\Fg_{\wt{\BQ}}(\Bm)) \cmod$ (resp. $\CU_{q, \wt{\BQ}}(\Bm) \cmod$) 
in a similar manner as $\ZC_{\BQ}(\Bm)$ and $\ZC_{\BQ}^{\geq 0}(\Bm)$ 
(resp. $\ZC_{q, \BQ}(\Bm)$ and $\ZC_{q, \BQ}^{\geq 0}(\Bm)$). 

Let $\He_{n,r}^{\wt{\KK}}$ (resp. $\He_{n,r}^{\wt{\AA}}$) 
be the Ariki-Koike algebra over $\wt{\KK}$ (resp. over $\wt{\AA}$) 
with parameters $q,Q_0,Q_1, \dots, Q_{r-1}$, 
and  
$\Sc_{n,r}^{\wt{\KK}}(\Bm)$ (resp. $\Sc_{n,r}^{\wt{\AA}} (\Bm)$) 
be the cyclotomic $q$-Schur algebra 
associated with 
$\He_{n,r}^{\wt{\KK}}$ (resp. $\He_{n,r}^{\wt{\AA}}$). 
Then, we have the following theorem.
\begin{thm}
\label{Thm UqQ to Sc}
We have a homomorphism of algebras 
\begin{align}
\label{surjection U to S}
\Psi : \CU_{q, \wt{\BQ}}(\Bm) \ra \Sc_{n,r}^{\wt{\KK}}(\Bm)
\end{align}
by taking 
$\Psi (\CX_{(i,k),t}^{\pm})=\CX_{(i,k),t}^{\pm}$,  
$\Psi(\CI_{(j,l),t}^{\pm})=\CI_{(j,l),t}^{\pm}$ 
and 
$\Psi(\CK_{(j,l)}^{\pm}) = \CK_{(j,l)}^{\pm}$. 

The restriction of $\Psi$ to $\CU_{\wt{\AA}, q, \wt{\BQ}}(\Bm)$ 
gives a homomorphism of algebras 
\begin{align*}
\Psi_{\wt{\AA}} : \CU_{\wt{\AA}, q, \wt{\BQ}}(\Bm) \ra \Sc_{n,r}^{\wt\AA}(\Bm). 
\end{align*}

Moreover, if $m_k \geq n$ for all $k=1,2,\dots, r-1$, 
the homomorphism $\Psi$ (resp. $\Psi_{\wt\AA}$) is surjective. 
\end{thm}
\begin{proof}
The well-definedness of $\Psi$ follows from 
Lemma \ref{Lemma commute K H}, 
Lemma \ref{Lemma CK+- 2}, 
Lemma \ref{Lemma K+ X K-}, 
Proposition \ref{Proposition CX CX}, 
Proposition \ref{Prop q-Serre relations}, 
Lemma \ref{Lemma CI 0 CX+-}, 
Proposition \ref{Proposition IX - XI}, 
Proposition \ref{Proposition CX+ CX- - CX- CX+ (i,k) (j,l)}, 
and 
Proposition \ref{Proposition CX+ CX- - CX- CX+}. 

Note that $\He_{n,r}^{\AA}$ (resp. $\Sc_{n,r}^{\wt{\AA}}(\Bm)$) 
is an $\AA$-subalgebra of $\He_{n,r}^{\wt{\KK}}$ (resp. $\Sc_{n,r}^{\wt{\KK}}(\Bm)$) 
by definitions. 
In particular, 
in order to see that $\vf \in \Sc_{n,r}^{\wt{\KK}}(\Bm)$ belong to $\Sc_{n,r}^{\wt{\AA}}(\Bm)$, 
it is enough to show that 
$\vf (m_{\mu}) \in \He_{n,r}^{\AA}$ for any $\mu \in \vL_{n,r}(\Bm)$.

For $\mu \in \vL_{n,r}(\Bm)$ and $d \in \ZZ_{\geq 0}$, 
we see  that, 
\begin{align} 
\left[ \begin{matrix} \CK_{(j,l)} ; 0 \\ d \end{matrix} \right] (m_{\mu}) 
= \begin{cases} 
	\left[ \begin{matrix} \mu_j^{(l)} \\ d \end{matrix} \right] m_{\mu} 
	& \text{ if } d \leq \mu_j^{(l)},  
	\\
	0 & \text{ if } d  > \mu_j^{(l)} 
\end{cases}
\end{align} 
in  $\Sc_{n,r}^{\wt{\KK}}(\Bm)$. 
This implies that 
$\Psi (\left[ \begin{smallmatrix} \CK_{(j,l)} ;0 \\ d \end{smallmatrix} \right] ) \in \Sc_{n,r}^{\wt{\AA}}(\Bm)$. 

For $(i,k) \in \vG'(\Bm)$ and $t,d \in \ZZ_{\geq 0}$, 
we see that 
\begin{align*} 
&(\CX_{(i,k),t}^{+})^d (m_{\mu}) 
\\
&= q^{- d \mu_{i+1}^{(k)} + d (d+1)/2} m_{\mu+ d \a_{(i,k)}} 
	(L_{N^{\mu}_{(i,k)}+1} L_{N^{\mu}_{(i,k)}+2} \dots L_{N^{\mu}_{(i,k)} +d})^t  
	\left[ \begin{smallmatrix} T; N^{\mu}_{(i,k)}, \mu_{i+1}^{(k)} \\ d \end{smallmatrix} \right] 
\\
&= q^{- d \mu_{i+1}^{(k)} + d (d+1)/2} m_{\mu+ d \a_{(i,k)}} 
	(L_{N^{\mu}_{(i,k)}+1} L_{N^{\mu}_{(i,k)}+2} \dots L_{N^{\mu}_{(i,k)} +d})^t  
	\\ & \quad \times 
	(T;N_{(i,k)}^{\mu}, d)^{+}! \mathfrak{H}^+ (N_{(i,k)}^{\mu}, \mu_{i+1}^{(k)},d)
\end{align*} 
by Lemma \ref{Lemma Xpmt} together with Lemma \ref{Lemma Li commute [T; N mu]} 
and Corollary \ref{Cor [T;N,mu ,, d]}. 
We also see that 
$(T; N_{(i,k)}^{\mu},d)^+!$
commute with 
$(L_{N^{\mu}_{(i,k)}+1} L_{N^{\mu}_{(i,k)}+2} \dots L_{N^{\mu}_{(i,k)} +d})^t$ 
by Lemma \ref{Lemma commute L T} (\roiii),  
and  
see that 
$m_{\mu + d \a_{(i,k)}} (T; N_{(i,k)}^\mu,d)^+! = q^{d(d-1)/2}[d] !  m_{\mu + d \a_{(i,k)}}$ 
by \eqref{ m mu T}. 
Thus we have 
\begin{align*} 
&(\CX_{(i,k),t}^+)^d (m_{\mu}) 
\\
&= [d]! q^{- d \mu_{i+1}^{(k)} + d^2} m_{\mu +  d \a_{(i,k)}} 
	(L_{N^{\mu}_{(i,k)}+1} L_{N^{\mu}_{(i,k)}+2} \dots L_{N^{\mu}_{(i,k)} +d})^t  
	\mathfrak{H}^+  (N_{(i,k)}^{\mu}, \mu_{i+1}^{(k)},d) 
\end{align*}
in $\Sc_{n,r}^{\wt{\KK}}(\Bm)$.  
This implies that $\Psi (\CX_{(i,k)t}^{+(d)}) \in \Sc_{n,r}^{\wt{\AA}}(\Bm)$ 
since $\mathfrak{H}^+(N^{\mu}_{(i,k)}, \mu_{i+1}^{(k)},d) \in \He_{n,r}^{\AA}$ 
by the argument in the proof of Corollary \ref{Cor [T;N,mu ,, d]}. 
Similarly, we see that 
$\Psi(\CX_{(i,k),t}^{-(d)}) \in \Sc_{n,r}^{\wt{\AA}}(\Bm)$. 
Thus,  
the restriction of $\Psi$ to $\CU_{\wt{\AA},q,\wt{\BQ}}(\Bm)$ 
gives a homomorphism $\Psi_{\wt{\AA}}$. 

The last assertion follows from \cite[Proposition 6.4]{W}. 
\end{proof} 

%%%%%%%%%%%%%%%%%%%%%%%%%%%%%%%%%%%%%%%
\remark 
\label{Remark subjectivity of Psi}
In order to prove the surjectivity of $\Psi$ (resp. $\Psi_{\wt{\AA}}$), 
we use the result of \cite[Proposition 6.4]{W}.  
In fact, we considered only the case where $m_k=n$ for all $k=1,2,\dots,r$ in \cite{W}. 
However, we can apply the result to the case where 
$m_k \geq n$ for all $k=1,2, \dots, r-1$ without any change 
since the surjectivity in \cite[Proposition 6.4]{W} 
follows from the result in \cite{DR}. 
The reason why we assume the condition $m_k \geq n$ for all $k=1,2,\dots,r-1$ 
to state the surjectivity of $\Psi$ 
is just the using results of \cite{DR}. 
We expect that $\Psi$ is also surjective without this condition.

%%%%%%%%%%%%%%%%%%%%%%%%%%%%%%%%%%%%%%%

\begin{thm}
\label{Thm Sc Rep}
Assume that $m_k \geq n$ for all $k=1,2,\dots, r-1$. 
Then we have the followings. 
\begin{enumerate}
\item 
$\Sc_{n,r}^{\wt{\KK}}(\Bm) \cmod$ is a full subcategory of $\ZC_{q,\wt{\BQ}}^{\geq 0}(\Bm)$ 
through the surjection $\Psi$ in \eqref{surjection U to S}.

\item 
The Weyl module $\D(\la) \in \Sc_{n,r}^{\wt{\KK}}(\Bm)\cmod$ ($\la \in \vL_{n,r}^+ (\Bm)$) 
is the simple highest weight $\CU_{q,\wt{\BQ}}(\Bm)$-module of highest weight $(\la, \Bvf)$ 
through the surjection $\Psi$, 
where 
the multiset $\Bvf=(\vf_{(j,l),t}^{\pm} \in \wt{\KK} \,|\, (j,l) \in \vG(\Bm), t \geq 1)$ 
is given by 
\begin{align*}
\vf_{(j,l),t}^+ = Q_{l-1}^t  q^{(2 t -1) \la_j^{(l)}-t (2 j-1)} [\la_j^{(l)}] 
\text{ and }
\vf_{(j,l),t}^- = Q_{l-1}^t q^{\la_j^{(l)} - t (2 j -1)} [\la_j^{(l)}].
\end{align*}
\end{enumerate}
\end{thm}

\begin{proof}
For $\la \in \vL_{n,r}(\Bm)$, 
let $1_{\la}$ be an element of $\Sc_{n,r}^{\wt{\KK}} (\Bm)$ such that the identity on $M^{\la}$ 
and $1_{\la}(M^{\mu})=0$ for any $\mu \not=\la$. 
Then we have 
$1_{\la} 1_{\mu}= \d_{\la\mu} 1_{\la}$ and $\sum_{\la \in \vL_{n,r}(\Bm)} 1_{\la} =1$. 
Thus, for $M \in \Sc_{n,r}^{\wt{\KK}} \cmod$, we have the decomposition 
\begin{align}
\label{wt sp decom M} 
M = \bigoplus_{\mu \in \vL_{n,r}(\Bm)} 1_{\mu} M. 
\end{align}
Moreover, we see that 
\begin{align*} 
1_{\mu} M = \{ m \in M \,|\, \CK_{(j,l)}^+ \cdot m = q^{\mu_j^{(l)}} m \text{ for } (j,l) \in \vG(\Bm)\}
\end{align*} 
from the definition of $\Psi$. 
Thus, any object $M$ of $\Sc_{n,r}^{\wt{\KK}} \cmod$ has the weight space decomposition 
\eqref{wt sp decom M} 
as a $\CU_{q,\wt{\BQ}}(\Bm)$-module, 
where we remark that $\vL_{n,r}(\Bm) \subset P_{\geq 0}$. 

For $M \in \Sc_{n,r}^{\wt{\KK}}(\Bm) \cmod$, 
in order to see that all eigenvalues of the action of $\CI_{(j,l),t}^{\pm}$ 
($(j,l) \in \vG(\Bm), t \geq 0$) on $M$ belong to $\wt{\KK}$, 
it is enough to show them for $\D(\la)$ ($\la \in \vL_{n,r}^+(\Bm)$) 
since $\Sc_{n,r}^{\wt{\KK}}(\Bm)$ is semi-simple and 
$\{\D(\la)\,|\, \la \in \vL_{n,r}^+(\Bm)\}$ gives a complete set of isomorphism classes of 
simple $\Sc_{n,r}^{\wt{\KK}}$-modules. 
Recall that $\{\vf_T \,|\, T \in \CT_0 (\la,\mu) \text{ for some } \mu \in \vL_{n,r}(\Bm)\}$ 
gives a basis of $\D(\la)$. 

Note that 
$\Phi^{\pm}_t (L_{N_{(j,l)}^\mu}, L_{N_{(j,l)}^\mu -1}, \dots, L_{N_{(j,l)}^\mu - \mu_j^{(l)}+1})$ 
commute with $T_w$ for any $w \in \FS_{\mu}$ by Lemma \ref{Lemma commute L T},  
for $T \in \CT_0(\la,\mu)$, we have 
\begin{align} 
\label{CI on vfT}
\CI_{(j,l),t}^{\pm} \cdot \vf_T = 
\begin{cases} \dis 
q^{\pm (t-1)} \Phi_t^{\pm} (\res_{(j,l);T}) \vf_T + \sum_{S \vartriangleright T} r_S \vf_S \quad (r_S \in \wt{\KK}) 
	& \text{ if } \mu_j^{(l)} \not=0, 
	\\
	0 & \text{ if } \mu_j^{(l)}=0 
\end{cases}
\end{align} 
in a similar argument as in the proof of \cite[Theorem 3.10]{JM}, 
where 
\begin{align*}
\Phi_t^{\pm}(\res_{(j,l);T}) = \Phi_t^{\pm}( \res(x_1), \res(x_2), \dots, \res(x_{\mu_j^{(l)}}))
\end{align*} 
with $\{x_1,x_2, \dots, x_{\mu_j^{(l)}}\} = \{x \in [\la] \,|\, T(x) =(j,l)\}$, 
and $\vartriangleright$ is a partial order on $\CT_{0}(\la,\mu)$ defined in \cite[Definition 3.6]{JM}. 
This implies that all eigenvalues of the action of $\CI_{(j,l),t}^{\pm}$ on $\D(\la)$ belong to $\wt{\KK}$. 
Now we proved (\roi). 

We prove (\roii). 
For $\la \in \vL_{n,r}^+(\Bm)$, 
let $T^{\la}$ be the unique semi-standard tableau of shape $\la$ with weight $\la$. 
Then, we see easily that $\vf_{T^{\la}}$ is a highest weight vector of $\D(\la)$. 
Note that there is no tableau such that $S \vartriangleright T^{\la}$, 
then we have 
\begin{align}
\label{vfpm}
\vf_{(j,l),t}^{\pm} = q^{\pm (t-1)} \Phi_t^{\pm} (Q_k q^{2(1-j)}, Q_k q^{2 (2-j)}, \dots, Q_k q^{2 (\la_j^{(l)}-j)}) 
\end{align}
by \eqref{CI on vfT}. 
Then we can prove (\roii) by the induction on $t$ 
using \eqref{vfpm} and \eqref{recursive relation of Phi}. 
\end{proof}

% #################

Let $\Sc_{n,r}^\mathbf{1}(\Bm)$ be the cyclotomic $q$-Schur algebra over $\QQ(\wt{\BQ})$ 
with parameters $q=1$, $Q_0,Q_1,\dots,Q_{r-1}$. 
Then we have the following theorem. 

\begin{thm}\
\label{Theorem Sc1 in ZC}
\begin{enumerate}
\item 
We have a homomorphism of algebras 
\begin{align}
\label{hom U(gm) to Sc1}
\Psi_{\mathbf{1}} : U(\Fg_{\wt{\BQ}}(\Bm)) \ra \Sc_{n,r}^{\mathbf{1}}(\Bm) 
\end{align}
by taking $\Psi_{\mathbf{1}}(\CX_{(i,k),t}^{\pm})= \CX_{(i,k),t}^{\pm}$ 
and $\Psi_{\mathbf{1}}(\CI_{(j,l),t})=\CI_{(j,l),t}^+ (=\CI_{(j,l),t}^-)$. 

Moreover, if $m_k \geq n$ for all $k=1,2, \dots, r-1$, 
the homomorphism $\Psi_{\mathbf{1}}$ is surjective. 

\item 
Assume that $m_k \geq n$ for all $k=1,2, \dots, r-1$. 
Then 
$\Sc_{n,r}^{\mathbf{1}}(\Bm) \cmod$ is a full subcategory of $\ZC_{\wt{\BQ}}^{\geq 0} (\Bm)$ 
through the surjection $\Psi_{\mathbf{1}}$. 

Moreover, 
the Weyl module $\D(\la) \in \Sc_{n,r}^{\mathbf{1}}(\Bm)\cmod$ ($\la \in \vL_{n,r}^+$) 
is the simple highest weight $U(\Fg_{\wt{\BQ}}(\Bm))$-module of highest weight $(\la, \Bvf)$ 
through the surjection $\Psi_{\mathbf{1}}$, 
where 
the multiset $\Bvf=(\vf_{(j,l),t} \in \QQ(\wt{\BQ}) \,|\, (j,l) \in \vG(\Bm), t \geq 1)$ 
is given by 
\begin{align*}
\vf_{(j,l),t} = Q_{l-1}^t   \la_j^{(l)}. 
\end{align*}

\end{enumerate}

\end{thm} 
\begin{proof}
Note Lemma \ref{Lemma wtKJ0} and Lemma  \ref{Lemma K I at q=1}, 
then we can prove the theorem in a similar way 
as in the proof of Theorem \ref{Thm UqQ to Sc} and Theorem \ref{Thm Sc Rep}. 
\end{proof}

%%%%%%%%%%%%%%%%%%%%%%%%%%%%%%%%%%%%%%%%%%%%%%%%%%%%%%%%%%%

\section{Characters of Weyl modules of cyclotomic $q$-Schur algebras} 
\label{Ch of Weyl}
In this section, we study the characters of Weyl modules of cyclotomic $q$-Schur algebras 
as symmetric polynomials. 
In particular, we prove the conjecture given in \cite{W-2} (the formula \eqref{Conjecture ch W} below)  
which will be understood as the decomposition of the tensor product of Weyl modules 
in the case where $q=1$.

\para \textbf{Characters.} 
For $k=1,\dots,r$, 
let $\bx_{\Bm}^{(k)} =(x_{(1,k)}, x_{(2,k)}, \dots, x_{(m_k,k)})$ 
be the set of $m_k$ independent variables, 
and put 
$\bx_{\Bm} = \cup_{k=1}^r \bx_{\Bm}^{(k)}$. 
Let $\ZZ[\bx_{\Bm}^{\pm}]$ (resp. $\ZZ[ \bx_{\Bm}]$) be the ring of Laurent polynomials 
(resp. the ring of polynomials) 
with variables $\bx_{\Bm}$. 
For $\la \in P$, we define the monomial $x^\la \in \ZZ[\bx_{\Bm}^{\pm}]$ 
by $x^\la =\prod_{k=1}^r \prod_{i=1}^{m_k} x_{(i,k)}^{\lan \la, h_{(i,k)} \ran} $. 

For $M \in \ZC_{\wt{\BQ}}(\Bm)$ (resp. $M \in \ZC_{q, \wt{\BQ}}(\Bm)$), we define the character of $M$ by 
\begin{align}
\label{Def character}
\ch M = \sum_{\la \in P} \dim M_{\la} x^\la \in \ZZ[\bx_{\Bm}^{\pm}].
\end{align} 
It is clear that 
$\ch M \in \ZZ[\bx_{\Bm}]$ if $M \in \ZC_{\wt{\BQ}}^{\geq 0} (\Bm)$ 
(resp. $M \in \ZC_{q,\wt{\BQ}}^{\geq 0}(\Bm)$). 

When we regard $M \in \ZC_{\wt{\BQ}}(\Bm)$ 
as a $U (\Fgl_{m_1} \oplus \dots \oplus \Fgl_{m_r})$-module 
through the injection \eqref{injection g Levi to gQm}, 
$\ch M$ defined by \eqref{Def character} 
coincides with the character of $M$ as a $U (\Fgl_{m_1} \oplus \dots \oplus \Fgl_{m_r})$-module 
since $M_{\la}$ is also the weight space of weight $\la$ 
as a $U (\Fgl_{m_1} \oplus \dots \oplus \Fgl_{m_r})$-module . 
Thus, by the known results for $U (\Fgl_{m_1} \oplus \dots \oplus \Fgl_{m_r})$-modules, 
we see that 
\begin{align*} 
\ch M \in \bigotimes_{k=1}^r \ZZ[\bx_{\Bm}^{(k)}]^{\FS_{m_k}}  
\text{ if } 
M \in \ZC_{\wt{\BQ}}^{\geq 0} (\Bm),
\end{align*} 
where $\ZZ[\bx_{\Bm}^{(k)}]^{\FS_{m_k}}$ is the ring of symmetric polynomials with variables $\bx_{\Bm}^{(k)}$, 
and we regard $\bigotimes_{k=1}^r \ZZ[\bx_{\Bm}^{(k)}]^{\FS_{m_k}}$ as a subring of $\ZZ [\bx_{\Bm}]$ 
through the multiplication map 
$\bigotimes_{k=1}^r \ZZ[\bx_{\Bm}^{(k)}]^{\FS_{m_k}} \ra \ZZ [\bx_{\Bm}]$ 
($\otimes_{k=1}^r f(\bx_{\Bm}^{(k)}) \mapsto \prod_{k=1}^r f(\bx_{\Bm}^{(k)})$). 
It is similar for $M \in \ZC_{q, \wt{\BQ}}(\Bm)$ through the injection \eqref{injection Levi to U}. 

%%%%%%

\para 
The character of the Weyl module $\D(\la) \in \Sc_{n,r}(\Bm)$ ($\la \in \wt{\vL}_{n,r}^+(\Bm)$) 
is studied in \cite{W-2}. 
Note that $\ch \D(\la)$ ($\la \in \wt{\vL}_{n,r}^+(\Bm)$) 
does not depend on the choice of the base field and parameters.  
Put  $\wt{\vL}^+_{\geq 0,r}(\Bm)= \cup_{n \geq 0} \wt{\vL}_{n,r}^+ (\Bm)$. 
For $\la,\mu \in \wt{\vL}^+_{\geq 0,r}(\Bm)$, 
the following formula was conjectured in \cite[Conjecture 2]{W-2}:
\begin{align}
\label{Conjecture ch W}
\ch \D(\la) \ch \D(\mu)
= \sum_{\nu \in \wt{\vL}^+_{\geq 0,r}(\Bm)} 
	\LR_{\la\mu}^\nu \ch \D(\nu) 
	\, 
	\text{ for } \la, \mu \in  \wt{\vL}^+_{\geq 0,r}(\Bm), 
\end{align}
where 
$\LR_{\la\mu}^\nu = \prod_{k=1}^r \LR_{\la^{(k)}\mu^{(k)}}^{\nu^{(k)}}$, 
and $\LR_{\la^{(k)}\mu^{(k)}}^{\nu^{(k)}}$ is the Littlewood-Richardson coefficient for the partitions 
$\la^{(k)}$, $\mu^{(k)}$ and $\nu^{(k)}$.   
We prove this conjecture as follows. 

%%%%%

\para 
For $\la =(\la^{(1)}, \dots, \la^{(r)}) \in \wt\vL_{n,r}^+(\Bm)$, 
we denote 
\begin{align*}
(\underbrace{0, \dots, 0}_{k-1}, \la^{(k)}, 0, \dots, 0) \in \wt\vL_{n_k,r}^+(\Bm)
\end{align*} 
by $(0, \dots, \la^{(k)}, \dots,0)$ simply, 
where $n_k= \sum_{i=1}^{m_k} \la_i^{(k)}$ 
(i.e. $\la^{(k)}$ appears in the $k$-th component in $(0, \dots, \la^{(k)}, \dots,0)$). 
Let 
\[
S_{\la^{(k)}}(\bx_{\Bm}^{(k)} \cup \dots \cup \bx_{\Bm}^{(r)}) 
\in \ZZ[ \bx_\Bm^{(k)} \cup \dots \cup \bx_\Bm^{(r)}]^{\FS(\bx_\Bm^{(k)} \cup \dots \cup \bx_\Bm^{(r)})} 
\]
be the Schur polynomial for the partition $\la^{(k)}$ 
with variables $\bx_\Bm^{(k)} \cup \dots \cup \bx_\Bm^{(r)}$, 
where we regard 
$ \ZZ[ \bx_\Bm^{(k)} \cup \dots \cup \bx_\Bm^{(r)}]^{\FS(\bx_\Bm^{(k)} \cup \dots \cup \bx_\Bm^{(r)})} $ 
as a subring of $\bigotimes_{k=1}^r \ZZ[\bx_{\Bm}^{(k)}]^{\FS_{m_k}} \subset \ZZ[\bx_{\Bm}]$ 
in the natural way.
Put $\wt{S}_{\la}(\bx_{\Bm})= \ch \D(\la)$ ($\la \in \wt{\vL}_{\geq 0,r}^+(\Bm)$). 
Then we have the following proposition.  
%%%%%%

\begin{prop}
\label{Prop ch}
For $\la, \mu \in \wt\vL_{\geq 0,r}^+(\Bm)$, 
we have the following formulas. 
\begin{enumerate}
\item 
$\dis \wt{S}_{(0, \dots, \la^{(k)},\dots,0)} (\bx_\Bm) = 
S_{\la^{(k)}}(\bx_{\Bm}^{(k)} \cup \dots \cup \bx_{\Bm}^{(r)})$. 

\item 
$\dis \wt{S}_{\la}(\bx_{\Bm}) = \prod_{k=1}^r \wt{S}_{(0,\dots,\la^{(k)},\dots,0)}(\bx_{\Bm})$.

\item 
$\dis \wt{S}_\la (\bx_{\Bm}) \wt{S}_{\mu} (\bx_{\Bm}) 
	= \sum_{\nu \in \wt\vL_{\geq 0,r}^+(\Bm)} \LR_{\la\mu}^\nu
			\wt{S}_\nu(\bx_{\Bm})$.
\end{enumerate}
\end{prop}

\begin{proof}
(\roi). 
By the definition of the cellular basis of $\Sc_{n,r}(\Bm)$ in \cite{DJM98}, 
for $\la \in \wt{\vL}_{n,r}^+(\Bm)$, 
we have 
\begin{align}
\label{ch Dla semi-std}
\wt{S}_{\la}(\bx_{\Bm})=\ch \D(\la) = \sum_{\mu \in \vL_{n,r}(\Bm)} \sharp \CT_0(\la,\mu) x^\mu. 
\end{align}
Thus, we have 
\begin{align}
\label{Schur 1-1}
\wt{S}_{(0, \dots, \la^{(k)},\dots,0)} (\bx_\Bm) 
	= \sum_{\mu \in \vL_{n_k,r}(\Bm)} \sharp \CT_0((0, \dots, \la^{(k)},\dots,0) , \mu) x^\mu, 
\end{align}
where $n_k=\sum_{i=1}^{m_k} \la_i^{(k)}$. 
We see that 
\begin{align*}
\mu^{(1)}= \dots = \mu^{(k-1)}=0 
\text{ if }
\CT_0((0, \dots, \la^{(k)},\dots,0) , \mu) \not=\emptyset
\end{align*} 
by the definition of semi-standard tableaux. 
Thus, we have  
$\wt{S}_{(0, \dots, \la^{(k)},\dots,0)} (\bx_\Bm) \in \bigotimes_{l=k}^r \ZZ[\bx_{\Bm}^{(l)}]^{\FS_{m_k}}$.
Put 
\[ 
\vL^{\geq k}_{n_k,r}(\Bm) 
=\{ \mu=(\mu^{(1)},\dots,\mu^{(r)})  \in \vL_{n_k,r}(\Bm) \,|\, \mu^{(l)}=0 \text{ for } l=1,\dots,k-1\}.
\]
Put $m' = m_k+\dots + m_r$. 
We identify the set $\vL_{n_k,1}(m')$ with $\vL_{n_k,r}^{\geq k}(\Bm)$ by 
the bijection 
$\theta^k : \vL_{n_k,1}(m') \mapsto \vL_{n_k,r}^{\geq k}(\Bm)$ 
such that 
\[
(\theta^k(\mu))^{(k+l)}_i = 
\begin{cases} 
\mu_{i} & \text{ if } l=0,
\\
\mu_{m_{k} + m_{k+1} + \dots +m_{k+l-1} +i} & \text{ if } 1 \leq l \leq r-k  
\end{cases}
\] 
for $\mu=(\mu_1,\mu_2,\dots,\mu_{m'}) \in  \vL_{n_k,1}(m')$.
By the well-known fact, we can describe  
the Schur polynomial 
$S_{\la^{(k)}}(\bx_{\Bm}^{(k)} \cup \dots \cup \bx_{\Bm}^{(r)})$ 
as 
\begin{align}
\label{Schur 1-2}
S_{\la^{(k)}}(\bx_{\Bm}^{(k)} \cup \dots \cup \bx_{\Bm}^{(r)}) 
= \sum_{\mu \in \vL_{n_k,1}(m')} \sharp \CT_0(\la^{(k)}, \mu) x^\mu,
\end{align}
where 
we put 
$x^\mu = \prod_{i=1}^{m_k} x_{(i,k)}^{\mu_i}  
	\prod_{l=1}^{r-k} \prod_{i=1}^{m_l} x_{(i,k+l)}^{\mu_{m_{k} + m_{k+1} + \dots +m_{k+l-1} +i}} $. 
From the definition of semi-standard tableaux, 
we see that 
\[
\sharp \CT_0(\la^{(k)},\mu) = \sharp \CT_0((0, \dots, \la^{(k)},\dots,0) , \theta^k(\mu))
\] 
for $\mu \in \vL_{n_k,1}(m')$.  
Thus, by comparing the right hand sides of \eqref{Schur 1-1} and of \eqref{Schur 1-2}, 
we obtain (\roi). 

(\roii). 
First we prove that 
\begin{align}
\label{Schur decom}
\wt{S}_{(\la^{(1)},\la^{(2)},\dots,\la^{(r)})} (\bx_\Bm) 
= \wt{S}_{(\la^{(1)}, 0, \dots, 0)} (\bx_\Bm) \wt{S}_{(0, \la^{(2)},\dots, \la^{(r)})} (\bx_\Bm). 
\end{align}
By \eqref{ch Dla semi-std}, we have 
\begin{align}
\label{Schur decom 1}
 \wt{S}_{(\la^{(1)},\la^{(2)},\dots,\la^{(r)})} (\bx_\Bm)  
	= \sum_{\mu \in \vL_{n,r}(\Bm)} \sharp \CT_0(\la,\mu) \, x^\mu.
\end{align}
On the other hand, 
we have 
\begin{align}
\label{Schur decom 2}
\begin{split}
&\wt{S}_{(\la^{(1)}, 0, \dots, 0)} (\bx_\Bm) \wt{S}_{(0, \la^{(2)},\dots, \la^{(r)})} (\bx_\Bm) 
\\
&= 
(\sum_{\nu \in \vL_{n_1,r}(\Bm)} \sharp \CT_0((\la^{(1)},0,\dots,0),\nu) \, x^\nu)
(\sum_{\t \in \vL_{n',r}(\Bm)} \sharp \CT_0((0,\la^{(2)},\dots,\la^{(r)}),\t) \, x^\t) 
\\
&= 
\sum_{\mu \in \vL_{n,r}(\Bm)} 
\Big( \sum_{\nu \in \vL_{n_1,r}(\Bm), \t \in \vL_{n',r}(\Bm) \atop \nu + \t =\mu} 
	\hspace{-2em} 
	 \sharp \CT_0((\la^{(1)},0,\dots,0),\nu) \sharp \CT_0((0,\la^{(2)},\dots,\la^{(r)}),\t)
	\Big) x^\mu
\end{split}
\end{align}
where 
$n_1=\sum_{i=1}^{m_1}\la_i^{(1)}$ and $n'=n-n_1$. 
From the definition of semi-standard tableaux, 
we can check that 
\begin{align}
\label{semi std comparsion}
 \sharp \CT_0(\la,\mu)  = \sum_{\nu \in \vL_{n_1,r}(\Bm), \t \in \vL_{n',r}(\Bm) \atop \nu + \t =\mu} 
	\hspace{-2em} 
	 \sharp \CT_0((\la^{(1)},0,\dots,0),\nu) \sharp \CT_0((0,\la^{(2)},\dots,\la^{(r)}),\t).
\end{align}
Thus, \eqref{Schur decom 1}, \eqref{Schur decom 2} and \eqref{semi std comparsion} 
imply \eqref{Schur decom}. 
By applying a similar argument to 
$ \wt{S}_{(0, \la^{(2)},\dots, \la^{(r)})} (\bx_\Bm) $ inductively, 
we obtain (\roii). 

By (\roi) and (\roii), we have 
\begin{align*}
\wt{S}_\la (\bx_{\Bm}) \wt{S}_{\mu} (\bx_{\Bm})  
&= \big(\prod_{k=1}^r \wt{S}_{(0,\dots,\la^{(k)},\dots,0)}(\bx_{\Bm}) \big)
	 \big( \prod_{k=1}^r \wt{S}_{(0,\dots,\mu^{(k)},\dots,0)}(\bx_{\Bm}) \big)
\\
&= \big(\prod_{k=1}^r  S_{\la^{(k)}}(\bx_{\Bm}^{(k)} \cup \dots \cup \bx_{\Bm}^{(r)}) \big) 
	\big( \prod_{k=1}^r  S_{\mu^{(k)}}(\bx_{\Bm}^{(k)} \cup \dots \cup \bx_{\Bm}^{(r)}) \big) 
\\
&= \prod_{k=1}^r  S_{\la^{(k)}}(\bx_{\Bm}^{(k)} \cup \dots \cup \bx_{\Bm}^{(r)}) \big)
	S_{\mu^{(k)}}(\bx_{\Bm}^{(k)} \cup \dots \cup \bx_{\Bm}^{(r)}) \big) 
\\
&= \prod_{k=1}^r \Big( \sum_{\nu^{(k)} \in \vL^+_{\geq 0,1}(m_k+\dots+m_r)} 
	\LR_{\la^{(k)} \mu^{(k)}}^{\nu^{(k)}}  S_{\nu^{(k)}}(\bx_{\Bm}^{(k)} \cup \dots \cup \bx_{\Bm}^{(r)}) \Big)
\\
&=\sum_{\nu \in \wt\vL_{\geq 0,r}^+(\Bm)} \Big( \prod_{k=1}^r \LR_{\la^{(k)} \mu^{(k)}}^{\nu^{(k)}} \Big) 
			\prod_{k=1}^r \wt{S}_{(0,\dots,\nu^{(k)},\dots,0)}(\bx_{\Bm}) 
\\
&=\sum_{\nu \in \wt\vL_{\geq 0,r}^+(\Bm)} \LR_{\la\mu}^\nu
			\wt{S}_\nu(\bx_{\Bm}), 
\end{align*}
where we note that, if $\ell (\la^{(k)}) > m_k + \dots + m_r$ for some $k$, 
we have  
$S_{\la^{(k)}}(\bx_{\Bm}^{(k)} \cup \dots \cup \bx_{\Bm}^{(r)})=0$ 
and $\CT_0(\la,\mu)= \emptyset$  
for any $\mu \in \vL_{n,r}(\Bm)$. 
Now we obtained (\roiii). 
\end{proof}

%%%%%%%%%%%%%%%%%%%%%%%%%%%%%%%%%%%%%%%%%%%%%%%%%%%%%%%%%%%%%%%%%

%%%%%%%%%%%%%%%%%%%%%%%%%%%%%%%%%%%%%%%%%%%%%%%%%%%%%%%%%%%%%%%%%

\section{Tensor products for Weyl modules of cyclotomic $q$-Schur algebras at $q=1$}

By using the comultiplication $\D : U(\Fg_{\wt{Q}}(\Bm)) \ra U(\Fg_{\wt{Q}}(\Bm)) \otimes U(\Fg_{\wt{Q}}(\Bm))$ 
($\D(x)=x \otimes 1 + 1 \otimes x$), 
we define the $U(\Fg_{\wt{Q}}(\Bm))$-module $M \otimes N$ 
for $U(\Fg_{\wt{Q}}(\Bm))$-module $M$ and $N$. 
We regard $\Sc_{n,r}^{\mathbf{1}}(\Bm)$-modules ($n \geq 0$) 
as a $U(\Fg_{\wt{\BQ}}(\Bm))$-modules 
through the homomorphism $\Psi_{\mathbf{1}}$ in \eqref{hom U(gm) to Sc1}. 
Note that $\Sc_{n,r}^{\mathbf{1}}(\Bm)$ is semi-simple, 
and $\{\D(\la)\,|\, \la \in \vL_{n,r}^+(\Bm)\}$ 
gives a complete set of isomorphism classes of simple 
$\Sc_{n,r}^{\bf{1}}(\Bm)$-modules 
if $m_k \geq n$ for all $k=1,2, \dots, r-1$. 
Then, we have the following proposition. 

\begin{prop}
\label{Prop decom tensor Weyl}
Assume that $m_k \geq n$ for all $k=1,2,\dots,r-1$. 
Take $n_1, n_2 \in \ZZ_{>0}$ such that $n=n_1+n_2$. 
For $\la \in \vL_{n_1,r}^+(\Bm)$ (resp. $\mu \in \vL_{n_2,r}^+(\Bm)$), 
let $\D(\la)$ (resp. $\D(\mu)$) be the Weyl module of 
$\Sc_{n_1,r}^{\mathbf{1}}(\Bm)$ (resp. $\Sc_{n_2,r}^{\mathbf{1}}(\Bm)$) 
corresponding $\la$ (resp. $\mu$). 
Then we have 
\begin{align}
\label{decom tensor Weyl}
\D(\la) \otimes \D(\mu) \cong \bigoplus_{\nu \in \vL_{n,r}^+(\Bm)} \LR_{\la \mu}^\nu \D(\nu) 
\text{ as $U(\Fg_{\wt{\BQ}}(\Bm))$-modules}, 
\end{align}
where $\D(\nu)$ is the Weyl module of $\Sc_{n,r}^{\mathbf{1}}(\Bm)$ corresponding $\nu$, 
and $\LR_{\la \mu}^\nu \D(\nu) $ means the direct sum of $\LR_{\la\mu}^{\nu}$ copies of $\D(\nu)$. 
In particular, 
$\D(\la) \otimes \D(\mu) \in \Sc_{n,r}^{\mathbf{1}}(\Bm) \cmod$. 
\end{prop}
\begin{proof}
For $\t \in P_{\geq 0}$, 
put 
\begin{align*}
\pi_{\Bm}(\t)= (|\t^{(1)}|, |\t^{(2)}|, \dots, |\t^{(r)}|) \in \ZZ_{\geq 0}^r, 
\end{align*}
where $|\t^{(l)}| = \sum_{j=1}^{m_l} \lan \t, h_{(j,l)} \ran $ for $l=1,\dots,r$. 
We denote by $\geq$ the lexicographic order on $\ZZ_{\geq 0}^r $. 
Then we have the weight space decomposition 
\begin{align}
\label{wt decom Dla tensor Dmu}
\D(\la) \otimes \D(\mu) 
= \bigoplus_{\t \in \vL_{n,r}(\Bm) \atop \pi_{\Bm}(\t) \leq \pi_{\Bm}(\la+\mu)} 
	(\D(\la) \otimes \D(\mu))_{\t}. 
\end{align}
On the other hand, it is clear that 
$\D(\la) \otimes \D(\mu) \in \ZC_{\wt{\BQ}}^{\geq 0} (\Bm)$. 
Thus, we have 
\begin{align}
\label{decom Dla tensor Dmu in K0}
[\D(\la) \otimes \D(\mu)] 
= \sum_{\nu \in \vL_{n,r}^+(\Bm) \atop \pi_{\Bm}(\nu) \leq \pi_{\Bm}(\la+\mu)} 
	\sum_{\Bvf} d_{\nu, \Bvf} [ L(\nu, \Bvf)] 
	\,\, 
	\text{ in } K_0( \ZC_{\wt{\BQ}}^{\geq 0}(\Bm)), 
\end{align}
where $d_{\nu, \Bvf}$ is the composition multiplicity of the simple highest weight $U(\Fg_{\wt{\BQ}}(\Bm))$-module 
$L(\nu, \Bvf)$ of highest weight $(\nu, \Bvf)$ in $\D(\la) \otimes \D(\mu)$. 

Note that $L_{i+1} T_i = T_i L_i$ and $L_i T_i = T_i L_{i+1}$ since $q=1$. 
Then,  
for $(j,l) \in \vG(\Bm)$ and $t \geq 1$, we see that 
\begin{align} 
\label{action CI on nu space}
\CI_{(j,l),t} \cdot v = Q_{l-1}^t \nu_{j}^{(l)} v 
\text{ for any } v \in (\D(\la) \otimes \D(\mu))_{\nu} 
\end{align}
if $\pi_{\Bm}(\nu) = \pi_{\Bm}(\la+\mu)$ 
by the argument in the proof of \cite[Proposition 3.7 and Theorem 3.10]{JM}. 
This implies that 
\begin{align}
\label{L(nu bvf) cong D(nu)}
L(\nu, \Bvf) \cong \D(\nu) 
\text{ if } d_{\nu,\Bvf} \not=0 \text{ and } \pi_{\Bm}(\nu) = \pi_{\Bm}(\la+\mu) 
\end{align}
by Theorem \ref{Theorem Sc1 in ZC} (\roii). 
By Proposition \ref{Prop ch} (\roiii) together with 
\eqref{decom Dla tensor Dmu in K0} and \eqref{L(nu bvf) cong D(nu)}, we have 
\begin{align}
\label{ch D(la) tensor D(mu)}
\ch (\D(\la) \otimes \D(\mu)) 
&= \wt{S}_{\la}(\bx_{\Bm}) \wt{S}_{\mu}(\bx_{\Bm})
\\ \notag 
&= \sum_{\nu \in \vL_{n,r}^+(\Bm)} \LR_{\la\mu}^{\nu} \wt{S}_\nu (\bx_{\Bm}) 
\\ \notag 
&= \sum_{\nu \in \vL_{n,r}^+(\Bm) \atop \pi_{\Bm}(\nu) = \pi_{\Bm}(\la+\mu)} d_{\nu} \wt{S}_{\nu}(\bx_{\Bm}) 
	+ \sum_{\nu \in \vL_{n,r}^+(\Bm) \atop \pi_{\Bm}(\nu) < \pi_{\Bm}(\la+\mu)} 
		\sum_{\Bvf} d_{\nu, \Bvf} \ch L(\nu, \Bvf), 
\end{align}
where $d_{\nu}$ is the composition multiplicity of $\D(\nu)$ in $\D(\la) \otimes \D(\mu)$.  
Note that $\LR_{\la \mu}^\nu =0$ unless $\pi_{\Bm}(\nu)= \pi_{\Bm}(\la+\mu)$, 
the equations \eqref{ch D(la) tensor D(mu)} imply 
$d_\nu = \LR_{\la \mu}^{\nu}$ if $\pi_{\Bm}(\nu)= \pi_{\Bm}(\la+\mu)$ 
and $d_{\nu, \Bvf}=0$ if $\pi_{\Bm}(\nu) < \pi_{\Bm}(\la+\mu)$. 
Thus, we have 
\begin{align}
\label{D(la) tensor D(mu) in K0}
[\D(\la) \otimes \D(\mu)] 
= \sum_{\nu \in \vL_{n,r}^+(\Bm)} \LR_{\la\mu}^\nu [\D(\nu)].
\end{align}

By \eqref{wt decom Dla tensor Dmu}, 
for any $k=1,2,\dots,r-1$ and any $t \geq 0$,  
we have 
\begin{align}
\label{CX+ mk vanish}
\CX_{(m_k,k),t}^+ \cdot \Big( \bigoplus_{\nu \in \vL_{n,r}(\Bm) \atop \pi_{\Bm}(\nu) = \pi_{\Bm}(\la+\mu)} 
	(\D(\la) \otimes \D(\mu))_{\nu} \Big) =0 
\end{align}
since $\pi_{\Bm}(\nu + \a_{(m_k,k)}) > \pi_{\Bm}(\nu)$. 
Then, 
by \eqref{action CI on nu space} and \eqref{CX+ mk vanish} together with the relation (L2), 
we see that  
\begin{align}
\label{number of maximal vector}
\begin{split}
&\{ v \in (\D(\la) \otimes \D(\mu))_{\nu} 
	\,|\, \CX_{(i,k),t}^+ \cdot v \text{ for all } (i,k) \in \vG'(\Bm) \text{ and } t \geq 0\} 
\\
&=\{ v \in (\D(\la) \otimes \D(\mu))_{\nu} \,|\, e_{(i,k)} \cdot v 
	\text{ for all } (i,k) \in \vG(\Bm) \setminus \{(m_k,k)\,|\, 1 \leq k \leq r\}\} 
\end{split}
\end{align}
for $\nu \in \vL_{n,r}^+(\Bm)$ such that $\pi_{\Bm}(\nu) =\pi_{\Bm}(\la+\mu)$,  
where $e_{(i,k)} \in U(\Fgl_{m_1} \oplus \dots \oplus \Fgl_{m_r})$ acts on $\D(\la) \otimes \D(\mu)$ 
through the injection  \eqref{injection g Levi to gQm}. 
On the other hand, 
$ \bigoplus_{\nu \in \vL_{n,r}(\Bm) \atop \pi_{\Bm}(\nu) = \pi_{\Bm}(\la+\mu)} 
	(\D(\la) \otimes \D(\mu))_{\nu} $ 
is a $U(\Fgl_{m_1} \oplus \dots \oplus \Fgl_{m_r})$-submodule of $\D(\la) \otimes \D(\mu)$ 
and we have 
\begin{align}
\label{decom D(la) tensor D(mu) nu as Levi}
\bigoplus_{\nu \in \vL_{n,r}(\Bm) \atop \pi_{\Bm}(\nu) = \pi_{\Bm}(\la+\mu)} 
	(\D(\la) \otimes \D(\mu))_{\nu} 
\cong \bigoplus_{\nu \in \vL_{n,r}^+(\Bm)} 
	\LR_{\la\mu}^{\nu} \D_{\Fgl_{m_1}}(\nu^{(1)}) \otimes \dots \otimes \D_{\Fgl_{m_r}}(\nu^{(r)})
\end{align}
as $U(\Fgl_{m_1} \oplus \dots \oplus \Fgl_{m_r})$-modules 
by comparing the character (note \cite[Lemma 2.6]{W-2}). 
By \eqref{D(la) tensor D(mu) in K0}, \eqref{number of maximal vector} 
and \eqref{decom D(la) tensor D(mu) nu as Levi}, 
we see that 
\begin{align*}
\D(\la) \otimes \D(\mu) \cong \bigoplus_{\nu \in \vL_{n,r}^+(\Bm)} \LR_{\la\mu}^\nu \D(\nu) 
\end{align*}
as $U(\Fg_{\wt{\BQ}}(\Bm))$-modules. 
\end{proof}

%%%%%

\remarks 
\begin{enumerate} 
\item 
For $M,N \in \ZC_{\wt{Q}}(\Bm)$, 
we see that 
$\ch (M \otimes N) = \ch(M) \ch (N)$ 
by definition of characters.  
Then 
the decomposition \eqref{decom tensor Weyl} 
gives an interpretation of the formula \eqref{Conjecture ch W} (Proposition \ref{Prop ch} (\roiii)) 
in the category $\ZC_{\wt{\BQ}}(\Bm)$. 

\item 
We conjecture that the algebra $\CU_{q,\wt{\BQ}}(\Bm)$ has a structure as a Hopf algebra. 
Then we also conjecture the similar decomposition for the tensor product of Weyl modules 
of $\Sc_{n,r}^{\wt{\KK}}(\Bm)$ ($n \geq 0$) as in \eqref{decom tensor Weyl}. 

\end{enumerate}

%%%%%%%%%%%%%%%%%%%%%%%%%%%%%%%%%%%%%%%%%%%%%%%%%%%%%%%%%%%%%%%%%

%%%%%%%%%%%%%%%%%%%%%%%%%%%%%%%%%%%%%%%%%%%%%%%%%%%%%%%%%%%%%%%%%

\end{document}